\documentclass{article}

\usepackage{geometry}

\newgeometry{vmargin={30mm}, hmargin={30mm,30mm}}   

\usepackage[utf8]{inputenc}
\usepackage[T1]{fontenc}
\usepackage{lmodern}
\usepackage{amsmath} 
\usepackage{amssymb} 
\usepackage{amsfonts}
\usepackage{amsthm}  
\usepackage{mathtools}
\include{macros}
\theoremstyle{definition}
\newtheorem{definition}{Definition}
\theoremstyle{theorem}
\newtheorem{theorem}{Theorem}
\newtheorem{corollary}{Corollary}
\newtheorem{lemmach}{Lemma}
\theoremstyle{theorem}
\newtheorem{claimm}{Claim}
\theoremstyle{remark}

\theoremstyle{remark}
\newtheorem{remark}{Remark}
\theoremstyle{example}

\theoremstyle{notation}
\newtheorem{notation}{Notation}

\newtheorem{innercustomthm}{Theorem}
\newenvironment{mythm}[1]
  {\renewcommand\theinnercustomthm{#1}\innercustomthm}
  {\endinnercustomthm}

\usepackage{titling}
\usepackage[a-2u]{pdfx}
\usepackage{pdfpages}

\newgeometry{vmargin={30mm}, hmargin={30mm,30mm}}   

\usepackage[utf8]{inputenc}
\usepackage[T1]{fontenc}
\usepackage{lmodern}
\usepackage{amsmath} 
\usepackage{amssymb} 
\usepackage{amsfonts}
\usepackage{amsthm}  
\usepackage{bbding} 
\usepackage{mathtools}
\usepackage[scr]{rsfso}
\usepackage{float}
\usepackage{ragged2e}

\usepackage{bbding} 

\usepackage{pdfpages}

\usepackage{lmodern}

\usepackage{amsmath}        
\usepackage{amsfonts}       
\usepackage{amsthm}         
\usepackage{bbding}         
\usepackage{bm}             
\usepackage{graphicx}       
\usepackage{fancyvrb}       

\usepackage{dcolumn}        
\usepackage{booktabs}       
\usepackage{paralist}       
\usepackage{xcolor}         

\usepackage{graphicx}       
\usepackage{fancyvrb}       
\usepackage[nottoc]{tocbibind} 
\usepackage{dcolumn}        
\usepackage{booktabs}       
\usepackage{paralist}       

\usepackage{pgfplots}
\usetikzlibrary{intersections, pgfplots.fillbetween}

\usepackage{tikz} 
\usepackage{tikz-3dplot} 
\usetikzlibrary{arrows}
\usepackage{xparse}
\usetikzlibrary{3d}
\usetikzlibrary{calc,intersections,through,backgrounds}
\usetikzlibrary{arrows}
\usepackage{tikz-3dplot} 

\tikzset{
    >=stealth',
    punkt/.style={
           rectangle,
           rounded corners,
           draw=black, thick,
           text width=6.5em,
           minimum height=2em,
           text centered},
    pil/.style={
           ->,
           thick,
           shorten <=2pt,
           shorten >=2pt,}
}

\title{Sections and Chapters}
\title{\Huge Proof complexity of universal algebra in a CSP dichotomy proof}
\author{\huge Azza Gaysin
  }
\date{%
    Department of Algebra, Faculty of Mathematics and Physics\\%
    Charles University in Prague, Prague 186 00, Czech Republic.\\%
    E-mail: azza.gaysin@gmail.com
}

\begin{document}
\allowdisplaybreaks

\maketitle

\begin{abstract}
The constraint satisfaction problem (CSP) can be formulated as a homomorphism problem between relational structures: given a structure $\mathcal{A}$, for any structure $\mathcal{X}$, whether there exists a homomorphism from $\mathcal{X}$ to $\mathcal{A}$. For years, it has been conjectured that all problems of this type are divided into polynomial-time and NP-complete problems, and the conjecture was proved in 2017 separately by Zhuk \cite{zhuk2020proof} and Bulatov \cite{8104069}. 

Zhuk's algorithm solves tractable CSPs in polynomial time. The algorithm is partly based on universal algebra theorems, proved in \cite{zhuk2020proof}. Informally, they state that after reducing some domain of an instance to its strong subuniverses, a satisfiable instance maintains a solution. 

In this paper, we present the formalization of the proofs of these theorems in the bounded arithmetic $W^1_1$ introduced in \cite{10.1007/978-3-540-30124-0_27}. The formalization, together with our previous results in \cite{gaysin2023proof}, shows that $W_1^1$ proves the soundness of Zhuk's algorithm, where by soundness we mean that any rejection of the algorithm is correct.  From the known relation of the theory to propositional calculus $G$, it follows that tautologies, expressing the non-existence of a solution for unsatisfiable instances, have short proofs in $G$.

\end{abstract}

\section{Introduction}

Constraint satisfaction problems (CSPs) form a wide class of decision problems, first studied in 1998 by Feder and Vardi in their attempt to discover a large subclass of NP that exhibits dichotomy \cite{doi:10.1137/S0097539794266766}. The problem CSP($\Gamma$) consists of a finite set $D$ and a finite collection $\Gamma=\{R_1,...,R_n\}$ of relations on $D$, or constraint language. The question is, given as input a list of variables $V$ and a list of constraints $\mathcal{C}=\{C_1,...,C_m\}$ - pairs of tuples of distinct variables and relations $R_i$, whether there is an assignment of variables to values in $D$ satisfying the given constraints. If such an assignment exists, an instance is called satisfiable and unsatisfiable otherwise. The equivalent definition is formulated as a homomorphism problem between relational structures. The obvious example is the $\mathcal{H}$-coloring problem, where $\mathcal{H}$ is a simple undirected graph without loops, whose vertices are considered as different colors. The $\mathcal{H}$-coloring of a graph $\mathcal{G}$ is an assignment of colors to the vertices of $\mathcal{G}$ such that adjacent vertices of $\mathcal{G}$ obtain adjacent colors. The generalization of this problem is a homomorphism problem from an input directed graph to a fixed target digraph. It is known that the latter problem is universal: for any constraint language $\Gamma$, CSP($\Gamma$) is logspace equivalent to a homomorphism problem for some digraph $\mathcal{H}$ \cite{barto_et_al}.

The first dichotomy result was formulated by Schaefer in 1978 for a problem over a binary domain called Generalized Satisfiability
\cite{10.1145/800133.804350}. The second result of this form was proved by Hell and Nešetřil in 1990 for the $\mathcal{H}$-coloring problem \cite{HELL199092}. Feder and Vardi conjectured that the whole class of CSPs satisfies the general dichotomy between P and NP \cite{doi:10.1137/S0097539794266766}, i.e. for a given $\Gamma$ the problem is either in P or is NP-complete. The proof of this conjecture was presented in 2017 by Zhuk \cite{zhuk2020proof} and Bulatov \cite{8104069}.

Zhuk's algorithm solves the problem of whether there exists a homomorphism from $\mathcal{X}$ to $\mathcal{A}$ for any tractable CSP($\mathcal{A}$) in polynomial time in size of $\mathcal{X}$, for any fixed $\mathcal{A}$. If an instance is satisfiable, then the algorithm produces a solution, i.e. a polynomial-size witness of an affirmative answer that one can independently check in polynomial time. The qualification 'independent' means that we can check the validity of the witness irrespective of how it was obtained, i.e. not understanding anything about Zhuk's algorithm. That is not the case for unsatisfiable instances. The only apparent polynomial size witness is the particular computation of Zhuk's algorithm on the instance. 

We use some proof complexity methods (formalization in theories of bounded arithmetic, propositional translations, etc.) to show that
the algorithm may be appended to provide an independent proof of the correctness of the algorithm for negative answers, too. The witness in the case of the homomorphism problem between relational structures is a short propositional proof of a formula $\neg HOM(\mathcal{X},\mathcal{A})$, encoding that there is no homomorphism from $\mathcal{X}$ to $\mathcal{A}$, in a particular well-known proof system, namely the quantified propositional calculus $G$ \cite{https://doi.org/10.1002/malq.19900360106}. When relational structures are directed graphs, for example, for a given digraph $\mathcal{X}$, the non-existence of a homomorphism from $\mathcal{X}$ to $\mathcal{A}$ can be expressed by the fact that a simple and transparent set of clauses is not satisfiable: it is built from propositional atoms $p_{ij}$, one for each vertex $i$ in $\mathcal{X}$ and vertex $j$ in $\mathcal{A}$, and says that the set of pairs $(i,j)$ for
which $p_{i j}$ is true is the graph of a map from $\mathcal{X}$ to $\mathcal{A}$, which is a homomorphism. Namely, for any two digraphs $\mathcal{X}=(V_\mathcal{X},E_\mathcal{X})$, $\mathcal{A}=(V_\mathcal{A},E_\mathcal{A})$, consider the following set of clauses: 
\begin{itemize}
    \item a clause $\bigvee_{j\in V_\mathcal{A}}^{} p_{i,j}$ for each $i\in V_\mathcal{X}$ (every vertex of $\mathcal{X}$ is sent to some vertex of $\mathcal{A}$);
    \item a clause $\neg p_{i,j_1}\vee\neg p_{i,j_2}$ for each $i\in V_\mathcal{X}$ and $j_1,j_2 \in V_\mathcal{A}$ with $j_1\neq j_2$ (the map is well-defined);
    \item a clause $\neg p_{i_1,j_1}\vee \neg p_{i_2,j_2}$ for every edge $(i_1,i_2)\in E_\mathcal{X}$ and $(j_1,j_2) \notin E_\mathcal{A}$ (a map is indeed a homomorphism).
\end{itemize}
A propositional refutation of these clauses in a transparent propositional calculus, whose soundness is obvious, thus indeed serves as a simple (and simple to check) witness for a negative answer to the algorithm. 

For this, we establish that the soundness of Zhuk’s algorithm can be proved in a theory of bounded arithmetic $W^1_1$ introduced in \cite{10.1007/978-3-540-30124-0_27}. By soundness, we mean the formula $Reject_{\mathcal{A}}(\mathcal{X},W) \implies \neg HOM(\mathcal{X},\mathcal{A})$,
where $Reject_{\mathcal{A}}(\mathcal{X},W)$ formalizes naturally that $W$ is the algorithm computation on input $\mathcal{X}$ that results in rejection, and $\neg HOM(\mathcal{X},\mathcal{A})$ means that there is no homomorphism from $\mathcal{X}$ to $\mathcal{A}$.

In \cite{gaysin2023proof} we have shown that for any fixed relational structure $\mathcal{A}$ that corresponds to an algebra $\mathbb{A}$ with a WNU operation, a new theory $V^1_{\mathcal{A}}$ proves the soundness of Zhuk's algorithm. Theory $V^1_{\mathcal{A}}$ extends the theory $V^1$, corresponding to the Extended Frege EF proof system, with three universal algebra axiom schemes BA$_{\mathcal{A}}$-axioms, CR$_{\mathcal{A}}$-axioms, and PC$_{\mathcal{A}}$-axioms. These axiom schemes consist of finitely many $\forall\Sigma^{1,b}_2$-formulas and reflect the main universal algebra theorems in Zhuk's paper \cite{zhuk2020proof}. Informally, they state that by reducing a domain of an instance to its binary absorbing subuniverse, central subuniverse, or PC subuniverse (see \cite{zhuk2020strong}), the algorithm does not lose all the solutions to the instance. In the paper, we use bounded arithmetic $W^1_1$ to formalize the proofs of these three axiom schemes. We show that all notions used in the proof of the soundness of Zhuk's algorithm in \cite{zhuk2020proof} can be formalized using bounded quantifier formulas (two or three sorted) and that the statements about them can be proved in the theory.
To show the latter, we selected a number of statements whose proofs represent
all types of argument (in particular, all types of inductive argument)
in \cite{zhuk2020proof}. To show that the formalization exists, we write the formal definitions of all objects used. This is because we need to know their quantifier complexity (various bounded arithmetic theories differ mainly in the class of formulas for which they assume induction).

The formalization, together with our previous results in \cite{gaysin2023proof} and the known relation of the theory to propositional calculus $G$, completes the proof of the following theorem. 

\begin{theorem}[The main theorem]\label{mainhhhdkruf} For any relational structure $\mathcal{A}$ such that CSP($\mathcal{A}$) is in $P$:
\item [1.] Theory $W_1^1$ proves the soundness of Zhuk's algorithm. That is, the theory proves the formula $Reject_{\mathcal{A}}(\mathcal{X},W) $ $\implies \neg HOM(\mathcal{X},\mathcal{A})$.
\item [2.] There exists a $p$-time algorithm $F$ such that for any unsatisfiable instance $\mathcal{X}$, i.e. such that $\neg HOM(\mathcal{X},$ $\mathcal{A})$, the output $F(\mathcal{X})$
of $F$ on $\mathcal{X}$ is a propositional proof of the proposition translation of
formula $\neg HOM(\mathcal{X},\mathcal{A})$ in propositional calculus $G$.
\end{theorem}
The paper is organized as follows. In Section \ref{akkasjdhkajdg} we briefly recall all necessary notions and facts from CSP theory, proof complexity, and bounded arithmetic, and outline the theoretical basis for Zhuk's algorithm. In Section \ref{'a'a;sldfkjfjfuryt} we
formalize all the notions used in the proof of the soundness of the algorithm in \cite{zhuk2020proof} and formulate three universal algebra axiom schemes. The formalization inherits the one carried out in \cite{gaysin2023proof} and we will not repeat some definitions already introduced there. In section \ref{'a'tteyrghwgdfh} we formalize the proof of these three axiom schemes in theory $W^1_1$ and prove the main theorem. The paper is concluded with Section \ref{akkajshddf} that points out possible continuations of this research.

\section{Preliminaries}\label{akkasjdhkajdg}

\subsection{A Third-Order Language and Correspondence to Sequent calculus $G$}

In this subsection, most definitions and results are adapted from \cite{10.5555/1734064}, \cite{krajicek_1995}, \cite{krajíček_2019}, \cite{10.1007/978-3-540-30124-0_27}.

Theories of bounded arithmetic, referred to as \emph{second-order} (or \emph{two-sorted} first-order) theories use a specific framework. There are two distinct categories of variables: variables $x,y,z,...$ of the first kind are \emph{number variables}, and they represent natural numbers. Variables $X,Y,Z,...$ of the second kind are \emph{set variables}, and they correspond to finite subsets of natural numbers, which are equivalent to binary strings. Functions and predicate symbols can involve variables of both types, and functions are classified as either \emph{number-valued} or \emph{set-valued}. Furthermore, there are two distinct types of quantifiers: those over number variables, called \emph{ number quantifiers}, and those used for set variables, called \emph{string quantifiers}. The language used in these second-order theories expands on the conventional language of Peano arithmetic $\mathcal{L}\mathcal{_{PA}}$, and is represented as:
$$\mathcal{L}^2\mathcal{_{PA}} = \{0,1,+,\cdot,\lvert\,    \rvert,=_1,=_2,\leq, \in\}.$$ 
In this notation, the symbols $0,1,+,\cdot,=_1$ and $\leq$ are functions and predicates related to the number variables. The function $\lvert X\rvert$ (called the \emph{length of $X$}) is number-valued and indicates the length or upper bound of the string $X$. The binary predicate $\in$ denotes set membership between number and set variables, and we abbreviate $t \in X$ as $X(t)$, thinking of $X(i)$ as of the $i$-th bit of binary string $X$ of length $\lvert X\rvert$. The relation $=_2$ represents the equality predicate for sets. The set of axioms that define the basic properties of symbols from $\mathcal{L}^2\mathcal{_{PA}}$ is denoted by $2$-BASIC \cite{10.5555/1734064}. 

When $x, X$ do not occur in a term $t$, then the expression $\exists x \leq t \phi$ is equivalent to $\exists x( x \leq t \wedge \phi)$, $\forall x \leq t \phi$ corresponds to $\forall x(x \leq t \to \phi)$, $\exists X \leq t\phi$ corresponds to $\exists X(|X|\leq t \wedge \phi)$ and $\forall X\leq t \phi$ stands for $\forall X(|X|\leq t \to \phi)$. Quantifiers used in this manner are called \emph{bounded}. A \emph{bounded formula} refers to any formula in which every quantifier is bounded. We denote by $\Sigma^{1,b}_0=\Pi^{1,b}_0$ the set of formulas with all number quantifiers bounded and with no string quantifiers. For $i\geq 0$, $\Sigma^{1,b}_i$ (resp. $\Pi^{1,b}_i$) is the set of formulas of the form $\exists \bar X\leq \bar t\phi$ (resp. $\forall \bar X\leq \bar t\phi$), where $\phi$ is the $\Pi^{1,b}_i$-formula (resp. $\Sigma^{1,b}_i$-formula), and $\bar t$ is a sequence of $\mathcal{L}^2\mathcal{_{PA}}$-terms that do not involve any variable in $\bar X$. 

Let $\Phi$ be a set of two-sorted formulas, and $\phi\in \Phi$. The \emph{number induction axiom scheme} for the set $\Phi$, denoted by $\Phi$-IND, is the set of formulas
\begin{equation}
  \phi(\bar{x},\bar{X},0) \wedge \forall y (\phi(\bar{x},\bar{X},y)\to \phi(\bar{x},\bar{X},y+1))\to \forall z \phi(\bar{x},\bar{X},z).
\end{equation}
The \emph{number minimization axiom scheme} (or the \emph{least number principle}) for the set $\Phi$ is denoted by $\Phi$-MIN and consists of the formulas
\begin{equation}
   \phi(\bar{x},\bar{X},y)\to \exists z\leq y (\phi(\bar{x},\bar{X},z)\wedge \neg \exists u<z \,\phi(\bar{x},\bar{X},u)).
\end{equation}
Finally, the \emph{comprehension axiom scheme} $\Phi$-COMP is the set of all formulas
\begin{equation}
    \forall y\exists Y\leq y\,\forall z< y \, \phi(\bar{x},\bar{X},z)\iff Y(z).
\end{equation}
and $Y$ does not occur free in $\phi(\bar{x},\bar{X},z)$.

\begin{definition}[The theory $V^1$]
The theory $V^1$ is a second-order theory axiomatized by $2$-BASIC, $\Sigma^{1,b}_0$-COMP, and $\Sigma^{1,b}_1$-IND.

\end{definition}
In general, for $i\geq 0$, $V^i$ is the same as $V^1$ except that $\Sigma^{1,b}_1$-IND is replaced by $\Sigma^{1,b}_i$-IND.
\begin{lemmach}
For $i>0$, the theory $V^i$ proves $\Sigma^{1,b}_i$-COMP and $\Sigma^{1,b}_i$-MIN. 
\end{lemmach}

To refer to exponentially large objects such as power sets, congruences on products of algebras, and so forth, we use the setting introduced in \cite{10.1007/978-3-540-30124-0_27}. In addition to free and bound variables of first and second sorts, we consider variables of a third sort that represent finite sets of finite sets, named $\mathscr{A,B,C},...$ and $\mathscr{X,Y,Z},...$. We refer to second-sort objects as 'strings', and to third-sort objects as 'classes' (note that in the original setting in \cite{10.1007/978-3-540-30124-0_27} classes were referred to as 'superstring', but this name does not reflect the type of objects we discuss). The language $\mathcal{L}^3\mathcal{_{PA}}$ contains an additional symbol for the third-order membership predicate $A\in_3 \mathscr{B}$,
$$\mathcal{L}^3\mathcal{_{PA}} = \{0,1,+,\cdot,\lvert\,  \rvert,=_1,=_2,\leq, \in_2,\in_3\}.$$ 
Classes can be also thought of as strings of bits, where each bit is indexed by a set referred to as \emph{bit-index}. There is no length-function analog for classes, so the 'length' of a class in this setting is the lexicographically maximal bit-index under consideration. Number terms are defined as in $V^1$, not including any reference to third-order variables, while formulas additionally may have third-order variables and quantifiers. Although third-order variables are unbounded because of the absence of a length function, they will be implicitly bounded, in the sense that the bounds on first-order and second-order quantifiers will limit the part of the class that affects the truth value of a given formula.

We extend the hierarchy $\Sigma^{1,b}_i$ of second-order formulas to third-order classes $\Sigma^{\mathscr{B}}_i$ that consist of those formulas with arbitrarily many bounded first-order and second-order quantifiers, and exactly $i$ alternating unbounded third-order quantifiers, the outermost
being restricted, i.e. equivalent to the existential quantifier. Let $\Phi$ be a set of third-sorted formulas. The number induction axiom scheme for the set $\Phi$, $\Phi$-IND, is the set of formulas\begin{equation}
  \phi(\bar{x},\bar{X},\bar{\mathscr{X}},0) \wedge \forall y (\phi(\bar{x},\bar{X},\bar{\mathscr{X}},y)\to \phi(\bar{x},\bar{X},\bar{\mathscr{X}},y+1))\to \forall z \phi(\bar{x},\bar{X},\bar{\mathscr{X}},z).
\end{equation}
We define the following two comprehension axiom schemes. $\Phi$-2COMP is defined as 
\begin{equation}
   \forall y\exists Y\leq y\forall z\leq y\big(\phi(\bar{x},\bar{X},\bar{\mathscr{X}},z)\iff Y(z)\big),
\end{equation}
and $\Phi$-3COMP is
\begin{equation}
    \forall y\exists \mathscr{Y}(\forall Z\leq y)\big(\phi(\bar{x},\bar{X},\bar{\mathscr{X}},Z)\iff \mathscr{Y}(Z)\big), 
\end{equation}
where in each case $\phi\in\Phi$ subjects to the restriction that neither $Y$ nor $\mathscr{Y}$ occur free in $\phi$. 

\begin{definition}[The theory $W^1_1$, \cite{10.1007/978-3-540-30124-0_27}]
Theory $W^1_1$ is a third-sorted theory axiomatized by induction axiom scheme $\Sigma^{\mathscr{B}}_1$-IND, and both comprehension axiom schemes $\Sigma^{\mathscr{B}}_0$-2COMP, $\Sigma^{\mathscr{B}}_0$-3COMP.
\end{definition}
In general, for $i\geq 0$, $W^i_1$ is the same as $W^1_1$ except that $\Sigma^{\mathscr{B}}_1$-IND is replaced by $\Sigma^{\mathscr{B}}_i$-IND. Thus, $\Sigma^{\mathscr{B}}_0$-definable functions of number and string arguments are usual $p$-time hierarchy functions. $\Sigma^{\mathscr{B}}_1$-definable functions of $W^1_1$ are exactly FPSPACE$^+$, a third-order analogue of PSPACE functions (see \cite{articledddfgr}). $W^1_1$ can $\Sigma^{\mathscr{B}}_0$-define all number and string-valued functions of number and string arguments from the polynomial-time hierarchy.

The class of quantified propositional formulas, denoted by $\Sigma^{q}_{\infty}$, is the smallest class of formulas containing atoms $0,1$, and closed under logical connectives and Boolean quantification: if $\phi(x)$ is a formula in $\Sigma^{q}_{\infty}$, then so are $\exists x\phi(x)$ and $\forall x\phi(x)$, and the meaning is $\phi(0)\vee\phi(1)$ and $\phi(0)\wedge \phi(1)$, respectively. The proof system $G$ for quantified propositional formulas extends a classical proof system, the sequent calculus LK, see \cite{krajíček_2019}.

\begin{definition}[Sequent calculus $G$] 
Quantified propositional calculus $G$ extends system LK by allowing
$\Sigma^q_{\infty}$-formulas in sequents and by accepting two quantifier rules:
\begin{enumerate}
  \item $\forall:$introduction
$$\text{left:\,\,\,}
  \frac{\phi(\psi),\Gamma\longrightarrow \Delta}{\forall x\phi(x),\Gamma \longrightarrow \Delta}\,\,\,\,\,\,\,\,\,\,\,\,\,\,\,\,\,\,\,\,\text{right:\,\,\,}\frac{\Gamma\longrightarrow\Delta,\phi(p)}{\Gamma\longrightarrow \Delta,\forall x\phi(x)}$$

  \item $\exists:$introduction $$\text{left:\,\,\,}
  \frac{\phi(p),\Gamma\longrightarrow \Delta}{\exists x\phi(x),\Gamma \longrightarrow \Delta}\,\,\,\,\,\,\,\,\,\,\,\,\,\,\,\,\,\,\,\,\text{right:\,\,\,}\frac{\Gamma\longrightarrow\Delta,\phi(\psi)}{\Gamma\longrightarrow \Delta,\exists x\phi(x)}$$
\end{enumerate}
where $\psi$ is any formula such that no variable occurrence free in $\psi$ becomes quantified in $\phi(\psi)$, and variable $p$ does not occur in lower sequences of inference rules. 
\end{definition}

In \cite{10.5555/1734064} there is presented the well-known translation of any $\phi(\bar{x},\bar{X}) \in \Sigma^{1,b}_0$ into a family of propositional formulas,
\begin{equation}
||\phi(\bar{x},\bar{X})|| = \{\phi(\bar{x},\bar{X})[\bar{m},\bar{n}]: \bar{m},\bar{n} \in \mathbb{N}\}
\end{equation}
such that the following lemma holds: 
\begin{lemmach}[\cite{10.5555/1734064}]\label{LPropositionalTranslation} For every $\Sigma_0^{1,b}(\mathcal{L}^2\mathcal{_{PA}})$ formula $\phi(\bar{x},\bar{X})$, there is a constant $d\in \mathbb{N}$ and a polynomial $p(\bar{m},\bar{n})$ such that for all $\bar{m},\bar{n} \in \mathbb{N}$, the propositional formula $\phi(\bar{x},\bar{X})[\bar{m},\bar{n}]$ has depth at most $d$ and size at most $p(\bar{m},\bar{n})$.
\end{lemmach}
Propositional translation of formulas $\Sigma_0^{1,b}(\mathcal{L}^2\mathcal{_{PA}})$ can be extended to the translation of any bounded $\mathcal{L}^2\mathcal{_{PA}}$-formula into a quantified propositional formula, using strings of Boolean quantifiers to represent second-order quantifiers, \cite{10.5555/1734064}. Let us denote the class of all bounded $\mathcal{L}^2\mathcal{_{PA}}$ formulas by $$\Sigma^{1,b}_{\infty} = \bigcup_{i}\Sigma^{1,b}_i= \bigcup_{i}\Pi^{1,b}_i.$$ 
The following theorem establishes the correspondence between theory $W^1_1$ and quantified propositional calculus $G$. It follows from \cite{10.1007/978-3-540-30124-0_27}, specifically from Theorems $12,13$.
\begin{theorem} \label{Translation[[[po]]]} 
Suppose that $\phi(\bar{x},\bar{X})$ is a $\Sigma_{\infty}^{1,b}$ formula such that $W^1_1 \vdash \phi(\bar{x},\bar{X})$. Then the propositional family $||\phi(\bar{x},\bar{X})||$ has quantified propositional calculus proofs of polynomial size. That is, there is a polynomial $p(\bar{m},\bar{n})$ such that for all $1\leq \bar{m},\bar{n}\in \mathbb{N}$, $\phi(\bar{x},\bar{X})[\bar{m},\bar{n}]$ has a $G$-proof of size at most $p(\bar{m},\bar{n})$. Furthermore, there is an algorithm that finds a $G$-proof of $\phi(\bar{x},\bar{X})[\bar{m},\bar{n}]$ in time bounded by a polynomial in $(\bar{m},\bar{n})$.
\end{theorem}

\subsection{Universal algebra basics}

We say that a \emph{vocabulary} is a finite set of relational symbols $R_1$,..., $R_n$ of a fixed arity. A \emph{relational structure} over the vocabulary is a tuple $\mathcal{A} = (A, R^\mathcal{A}_1,...,$ $R^\mathcal{A}_n)$ such that $A \neq \emptyset$ is the \emph{universe} of $\mathcal{A}$, and each $R^\mathcal{A}_i$ is a relation on $A$ of the same arity as $R_i$. If $\mathcal{X}$, $\mathcal{A}$, are relational structures over the same vocabulary $R_1$,..., $R_n$, then a \emph{homomorphism} $h$ from $\mathcal{X}$ to $\mathcal{A}$ is a mapping $h: X \rightarrow A$ such that for every $m$-ary relation $R^\mathcal{X}$ and every tuple $(x_1,...,x_m) \in R^\mathcal{X}$ we have $(h(x_1),...,h(x_m)) \in R^\mathcal{A}$.

An \emph{algebra} $\mathbb{A}=(A,f_1,f_2,...)$ is a pair of a domain $A$ and basic operations $f_1,f_2,...$ of fixed arities from \emph{language} (or \emph{type}) $\Sigma = \{f_1,f_2,...\}$. The subset of $n$-ary function symbols in $\Sigma$ is denoted by $\Sigma_n$. If we denote by $X$ the set of all variables, then the \emph{set of terms of type} $\Sigma$ is the smallest set that contains $X\cup \Sigma_0$, and if $p_1,...,p_n$ are terms and $f\in \Sigma_n$, then the string $f(p_1,...,p_n)$ is also a term. An \emph{identity} of type $\Sigma$ over $X$ is an expression of the form $s\approx t$ where $s,t$ are terms. We call the set of all term operations of \emph{algebra} $\mathbb{A}=(A,F)$ the \emph{clone of term operations} of $\mathbb{A}$, and denote it by $Clone(\mathbb{A})$. In general, for any set of operation $O$ on $A$, we denote by $Clone(O)$ the set of all term operations generated by $O$.

An $m$-ary operation $f: A^m\rightarrow A$ is a \emph{polymorphism} of an $n$-ary relation $R\in A^n$ (or $f$ \emph{preservers} $R$, or $f$ is \emph{compatible} with $R$, or $R$ is \emph{invariant} under $f$) if $f(\bar{a_1},...,\bar{a_m}) \in R$ for all tuples $\bar{a_1},...,\bar{a_m} \in R$. For any relational structure $\mathcal{A}$ we denote by $Pol(\mathcal{A})$ the set of all operations on $A$ preserving each relation from $\mathcal{A}$. Generally, for every set of relations $\Gamma$ on $A$, by $Pol(\Gamma)$ we denote the set of all polymorphisms of $\Gamma$. The following theorem establishes the connection between algebras and relational structures. 

\begin{theorem}[\cite{bergman2011universal}]\label{fjjduh87}
For any algebra $\mathbb{A}$ there exists relation structure $\mathcal{A}$ such that $Clone(\mathbb{A})$ $ =Pol(\mathcal{A})$.
\end{theorem}
A subset $B\subseteq A$ is a \emph{subuniverse} of algebra $\mathbb{A}$ if it is closed under all operations of $\mathbb{A}$. Any subuniverse $B$ can be viewed as an algebra $\mathbb{B}$, and we call it \emph{subalgebra} of $\mathbb{A}$, denoted by $\mathbb{B}\leq \mathbb{A}$. For every subset $B\subseteq A$ we denote by $Sg(B)$ the subalgebra generated by $B$, i.e. the minimal subalgebra of $\mathbb{A}$ containing $B$. For any subset $B$, $Sg(B)$ can be generated by a closure operator $Cl$, defined as follows: $Cl(B)=B\cup\{f(a_1,...,a_n): f\text{ is a basic operation on }A,\,a_1,...,a_n\in B\}$, and $Cl^{t}(B)$ for $t\geq 0$ by $Cl^0(B)=B, Cl^{t+1}(B) = Cl(Cl^t(B))$. Then, $Sg(B) = B\cup Cl(B)\cup Cl^2(B)\cup...$

An equivalence relation $\sigma$ on $\mathbb{A}$ such that any term operation on $\mathbb{A}$ is compatible with it is a \emph{congruence}. Trivial congruences on $\mathbb{A}$ are the diagonal relation $\Delta_{A}=\{(a,a): a\in A\}$ and the full relation $\nabla_{A}=A^2$. A congruence is \emph{maximal} if it is inclusion maximal. For any congruence $\sigma$ on $\mathbb{A}$, one can introduce a \emph{quotient} (or \emph{factor}) algebra $\mathbb{A}/\sigma$. As the universe $\mathbb{A}/\sigma$ has the set of $\sigma$-classes $A/\sigma$ and the operations there are defined using arbitrary representatives from $A/\sigma$. Due to the compatibility property, this will indeed define a function on $A/\sigma$. A congruence $\sigma$ defines a subalgebra of $\mathbb{A}^2$: when applying term operation to elements from $\sigma$ coordinatewise, compatibility property gives an element from $\sigma$. In general, any invariant under all term operation $n$-ary relation $R$ on $\mathbb{A}$ forms a subalgebra of $\mathbb{A}^n$.

\subsection{Constraint satisfaction problems}

For this section, some definitions, examples, and results are adapted from
\cite{barto_et_al},
\cite{zhuk2020proof}, and \cite{Zivn:2012:CVC:2412078}. We define a constraint satisfaction problem as in \cite{zhuk2020proof}.  

\begin{definition}[CSP \cite{zhuk2020proof}] An instance of \emph{constraint satisfaction problem} (CSP) over finite domains is a triple $\Theta = (X,D,C)$, where
\begin{itemize}
    \item $X = \{x_0,...,x_{n-1}\}$ is a finite set of variables,
    \item $D = \{D_0,...,D_{n-1}\}$ is a set of non-empty finite domains,
    \item $C=\{C_0,...,C_{t-1}\}$ is a set of constraints.
\end{itemize}
Each variable $x_i$ can take values over the corresponding domain $D_i$. Every constraint $C_j \in C$ is a pair $(\vec{x}_j,\rho_j)$, where $\vec{x}_j$ is a tuple of variables of length $m_j$, called a \emph{constraint scope}, and $\rho_j$ is an $m_j$-ary relation on the product of the domains, called a \emph{constraint relation}. The solution to the instance $\Theta$ is an assignment to every variable $x_i$ such that for each constraint $C_j$ the image of the constraint scope is a member of the constraint relation.  
\end{definition}

In the definition above we can consider every $D_i$ as a unary constraint relation over the superdomain imposed on a variable $x_i$. In such an interpretation, the definition coincides with a classical one, given for example in \cite{barto_et_al}. We can define the CSP equivalently as a homomorphism problem between relational structures.
\begin{definition}
Let $\mathcal{A}$ be a relational structure over a vocabulary $R_1$,..., $R_n$. In the \emph{constraint satisfaction problem} associated with $\mathcal{A}$, denoted by CSP($\mathcal{A}$), the question is, given a structure $\mathcal{X}$ over the same vocabulary, whether there exists a homomorphism from $\mathcal{X}$ to $\mathcal{A}$. We will call $\mathcal{A}$ the \emph{target structure} and $\mathcal{X}$ the \emph{instance (or input) one}. 
\end{definition}

We call a set of relations $\Gamma$ over a finite domain a \emph{constraint language}. A CSP associated with $\Gamma$, denoted by CSP($\Gamma$), is a subclass of CSP such that any constraint relation in any instance of CSP($\Gamma$) belongs to $\Gamma$. For years, it has been conjectured that for any constraint language $\Gamma$, CSP($\Gamma$) is either polynomial time or NP-complete. To formulate the CSP dichotomy theorem, we need the following notion. We call an operation $\Omega$ on a set $A$ a \emph{weak-near unanimity operation} (WNU) if it satisfies the identities 
$$
\Omega(y,x,x,...,x)=\Omega(x,y,x,...,x)=...=\Omega(x,x,...,x,y)
$$
for all $x,y\in A$. An operation $\Omega$ is called \emph{idempotent} if $\Omega(x,...,x)=x$ for every $x\in A$, and is called \emph{special} if for all $x,y\in A$,
$
\Omega(x,...,x,\Omega(x,...,x,y))=\Omega(x,...,x,y).$ It is known \cite{MAROTI} that for any idempotent WNU operation $\Omega$ on a finite set, there exists a special WNU operation $\Omega' \in Clone(\Omega)$.

\begin{theorem}[CSP dichotomy theorem \cite{zhuk2020proof}] Suppose $\Gamma$ is a finite set of relations on a set $A$. Then CSP($\Gamma$) can be solved in polynomial time if there exists a WNU operation $\Omega$ on $A$ preserving $\Gamma$; CSP($\Gamma$) is NP-complete otherwise.
\end{theorem}

Instead of a constraint language $\Gamma$, it is reasonable to consider larger classes of languages containing $\Gamma$, whenever there is a polynomial time reduction. The so-called primitive-positive (abbreviated by $pp$) constructions provide such a reduction between the corresponding CSPs, i.e. if $\Gamma$ $pp$-constructs $\Gamma'$, then CSP($\Gamma'$) is log-space reducible to CSP($\Gamma$). $Pp$-construction is the most general technique for simulation between CSPs, which contains $pp$-interpretations and $pp$-definitions. We say that a language $\Gamma$ \emph{$pp$-defines} $\Gamma'$ if every relation in $\Gamma'$ can be defined by a first-order formula which only uses relations in $\Gamma$, the equality relation, conjunction, and
existential quantification. The following theorem explains the importance of this special case of constructions. 
\begin{theorem}\label{skskjdi}
Let $\mathcal{A}=(A,\Gamma)$ be a finite relational structure, and let $R\subseteq A^n$ be a non-empty relation. Then $R$ is preserved by all polymorphisms of $\Gamma$ if and only if $R$ is $pp$-definable from $\Gamma$.
\end{theorem}
Thus, all relations $pp$-definable over $\Gamma$ are invariant under all polymorphisms that preserve $\Gamma$. We will not define two other techniques, since their definitions are irrelevant to the matter of the paper, and we refer the reader to \cite{barto_et_al} for more information. We only mention two known results.

\begin{theorem}[\cite{10.1145/2677161.2677165}]
A constraint language $\Gamma$ $pp$-interpreters a constraint language $\Gamma'$ if and only if in $Pol(\Gamma')$ there exist operations satisfying all identities that are satisfied by operations in $Pol(\Gamma)$.
\end{theorem}
Thus, $pp$-interpretability does not change the identities satisfied in the corresponding algebras. The next result allows one to work with at most binary constraint languages, which often simplifies the representation of the problem.
\begin{theorem}[\cite{barto_et_al}]\label{aakrrrfetsplgh}
For any constraint language $\Gamma$ there is a constraint language $\Gamma'$ such that all relations in $\Gamma'$ are at most binary and $\Gamma$ and $\Gamma'$ $pp$-constructs each other.
\end{theorem}

At the end of this section, we define some properties of CSP instances that express different levels of consistency. A relation $R \subseteq D_{i_1}\times ...\times D_{i_{k}}$ is called \emph{subdirect} if for every $i$ projection of $R$ into the coordinate $i$-th is equal to $D_i$. A CSP instance is called \emph{$1$-consistent} if for every constraint $C_i$ of the instance the corresponding relation $R_i \subseteq D_{i_1}\times ...\times D_{i_{k}}$ is subdirect. Any instance can be turned into $1$-consistent with the same set of solutions using a simple algorithm \cite{barto_et_al}. Consider a CSP instance $\Theta$ on a domain $D=\{D_0,...,D_{n-1}\}$, and let $D_y$ denote the domain of the variable $y\in \{x_0,...,x_{n-1}\}$. We say that the sequence
$y_1-C_1-y_2 -...-y_{l-1} - C_{l-1} - y_l$ is a \emph{path} in $\Theta$ if $\{y_i,y_{i+1}\}$ are in the scope of $C_i$ for every $i<l$ (regardless of the order of the variables $y_i,y_{i+1}$ in $C_i$). We say that a path \emph{connects} $b$ and $c$ if for every $i$ there exists $a_i\in D_{y_i}$ such that $a_1 = b$, $a_l=c$ and the projection of $C_i$ onto $\{y_i,y_{i+1}\}$ contains the tuple $(a_i,a_{i+1})$. Then $\Theta$ is called \emph{cycle-consistent} if it is $1$-consistent and for every variable $y$ and $a\in D_y$ \emph{any} path starting and ending with $y$ connects $a$ and $a$. We say that $\Theta$ is \emph{linked} if for every variable $y$ in the scope of a constraint $C$ and for all $a,b\in D_y$ there \emph{exists} a path starting and ending with $y$ in $\Theta$ that connects $a$ and $b$. An instance is called \emph{fragmented} if the set of variables $X$ can be divided into $2$ disjoint sets $X_1$ and $X_2$ such that each of them is non-empty, and the constraint scope of any constraint of the instance has variables either only from
$X_1$, or only from $X_2$. We call an instance $\Theta =(X,D,C)$ \emph{irreducible} if any instance $\Theta' = (X',D',C')$ such that $X'\subseteq X$, $D'_x = D_x$ for every $x\in X'$, and every constraint of $\Theta'$ is a projection of a constraint from $\Theta$ on \emph{some} subset of variables from $X'$, is fragmented, or linked, or its solution set is subdirect. Finally, we say that a constraint $C_1 = ((y_1,...,y_t),$ $R_1)$ is \emph{weaker or equivalent} to a constraint $C_2 = ((z_1,...,z_s),R_2)$ if $\{y_1,...,y_t\}\subseteq \{z_1,...,z_s\}$ and $C_2$ implies $C_1$, i.e. the solution set to $\Theta_1=(\{z_1,...,z_s\},(D_{z_1},$ $...,D_{z_s}),C_1)$ contains the solution set to $\Theta_2=(\{z_1,...,z_s\},(D_{z_1},...,D_{z_s}),C_2)$. We say that $C_1$ is \emph{weaker} than $C_2$ (denoted by $C_1\leq C_2$) if $C_1$ is weaker or equivalent to $C_2$, but $C_1$ does not imply $C_2$. Thus, there are $2$ types of weaker constraints: $C_1 \leq C_2$ if either the arity of the relation $R_1$ is less than the arity of relation $R_2$ and for any tuple $(a_{z_1},...,a_{z_s})\in R_2$, $(a_{y_1},...,a_{y_t})\in R_1$, or the arities of relations $R_1$ and $R_2$ are equal, and $R_2\subsetneq R_1$.

\subsection{Zhuk's algorithm}
Zhuk's algorithm solves polynomial-time CSP cases and is based on the fact that in a CSP instance, any domain has either one of three kinds of proper strong subsets or an equivalence relation modulo which the domain is a product of prime fields. We first define these kinds of strong subsets. 
\begin{definition}[Absorbing subuniverse]
If $\mathbb{B}=(B,F_B)$ is a subalgebra of $\mathbb{A}=(A,F_A)$, then $B$ \emph{absorbs} $\mathbb{A}$ if there exists an $n$-ary term operation $f\in Clone(F_A)$ such that $f(a_1,...,a_n)\in B$ whenever the set of indices $\{i: a_i\notin B\}$ has at most one element. $B$ \emph{binary absorbs} $A$ if there exists a binary term operation $f \in Clone(F_A)$ such that $f(a,b)\in B$ and $f(b,a)\in B$ for any $a\in A$ and $b\in B$.
\end{definition}

\begin{definition}[Central subuniverse]
A subuniverse $C$ of $\mathbb{A}$ is called \emph{central} if it is an absorbing subuniverse and for every $a\in A\backslash C$ we have $(a,a)\notin Sg(\{a\}\times C\cup C\times \{a\})$.
\end{definition}
Every central subuniverse is a ternary absorbing subuniverse. In the original version of the algorithm, Zhuk uses a stronger notion of a center \cite{zhuk2020proof}. However, it is known that a central subuniverse has all the good properties of a center and can be used in the algorithm instead of the center.  Both versions of the algorithm, with the center or central universe, will correctly answer whether an instance has a solution. 

\begin{definition}[Polynomially complete algebra]
An $n$-ary operation $f$ on algebra $\mathbb{A}=(A,F_A)$ is called \emph{polynomial} if there exist some $(n+t)$-ary operation $g\in Clone(F_A)$ and constants $a_1,...,a_t\in A$ such that for all $x_1,...,x_n\in A$, 
$f(x_1,...,x_n)=g(x_1,...,x_n,a_1,...,$ $a_m)$. We call an algebra $\mathbb{A}=(A, F_A)$ \emph{polynomially complete} (PC) if its polynomial clone is the clone of all operations on $A$.
\end{definition}
In simple words, a universal algebra $\mathbb{A}$ is polynomially complete if every function on $A$ with values in $A$ is a polynomial function. The notion of linear algebra in the following form was introduced by Zhuk in \cite{zhuk2020proof}. 
\begin{definition}[Linear algebra, \cite{zhuk2020proof}]
An idempotent finite algebra $\mathbb{A}=(A,\Omega)$, where $\Omega$ is an $m$-ary idempotent special WNU operation, is called \emph{linear} if it is isomorphic to $(\mathbb{Z}_{p_1}\times ...\times \mathbb{Z}_{p_s},x_1+...+x_m)$ for prime (not necessarily distinct) numbers $p_1,...,p_s$. For every finite idempotent algebra, there exists the smallest congruence (not necessarily proper), called the \emph{minimal linear congruence}, such that the factor algebra is linear. 
\end{definition}
\begin{theorem}[\cite{zhuk2020proof}]\label{4cases}
If $\mathbb{A}$ is a non-trivial finite idempotent algebra with WNU operation, then at least one of the following is true:
\begin{itemize}
    \item $\mathbb{A}$ has a non-trivial binary absorbing subuniverse,
    \item $\mathbb{A}$ has a non-trivial central absorbing subuniverse,
    \item $\mathbb{A}$ has a non-trivial PC quotient,
    \item $\mathbb{A}$ has a non-trivial linear quotient.
\end{itemize}
\end{theorem}
We now briefly sketch Zhuk's algorithm without any details that are not important for the matter of the paper. For further information, we refer the reader to papers \cite{zhuk2020proof}, \cite{zhuk2020strong}. Consider any finite language $\Gamma'$ preserved by an idempotent special WNU operation $\Omega'$. Before running the algorithm, produce the following modifications of the language. If $k'$ is the maximal arity of the relations in $\Gamma'$, then denote by $\Gamma$ the set of all relations preserved by $\Omega$ of arity at most $k'$. Since all $pp$-definable relations of arity at most $k'$ are in $\Gamma$, it follows that CSP$(\Gamma')$ is an instance of CSP($\Gamma$). Consider a CSP instance of CSP($\Gamma$), $\Theta = (X,D,C)$. 

The algorithm is divided into two parts, which we will call general and linear parts. In the general part, the main idea is to gradually reduce domains until the algorithm moves to the linear part. There are several types of reduction. At any step, the algorithm either produces a reduced domain, or moves to the other type of reduction, or answers that there is no solution. When the algorithm reduces domains, it uses recursion: after outputting a reduced domain, the algorithm runs from the beginning for the same instance $\Theta$ but with a smaller domain $D'$. The algorithm first reduces domains based on different types of consistency and then based on different kinds of strong subuniverses. Thus, the algorithm checks if the instance is cycle-consistent. If not, it reduces the domains until it is. Further, the algorithm checks the irreducibility. Again, if the instance is not irreducible, the algorithm produces a reduction to some domain. Finally, it checks a weaker instance that is produced from the instance by simultaneously replacing all constraints with all weaker constraints: if the solution set to such an instance is not subdirect, then the algorithm reduces a domain. It is easy to prove that the instance after these consistency reductions does not lose any solution. After these steps, the reduction will be based on strong subuniverses. First, the algorithm checks whether domains have a non-trivial binary absorbing subuniverse or a non-trivial central subuniverse. If any of them does, the algorithm reduces the domain to the subuniverses. The algorithm then checks whether there are proper congruences on domains such that their factor algebras are polynomially complete. If there is such a congruence, then the algorithm reduces the domain to some equivalence class of the congruence. The reduction to these strong subsets is justified by the following theorems, proved in \cite{zhuk2020proof}.

\begin{theorem}[\cite{zhuk2020proof}]\label{fhfhfhydj} Suppose $\Theta$ is a cycle-consistent irreducible CSP instance, and $B$ is a non-trivial binary absorbing subuniverse of $D_i$. Then $\Theta$ has a solution if and only if $\Theta$ has a solution with $x_i\in B$.
\end{theorem}

\begin{theorem}[\cite{zhuk2020proof},\cite{zhuk2020strong}]\label{llldju67yr} Suppose $\Theta$ is a cycle-consistent irreducible CSP instance, and $B$ is a non-trivial central subuniverse of $D_i$. Then $\Theta$ has a solution if and only if $\Theta$ has a solution with $x_i\in B$.
\end{theorem}

\begin{theorem}[\cite{zhuk2020proof}]\label{fjfjhgd}
Suppose $\Theta$ is a cycle-consistent irreducible CSP instance, there does not exist a non-trivial binary absorbing subuniverse or a non-trivial center on $D_j$ for every $j$, $(D_i,\Omega)/\sigma_i$ is a polynomially complete algebra, and $E$ is an equivalence class of $\sigma_i$. Then $\Theta$ has a solution if and only if $\Theta$ has a solution with $x_i\in E$.
\end{theorem}
If at some point after the consistency reductions, there are no strong subuniverses in any domain, then by Theorem \ref{4cases} there must be a non-trivial linear congruence on every domain $D_i$, and the algorithm moves to the linear part. In that part, the algorithm considers minimal linear congruences in every domain and constructs the so-called factorized instance $\Theta_L$. This instance can be viewed as a system of linear equations and solved by polynomial-time Gaussian elimination. After that, the algorithm starts to compare the solution set to the original instance, factorized by linear congruences, with the solution set to the factorized instance. If the inclusion is proper, the algorithm adds a new linear equation to the system until a solution is found. For comparison, the algorithm restricts domains to some equivalence classes and again calls the recursion. We will not elaborate further on this part of the algorithm, since its formalization was already presented in \cite{gaysin2023proof}.

\subsection{Soundness of Zhuk's algorithm in a theory of bounded arithmetic}\label{SOUNDNggghESS}
Under the soundness of Zhuk's algorithm, we mean that any rejection of the algorithm is correct. To establish the soundness of the algorithm in a theory of bounded arithmetic, it suffices to prove that after every step the algorithm preserves some of the solutions to the initial instance. 

Consider any relational structure $\mathcal{A}$ and some negative instance $\Theta=(\mathcal{X},\mathcal{A})$ of CSP($\mathcal{A}$), i.e. such that there is no homomorphism from $\mathcal{X}$ to $\mathcal{A}$. Consider the computation of the algorithm on $(\mathcal{X},\mathcal{A})$, $W = (W_1,W_2,...,W_k)$, where: 
    \begin{itemize}
        \item $W_1 = (\mathcal{X},\mathcal{A})$;
        \item $W_{i+1}=(\mathcal{X}_{i+1},\mathcal{A}_{i+1})$ is obtained from $W_i=(\mathcal{X}_{i},\mathcal{A}_{i})$ by one algorithmic step, i.e. $\mathcal{X}_{i+1}$ and $\mathcal{A}_{i+1}$ are some in-between modifications of relational structures $\mathcal{X}_{i},\mathcal{A}_{i}$;
        \item $W_k$ has no solution.
    \end{itemize}

Assume that a theory of bounded arithmetic proves that with each step of the algorithm, a modified instance $W_{i+1}$ has a solution only if $W_{i}$ does, and that the algorithm terminates without a solution. Then the theory proves - by its level of bounded induction - that $\mathcal{X}$ is unsatisfiable and hence that $\neg HOM(\mathcal{X},\ddot{\mathcal{A}})$ is a tautology. Note that proving the converse is not essential for establishing the algorithm's soundness. Additionally, there is no need to demonstrate that the algorithm is well-defined. The detailed record of the algorithm's execution can encompass all required additional details.

In \cite{gaysin2023proof} we proved the soundness of Zhuk's algorithm in a new theory of bounded arithmetic, namely $V^1$ augmented with three universal algebra axiom schemes that reflect Theorems \ref{fhfhfhydj}, \ref{llldju67yr}, and \ref{fjfjhgd}. We will formalize these axiom schemes in Section \ref{===-06476357234609}. In this paper, we formalize the proof of these theorems in theory $W^1_1$. Together with the results in \cite{gaysin2023proof}, it implies the soundness of Zhuk's algorithm in $W^1_1$.

\section{Formalization of notions}\label{'a'a;sldfkjfjfuryt}

In this section, we shall formalize the notions of universal algebra that will be used to prove Theorems \ref{fhfhfhydj}, \ref{llldju67yr}, and \ref{fjfjhgd}.

Based on Theorem \ref{aakrrrfetsplgh}, we can restrict ourselves to languages with at most binary constraints. This simplifies the formalization for the reader, but does not change the way it is performed. We want to stress that all the results in the paper can be extrapolated to any other finite constraint language. 

The formalization of part of the notions was introduced in \cite{gaysin2023proof} and we refer the reader to that paper for the relations defining a special WNU operation $SwNU_m$, a polymorphism $Pol_{m,2}, Pol_{m,1}, Pol$, a congruence and a proper congruence $Cong_m$, $pCong_m$, an undirected cycle $CYCLE$, an undirected path $PATH$, and a $1$-consistent, cycle-consistent, linked, fragmented, and irreducible instance, i.e. $1C$, $CCInst$, $LinkedInst$, $FragmInst$, and $IRDInst$ respectively. We also refer the reader to \cite{gaysin2023proof} for the formalization of all consistency reductions, and the linear part of the algorithm.

We know of no way to prove that a formalization exists other than actually doing it. This leads to a quite formal (and occasionally tedious) text with long formulas. Writing the formulas explicitly allows us to see that their bounded quantifier complexity is what is claimed. 

\begin{notation} To simplify the notation, we will denote relations on numbers using capital letters and functions using lowercase letters. We sometimes omit arguments in relations that are implied and do not affect the content of the relation. We index elements of sets starting with $0$, while all indices not related to elements of sets (for example, a sequence of relations) start from $1$.
\end{notation}

\subsection{Auxiliary relations and functions}
For any two relations $R_1,R_2$ of the same arity, we will use standard denotations for $R_1\subseteq R_2,$ $ R_1\subsetneq R_2,$ $R_1=R_2$, and $ R_1\neq \emptyset$. We introduce the following relations and functions as in \cite{10.5555/1734064}, \cite{10.1007/978-3-540-30124-0_27}.

If $x,y \in \mathbb{N}$, we define the \emph{pairing function} $\langle x,y\rangle$ to be the following term
\begin{equation}
\langle x,y\rangle = \frac{(x + y)(x + y + 1)}{2} + y.
\end{equation}
Theory $V^1$ proves that the pairing function is a bijection from $\mathbb{N}\times\mathbb{N}$ to $\mathbb{N}$, and that $x,y \leq \langle x,y\rangle < (x+y+1)^2$ for all $x,y$.
Inductively, we get 
\begin{equation}
\langle x_1,x_2,...,x_k\rangle = \langle ...\langle \langle x_1, x_2\rangle, x_3 \rangle ,..., x_k \rangle,
\end{equation}
where $x_1,x_2,...,x_k \leq \langle x_1,x_2,...,x_k\rangle < (x_1+x_2+...+x_k+1)^{2^k}$. 
For any set $Z$, $m\geq 2$, we will abbreviate $Z(\langle x_1,...,x_m\rangle)$ by $Z(x_1,...,x_m)$. We use the pairing function to code functions by sets. We express that $Z$ is a function from sets $X_1,...,X_n$ to a set $Y$ by stating
$$
\forall x_1\in X_1,...,\forall x_n\in X_n
\exists !y\in Y \, Z(x_1,...,x_n,y).
$$
We will denote it as $Z(x_1,...,x_n) = y$. Using the pairing function, we can code relations of any fixed arity. Finite functions are usually represented by their digraphs. We can add pairing functions for second-order objects, such as $\langle X,Y \rangle$ and $\langle x,Y\rangle$, using the following natural definitions. 
\begin{equation}
  \langle X,Y\rangle = Z(i,a) \iff (i=0\wedge X(a))\vee(i=1\wedge Y(a)),
\end{equation}
and
\begin{equation}
  \langle x,Y\rangle = Z(i,a) \iff i=x\wedge Y(a).
\end{equation}
The string function $row(i, Z)$, or $Z_i$, representing the row $i$ of a binary array $Z$, has a bit-defining axiom:
  \begin{equation}
    Z_i(a)=row(i,Z)(a)\iff (a<|Z|\wedge Z(i,a)).
  \end{equation}
We can use $row$ to represent a tuple $Z_1,...,Z_k$ of strings by a single string $Z$. We use a similar idea to allow $Z$ coding a sequence $y_0,y_1,...$ of numbers. Now $y_i$ is the smallest element of $Z_i$, or $|Z|$ if $Z_i$ is empty. The number function $seq(i, Z)$ (also denoted by $z_i$) has the following defining axiom:
  \begin{equation}
    \begin{gathered}
       a=seq(i,Z)\iff (a<|Z|\wedge Z(i,a)\wedge \forall b<a,\neg Z(i,b))\vee\\
       \hspace {0pt}\vee (\forall b<|Z|,\neg Z(i,b)\wedge a=|Z|).
    \end{gathered}
  \end{equation}
For a third-order variable $\mathscr{X}$ define $\mathscr{X}^{[x]}(Y)\equiv \mathscr{X}(\langle x,Y\rangle)$ and $\mathscr{X}^{[X]}(Y)\equiv\mathscr{X}(\langle X,Y\rangle)$. This notation allows one to consider $\mathscr{X}$ as an array with rows
indexed by numbers or strings, where each row is a third-order object. Note that as opposed to a string-valued function $row(i,Z)$, this notation is just an abbreviation of the formula, not a class-valued function. However, if we can bound the size of all strings in a class we are interested in by some value $s$, then we can define a string-valued function $\tilde{row}(\cdot)$ analogous to $row(\cdot)$, 
\begin{equation}
\begin{gathered}
     \tilde{row}(i,\mathscr{X},s)=Y \iff (|Y| <s\wedge\mathscr{X}^{[i]}(Y)\wedge \forall Y'<Y\,\neg \mathscr{X}^{[i]}(Y'))\vee\\
     \hspace {0pt}\vee(\forall Y' <s\,\neg \mathscr{X}^{[i]}(Y')\wedge \forall a<s,\,\neg Y(a) \wedge |Y|=s \wedge Y(s-1)), 
\end{gathered}
\end{equation}
where $Y'<Y$ is string ordering relation (\ref{laalalskksskdjdjdjfh}). Thus, the function returns the minimum string (due to string ordering) $Y$ of length less than $s$ such that $\mathscr{X}^{[i]}(Y)$ or, if such a string does not exist, the string Y of length $s$ with the only element $s-1\in Y $. An analogous function $\tilde{row}(X,\mathscr{X},s)$ can be defined for string indexing. 

Given a set $X$, the \emph{census function} $\verb|#|X(n)$ for $X$ is a number function defined for $n\leq |X|$ such that $\verb|#|X(n)$ is the number of $x<n$, $x\in X$. Thus, $\texttt{\#}X(|X|)$ is the number of elements in $X$. The following relation says that $\verb|#|X$ is the census function for $X$:
\begin{equation}
    \begin{split}
        &\hspace{10pt}Census(X,\texttt{\#}X)\iff \texttt{\#}X\leq \langle |X|,|X|\rangle\wedge \texttt{\#}X(0)=0\wedge \forall x<|X|\\
        &(x\in X \rightarrow \texttt{\#}X(x+1)=\texttt{\#}X(x)+1\wedge x\notin X \rightarrow \texttt{\#}X(x+1)=\texttt{\#}X(x)).
    \end{split}
\end{equation}
It can be easily shown that $V^1$ proves that for any set $X$ there exists its census function. To get the maximum or minimum elements of the set $R$, we define functions $max$ and $min$ naturally:
\begin{equation}\label{ala9s7d5tdhgdfr}
  \begin{gathered}
     \hspace{0pt}max(R) = |R|-1, \\
     min(R)=x\iff \forall y< |R|,\, R(y)\rightarrow x\leq y.
  \end{gathered}
\end{equation}
We define the ordering relation for strings as follows:
\begin{equation}\label{laalalskksskdjdjdjfh}
  \begin{gathered}
     X\leq Y\iff X=Y\vee\big(|X|\leq |Y|\wedge \exists z\leq |Y|(Y(z)\wedge \neg X(z) \wedge \\
     \wedge\forall u\leq |Y|, z<u \rightarrow (X(u)\rightarrow Y(u))     )\big).
  \end{gathered}
\end{equation}
That is, we compare strings based on numbers they represent as binary coding (the greater the number, the greater the string). Finally, we give $\Sigma^{1,b}_0$ bit-definitions of the string functions $\emptyset$ (constant empty string) and $S(X)$ (successor):
\begin{equation}
    \begin{gathered}
         \emptyset(z)\iff z<0,
    \end{gathered}
\end{equation}
and 
\begin{equation}
    \begin{gathered}
         S(X)(i)\iff \big( i\leq |X|\wedge ( (X(i)\wedge \exists j<i,\neg X(j))   \vee      (\neg X(i)\wedge \forall j<i,X(j)))     \big).
    \end{gathered}
\end{equation}

\subsection{$\mathbb{A}-$Monster Set: objects we have in advance}
In this section, we describe the list of objects we will further refer to as objects given in advance. 
The algorithm works for any finite algebra with a weak near unanimity (WNU) term and uses the fact that this term and all the algebra properties are known. From now on, we fix the algebra $\mathbb{A}=(A,\Omega)$, fix $l$ to be its size, and suppose that the only basic operation of $\mathbb{A}$ is an idempotent special WNU $m$-ary operation $\Omega$. To be consistent with Zhuk's paper, we do not use bold font for subuniverses of $\mathbb{A}$. We encode the algebra $\mathbb{A}$ with a pair of sets $(A,\Omega)$, where $|A|=l$, $A(i)$ for every $i$, and $\Omega$ is a set of size $((m+1)l)^{2^{m+1}}$, while all subuniverses of $\mathbb{A}$ are encoded by subsets of $A$ closed under $\Omega$. Due to Theorem \ref{fjjduh87}, there is a corresponding relational structure $\mathcal{A}$ such that $Clone(\mathbb{A}) = Pol(\mathcal{A})$, and we fix the notation $\mathcal{A}=(A,\Gamma_\mathcal{A})$, where the encoding of $\Gamma_{\mathcal{A}}$ is explained below.

Let $Sound(\mathbb{A})$ denote the soundness of Zhuk's algorithm for algebra $\mathbb{A}$, i.e. the formalization that if the algorithm rejects an instance, then the instance has no solution. In theory $T$ we can consider proving not just $Sound(\mathbb{A})$ but more generally an implication of the form 
$$Cond(\mathbb{A})\implies Sound(\mathbb{A}),
$$
where $Cond(\mathbb{A})$ is any recursively enumerable property of algebra $\mathbb{A}$. It can be written as $$\exists Y Cond_0(\mathbb{A},Y),$$ where $Y$, in general, cannot be bounded (even recursively), and $Cond_0$ can be a second-order bounded formula. In our case, $Y$ is a list of various objects such as subuniverses, binary relations preserved by $\Omega$ on $A$ or any subuniverse $D$ of $\mathbb{A}$, ternary operations on $A$, isomorphisms from subalgebras $(D,\Omega)$ of $\mathbb{A}$ to products of finite fields, etc. together with $V^0$-proofs of
their various $\Sigma^{1,b}_0$-properties. The proofs are given simply by exhaustive searching, unwinding all quantifiers. Therefore, the size
of the monster list $Y$ may be huge, in particular exponential, in the size of $\mathbb{A}$, but in general it does not
matter: whatever function of $l$ it is, it is a constant for fixed $l$. 

Note that if $W$ witnesses $Cond(\mathbb{A})$, we can prove $Cond_0(\mathbb{A},W)$ in $V^0$ 
(a constant size proof) and apply modus ponens to the implication above to deduce $Sound(\mathbb{A})$, which is what we really want to prove in $T$. This argument applies whenever $T$ contains $V^0$, which is true in our case. 

The use of this can be illustrated as follows. Assume $P(D)$ and $Q(D)$ are two bounded properties of a subuniverse $D$ of $\mathbb{A}$ and assume that in the monster set $Y$ we have two lists of all subuniverses together with proofs
that they do or do not satisfy $P$ and $Q$ respectively. A universal statement
$$\forall D, subAlgebra(D,A), P(D) \rightarrow Q(D)$$
can then be simply proved by going through $Y$ and checking that every $D$ in the list of those satisfying $P$ is also in the list of those
satisfying
$Q$; this uses a composition of proofs listed in $Y$. Another example of use is the following. The properties of $Z$ that may involve
second-order universal bounded quantifiers, as, for example, in
$$\forall D, subAlgebra(D,A),\, \phi(Z,D)$$
with $\phi \in \Sigma^{1,b}_0$, can be rewritten as $\Sigma^{1,b}_0$-formulas: replace
the universal quantifier by a large (but constant size) conjunction over all subalgebras of $\mathbb{A}$ as listed in the monster set $Y$.

This allows us to use well-known facts from universal algebra, as well as facts proved by Zhuk in \cite{zhuk2020proof}, without proving them in a theory of bounded arithmetic when it comes to the formalization of objects related exclusively to algebra $\mathbb{A}$. For example, we will not prove that any PC congruence $\sigma$ on $\mathbb{A}$ is maximal or that polynomially complete algebra $\mathbb{A}/\sigma$ is simple. Although we believe that all the properties of different objects on $\mathbb{A}$ needed in the argument can be proved in $\Sigma^{1,b}_1$-reasoning even with $\mathbb{A}$ variable, it is not necessary for our purpose. On the contrary, we shall prove any property related to an input structure since this structure is variable. 

We further list all the given in advance objects related to $\mathbb{A}$ we will use in the formalization, so-called $\mathbb{A}$-Monster set. All of them will be defined in detail in the corresponding sections. 

\begin{itemize}
  \item All subuniverses of $\mathbb{A}$ and any of its subuniverse $D$, the lists $\Gamma^1_{\mathcal{A}}$, $\Gamma^1_{\mathcal{D}}$;
  
  \item All binary relations on $A$ and any of its subuniverse $D$, compatible with $\Omega$, the lists $\Gamma^2_{\mathcal{A}}$, $\Gamma^2_{\mathcal{D}}$;  
  
  \item All congruences $\sigma$ on $A$ and any of its subuniverse $D$, the lists $\Sigma_\mathcal{A}$, $\Sigma_{\mathcal{D}}$;
  
  \item All factor sets for congruences $\sigma$ on $A$ and any of its subuniverse $D$, $A/\sigma$ and $D/\sigma$, and all operations $\Omega/\sigma$, the lists $\mathrm{A}_{\mathcal{A}}(i,A/\Sigma_{\mathcal{A},i}, \Omega/\Sigma_{\mathcal{A},i})$, $\mathrm{A}_{\mathcal{D}}(i,D/\Sigma_{\mathcal{D},i}, \Omega/\Sigma_{\mathcal{D},i})$;
  
   \item All maximal congruences on $A$ and any of its subuniverse $D$, the lists $\Sigma^{max}_\mathcal{A}$, $\Sigma^{max}_\mathcal{D}$;
   
   \item For all congruences $\sigma$ on $A$ and any of its subuniverse $D$, the lists of all unary and binary quotient relations on $A$ and $D$, compatible with $\Omega/\sigma$. We will denote the lists by $\Gamma_{\mathcal{A}}/\sigma$, $\Gamma_{\mathcal{D}}/\sigma$.

   \item The sets of all binary and ternary polymorphisms on $A$ and any of its subuniverse $D$, the lists $\Pi^2_{\mathcal{A}},\Pi^3_{\mathcal{D}}$;
   
   \item For all congruences $\sigma$ on $A$ and any of its subuniverse $D$, the sets of all binary and ternary polymorphisms on $A/\sigma$ and $D/\sigma$, the lists $\Pi^2_{\mathcal{A}/\sigma},\Pi^3_{\mathcal{D}/\sigma}$;
   
   \item For all congruences $\sigma$ on $A$ and any of its subuniverse $D$, the sets of all maps $H$ from $A/\sigma$ to $Z_{p_0}$, all maps $H$ from $A/\sigma$ to $Z_{p_0}\times Z_{p_1}$,..., all maps $H$ from $A/\sigma$ to $Z_{p_0}\times Z_{p_1}\times...\times Z_{p_{s-1}}$, for $s=log_2l$ and any prime $p_0,...,p_{s-1}$, $p_0\cdot...\cdot p_{s-1}\leq l$. We will denote these lists by $\mathrm{M}_{A,\sigma,p_0,...,p_{t-1}}$, $\mathrm{M}_{D,\sigma,p_0,...,p_{t-1}}.$
   
   \item The set of all linear congruences on $A$ and any of its subuniverse $D$, the lists $\Sigma^{lin}_{\mathcal{A}}$, $\Sigma^{lin}_{\mathcal{D}}$;
   
   \item For any subuniverse $C$ of $A$ and any of its subuniverse $D$, all sets of the form $X=\{\{a\}\times C,C\times \{a\}\}$ for all $a\in A\backslash C$. We denote the lists by $\mathrm{X}_{\mathcal{A}}$, $\mathrm{X}_{\mathcal{D}}$;
   
   \item The set of all PC congruences on $A$ and any of its subuniverse $D$, the lists $\Sigma^{PC}_{\mathcal{A}}$, $\Sigma^{PC}_{\mathcal{D}}$.

     \item For all congruences $\theta$ on $A$ and any of its subuniverse $D$, and for all PC congruences $\sigma_0,...,\sigma_{s-1}$ on $A$ and any of its subuniverse $D$, the sets of all maps $H$ from $A/\theta$ to $A/\sigma_{j_0}$, all maps $H$ from $A/\theta$ to $A/\sigma_{j_0}\times A/\sigma_{j_1}$,..., all maps $H$ from $A/\theta$ to $A/\sigma_{j_0}\times A/\sigma_{j_1}\times...\times A/\sigma_{j_{s-1}}$, for $s=log_2l$. We denote these lists by $\mathrm{M}_{A,\theta,\sigma_{j_0},\sigma_{j_1},...,\sigma_{j_{t-1}}}$, $\mathrm{M}_{D,\theta,\sigma_{j_0},\sigma_{j_1},...,\sigma_{j_{t-1}}}$.
     
   \item For all subuniverses $D_i, D_j$ of $A$, all congruences $\sigma_i,\sigma_j$ on $D_i, D_j$, the set of all bridges from $\sigma_i$ to $\sigma_j$, the set of all reflexive bridges and the set of all optimal bridges, the lists $\Xi_{\sigma_i,\sigma_j}$, $\Xi_{\sigma_i,\sigma_j}^{\leftrightarrow}$ and $\Xi_{\sigma_i,\sigma_j}^{opt}$.
\end{itemize}
Due to the definitions of all these sets (given in the corresponding sections), they may be empty. 
\begin{notation}
In formulas, we use notation $\bigwedge_{\Sigma^{max}_{\mathcal{A},i}}$ or $\bigvee_{\Sigma_{\mathcal{D},i}}$ meaning $\bigwedge_{\Sigma^{max}_{\mathcal{A},i}\neq \emptyset}$ or $\bigvee_{\Sigma_{\mathcal{D},i}\neq \emptyset}$: here we consider conjunction over all maximal congruences on $A$ or disjunction over all congruences on its subuniverse $D$. Sometimes, we also write $\bigwedge_{\sigma\in \Sigma^{max}_{\mathcal{A}}}$ or $\bigvee_{B \in \Gamma^1_{\mathcal{A}}}\bigvee_{T\in \Pi^2_{\mathcal{A}}}$ with the same meaning. When needed for better clarity, we use the explicit notation $\exists j<2^{l^2},... \Sigma^{PC}_{{\mathcal{A}},j}...$.
\end{notation}

\subsection{Encoding directed graphs and CSP instances} 
We encode a CSP instance on relational structures with at most binary relations in the following way. 

\begin{definition}
A \emph{directed input graph} is a pair $\mathcal{X} = (V_\mathcal{X},E_\mathcal{X})$ with $V_\mathcal{X}(i)$ for all $i<V_\mathcal{X}=n$ and $E_\mathcal{X}(i,j)$ being a binary relation on $V_\mathcal{X}$ (there is an edge from $i$ to $j$). A \emph{target digraph with domains} is a pair of sets $\ddot{\mathcal{A}}=( V_{\ddot{\mathcal{A}}},E_{\ddot{\mathcal{A}}})$, where:
\begin{itemize}
  \item $V_{\ddot{\mathcal{A}}}< \langle n,l\rangle$ is the set corresponding to the superdomain; we denote set $V_{\ddot{\mathcal{A}},i}$ by $D_i$ and call it domain subset for variable $x_i$; 
  \item $E_{\ddot{\mathcal{A}}}<\langle\langle n,l\rangle,\langle n,l\rangle \rangle$ is the set encoding that there is an edge $(a,b)$ between $D_i$ and $D_j$: 
\begin{equation}
  \begin{gathered}
     E_{\ddot{\mathcal{A}}}(u,v)\rightarrow \exists i,j<n\,\exists a,b<l\,\, u = \langle i,a\rangle\wedge v = \langle j,b\rangle\wedge\\
     \hspace {0pt} D_i(a)\wedge D_j(b).
  \end{gathered}
\end{equation}
\end{itemize}
Sometimes we consider set $D=\{D_0,...,D_{n-1}\}$. We use the notation $E^{ij}_{\ddot{\mathcal{A}}}(a,b)$ instead of $E_{\ddot{\mathcal{A}}}(\langle i,a \rangle, \langle j,b\rangle)$ for simplicity. We will denote a pair of sets $\Theta =(\mathcal{X},\ddot{\mathcal{A}})$, satisfying all the above conditions, by DG$(\Theta)$, and we will call $\Theta$ an instance. This representation allows us to construct a homomorphism from $\mathcal{X}$ to $\ddot{\mathcal{A}}$ with respect to different relations $E^{ij}_{\ddot{\mathcal{A}}}$ and different domains for all vertices $x_1,...,x_n$.
\end{definition}

\begin{definition}
A pair of sets $\Theta=(\mathcal{X},\ddot{\mathcal{A}})$ is a CSP instance on $n$ domains over constraint language $\Gamma_{\mathcal{A}}$ if 
\begin{equation}
  \begin{gathered}
     Inst(\Theta,\Gamma_{\mathcal{A}})\iff DG(\Theta)\wedge \forall i<n, |D_i|=l\wedge\\
     \hspace {0pt}\wedge \forall i,j<n,a,b<l, \exists s <|\Gamma_{\mathcal{A}}|, E_{\ddot{\mathcal{A}}}(\langle i,a \rangle,\langle j,b \rangle)\leftrightarrow \Gamma_{\mathcal{A}}^2(s,a,b)\wedge\\
     \hspace {0pt}\wedge\forall i<n,a<l,\exists s<|\Gamma_{\mathcal{A}}|, D_i(a)\leftrightarrow \Gamma_{\mathcal{A}}^1(s,a).
  \end{gathered}
\end{equation}

\end{definition}
When considering the direct product $D_0\times...\times D_{n-1}$, we can refer to it as a set of solutions to a CSP instance $\Theta_{null}=(\mathcal{X}_{null}, \ddot{\mathcal{A}}_{null})$, where 
\begin{itemize}
  \item $V_{\mathcal{X}_{null}}=n$ and for all $i<n, V_{\mathcal{X}_{null}}(i)$;
  \item for all $i,j<n$, $\neg E_{\mathcal{X}_{null}}(i,j)$ (i.e. the instance digraph $\mathcal{X}_{null}$ has no edges at all);
  \item for all $a<l$, $V_{\ddot{\mathcal{A}}_{null}}(i,a)\iff D_{i}(a)$;
  \item for all $a,b<l$, for all $i,j<n$, $\neg E^{ij}_{\ddot{\mathcal{A}}_{null}}(a,b)$ (i.e. the target digraph $\ddot{\mathcal{A}}_{null}$ has no edges at all).
\end{itemize}
We will denote a pair of sets $\Theta_{null}=(\mathcal{X}_{null}, \ddot{\mathcal{A}}_{null})$ satisfying all the above conditions by $DG_{null}(\Theta_{null})$. Since as domains we consider only subuniverses $D_i$ of $\mathbb{A}=(A,\Omega)$, $\Theta_{null}$ is also a CSP instance over constraint language $\Gamma_{\mathcal{A}}$. 

Sometimes, we will work with so-called factorized instances, where we factorize all domains $D_i$ by congruences $\sigma_i$.

\begin{definition}
  A pair of sets $\Theta'=(\mathcal{X}',\ddot{\mathcal{A}}')$ is a factorized CSP instance by a list of $n$ congruences $\Sigma$ on $n$ domains from a CSP instance $\Theta$ over constraint language $\Gamma_{\mathcal{A}}$ if
  \begin{equation}
\begin{gathered}
   \hspace {0pt}FInst(\Theta',\Sigma,\Theta, \Gamma_{\mathcal{A}}) \iff Inst(\Theta,\Gamma_{\mathcal{A}})\wedge \mathcal{X}=\mathcal{X'}\wedge \\
   \wedge \forall i<n, Cong_m(D_i,\Omega,\Sigma_i) \wedge\forall a,b \in D_i,\, (\Sigma_i( a,b )\wedge (a<b) \rightarrow \neg D'_i (b))\wedge\\
 \hspace{0pt}\wedge  (\forall a\in D_i(\forall a'\in D_i,\, \Sigma_i( a,a') \rightarrow a\leq a') \rightarrow D_i' (a))\wedge\\  
   \wedge \forall i,j<n,\, E^{ij}_{\ddot{\mathcal{A}}'}(a,b)\leftrightarrow D'_i(a)\wedge D_j'(b)
    \wedge (\exists c,d<l,\,\,\Sigma_{i}(a,c)\wedge \Sigma_{j}(b,d)\wedge E^{ij}_{\ddot{\mathcal{A}}}(c,d)),
\end{gathered}
  \end{equation}
where the second and the third lines in the formula ensure that every $\Sigma_i$ is a congruence on $D_i$, and $D_i'$ is the factor set $D_i/\Sigma_i$. Each block of a factor set is represented by its minimum element. 
\end{definition}
By $MAP(X,x,Y,y, H)$ we denote the relation expressing that $H$ is a map from a set $X$ of length $x$ to a set $Y$ of length $y$. Its definition is straightforward.

\begin{definition}[Homomorphism from digraph $\mathcal{X}$ to digraph with domains $\ddot{\mathcal{A}}$] 
A map $H$ is a \emph{ homomorphism between the input digraph} $\mathcal{X}=(V_{\mathcal{X}},E_{\mathcal{X}})$, $V_{\mathcal{X}}=n$ \emph{and the target digraph with domains} $\ddot{\mathcal{A}}=(V_{\ddot{\mathcal{A}}},E_{\ddot{\mathcal{A}}})$, $V_\mathcal{A}<\langle n,l\rangle$ if $H$ is a homomorphism from $\mathcal{X}$ to $\ddot{\mathcal{A}}$ sending each $i\in V_\mathcal{X}$ to domain $D_i$ in $V_{\ddot{\mathcal{A}}}$.
The statement that there exists such $H$ can be expressed by the following $\Sigma^{1,b}_{1}$-formula.
\begin{equation}\label{HOMINSTAG}
 \begin{gathered}
 \hspace {0pt}\ddot{HOM}(\mathcal{X},\ddot{\mathcal{A}}) \iff \exists H < \langle n,\langle n,l\rangle\rangle \big(MAP(V_\mathcal{X}, n,V_{\ddot{\mathcal{A}}},\langle n,l\rangle,H)\wedge \\
 \hspace{0pt}(\forall i<n,s<\langle n,l\rangle \,\, H(i)=s\rightarrow \exists a<l, s=\langle i,a\rangle\wedge D_i(a))\wedge \\
 \hspace {0pt}\forall i_1,i_2<n,\forall j_1,j_2<\langle n,l\rangle\\
 \hspace {0pt}(E_{\mathcal{X}}(i_1,i_2)\wedge H(i_1)=j_1\wedge H(i_2)=j_2\rightarrow E_{\ddot{\mathcal{A}}}(j_1,j_2)).
 \end{gathered}
\end{equation}
\end{definition}

In addition to a homomorphism between two digraphs of different types, we will also need a classical homomorphism between digraphs of the same type. The existence of such a homomorphism between digraphs $\mathcal{G}$ and $\mathcal{H}$ with $V_{\mathcal{G}}<n$, $V_{\mathcal{G}}<m$ is again a $\Sigma^{1,b}_1$-formula.
\begin{equation}
 \begin{gathered}
 HOM(\mathcal{G},\mathcal{H}) \iff \exists H < \langle n,m\rangle \big(MAP(V_{\mathcal{G}},n,V_{\mathcal{H}},m,H)\wedge \\
 \hspace {0pt} \forall i_1,i_2<n, \forall j_1,j_2 < m \\ \hspace {0pt} (E_\mathcal{G}(i_1,i_2)\wedge H(i_1)=j_1 \wedge H(i_2)=j_2 \to E_\mathcal{H}(j_1,j_2))\big).
 \end{gathered}
\end{equation}
Further, for any instance $\Theta=(\mathcal{X},\ddot{\mathcal{A}})$ and a factorized instance $\Theta'=(\mathcal{X}',\ddot{\mathcal{A}}')$ by a list of $n$ congruences $\Sigma$ we can define a canonical homomorphism $H_c$ between the target digraph $\ddot{\mathcal{A}}$ and the factorized target digraph $\ddot{\mathcal{A}}'$  as follows: for every $u\in V_{\ddot{\mathcal{A}}}$, and every $v\in V_{\ddot{\mathcal{A}}'}$
$$
H_c(u,v)\iff \exists i<n,a,b<l,\,u = \langle i,a\rangle,\,v = \langle i,b\rangle\wedge \sigma(i,b,a) \wedge D_i'(b).
$$
That is, a vertex $a$  is sent to a vertex $b$ in $\ddot{\mathcal{A}}'$ in the factorized domain $D_i'$ if and only if $b\in D_i$, $b$ and $a$ are in the same congruence class under $\Sigma_i$, and $b$ is a represent of the class $a/\Sigma_i$ (the smallest element). It is straightforward to check that $H_c$ is indeed a homomorphism, and that there is a homomorphism from $\mathcal{X}$ to $\ddot{\mathcal{A}}'$.

\begin{notation}
Sometimes, we will write $\exists(\forall) H< \langle n,m\rangle, HOM(\mathcal{G},\mathcal{H}, H)$ and $\exists(\forall) H < \langle n,\langle n,l\rangle\rangle,$ $ \ddot{HOM}(\mathcal{X},\ddot{\mathcal{A}},H)$ to omit repetitions. Note that $HOM(\mathcal{G},\mathcal{H})$ and $\ddot{HOM}(\mathcal{X},\ddot{\mathcal{A}})$ are $\Sigma^{1,b}_1$-formulas, while $HOM(\mathcal{G},\mathcal{H}, H)$ and $\ddot{HOM}(\mathcal{X},\ddot{\mathcal{A}},H)$ are $\Sigma^{1,b}_0$.
\end{notation}

\subsection{Subalgebras and Solution sets to a CSP instance}

To define the direct and subdirect products of $k$ algebras for constant $k$, we first define a universe set for the product. For any sets $D_0,...,D_{k-1}$ of size bounded by $l$ we will denote by $D_0\times...\times D_{k-1}$ a $k$-ary set of the form 
\begin{equation}
D_0\times...\times D_{k-1}(a_0,...,a_{k-1}) \iff a_0\in D_0\wedge...\wedge a_{k-1}\in D_{k-1}.
\end{equation}
We define an $m$-ary operation $F:(D_0\times...\times D_{k-1})^m\rightarrow D_0\times...\times D_{k-1}$ on a set $D_0\times...\times D_{k-1}$. Denote $\langle a_i^0,...,a_i^{k-1}\rangle$ by $\bar{a}^k_i$, then
\begin{equation}
 \begin{gathered}
 \hspace {0pt}OP_m(F, D_0\times...\times D_{k-1}) \iff \forall \bar{a}^k_1,...,\bar{a}^k_m \in D_0\times...\times D_{k-1},\\
 \exists \bar{b}^k\in D_0\times...\times D_{k-1},\,F(\bar{a}^k_1,...,\bar{a}^k_m ,\bar{b}^k)\wedge \forall \bar{b}^k_1,\bar{b}^k_2 \in A_0\times...\times D_{k-1},
\\
 \hspace {0pt}(F(\bar{a}^k_1,...,\bar{a}^k_m,\bar{b}^k_1)\wedge F(\bar{a}^k_1,...,\bar{a}^k_m,\bar{b}^k_2) \to \bar{b}^k_1=\bar{b}^k_2).
 \end{gathered}
\end{equation}
In the same fashion, we can formalize a special idempotent WNU operation $\Omega$ on the set $D_0\times...\times D_{k-1}$. Further, we define a subuniverse $R$ of algebra $(D_0\times...\times D_{k-1},\Omega)$ as follows:
\begin{equation}
\begin{gathered}
  subTA(R,D_0\times...\times D_{k-1},\Omega)\iff |R|=|D_0\times...\times D_{k-1}|\wedge \\
  \hspace {0pt}\forall i<(kl)^{2^k}, R(i)\rightarrow D_0\times...\times D_{k-1}(i)\wedge SwNU_m(\Omega,R).
 \end{gathered}
\end{equation}
We say that an algebra $\mathbb{D}=(D,\Omega)$ is a direct product of $k$ algebras $(D_0,\Omega_0)$, ..., $(D_{k-1},\Omega_{k-1})$ of the same type (with $m$-ary operations) if
\begin{equation}
  \begin{gathered}
   \hspace {0pt}DP_{m,k}(D,\Omega,D_0,\Omega_0,...,D_{k-1},\Omega_{k-1})\iff D=D_0\times...\times D_{k-1}\wedge \\ 
    \hspace{0pt}\wedge \forall a_1^{0},...,a^{0}_m\in D_0,...,\forall a_1^{k-1},...,a^{k-1}_m \in D_{k-1} \exists b^{0} \in D_0,...,\exists b^{k-1}\in D_{k-1}\\
 \hspace {0pt}\Omega(\bar{a}^k_1,...,\bar{a}^k_m,\langle b^{0},...,b^{k-1}\rangle)\wedge \Omega_0(a^{0}_{1},...,a^{0}_{m},b^{0})\wedge...\wedge \Omega_{k-1}(a^{k-1}_{1},...,a^{k-1}_{m},b^{k-1}).
  \end{gathered}
\end{equation}
A subdirect product $(R,\Omega)$ of $k$ algebras $(D_0,\Omega_0)$, ...,$(D_{k-1},\Omega_{k-1})$ is encoded as follows:
\begin{equation}
\begin{gathered}
   subDP_{m,k}(R, \Omega, D_0,\Omega_0,...,D_{k-1},\Omega_{k-1})\iff  subTA(R,D_0\times...\times D_{k-1},\Omega) \\ 
   \hspace {0pt}\wedge DP_{m,k}(D_0\times...\times D_{k-1},\Omega,D_0,\Omega_0,...,D_{k-1},\Omega_{k-1})\wedge \\
   \hspace{0pt}\wedge \bigwedge_{i<k} \forall a_{i}\in D_i, \exists a_{0}\in D_0,...,\exists a_{i}\in D_{i-1},\exists a_{i+1}\in D_{i+1},...,\exists a_{k-1}\in A_{k-1},\\
   \hspace {0pt}R(a_1,...,a_{i-1},a_i,a_{i+1},...,a_{k}).
\end{gathered}
\end{equation}

Note that the set of solutions to any instance of CSP($\Gamma$) can be viewed as a subuniverse of power of $B$. By Theorem \ref{skskjdi}, every $n$-ary relation $R$ on $B^n$ preserved by all polymorphisms of $Pol(\Gamma)$ can be $pp$-defined from $\Gamma$. Thus, it is equal to some projection of the set of solutions to some instance of CSP($\Gamma$). However, the instance itself can be exponential in $n$ (see the construction in \cite{barto_et_al}). Furthermore, we cannot define a subalgebra $R$ of $B^n$ as an $n$-ary set $R(b_1,...,b_{n})$ as it requires $(ln)^{2^n}$ length. We shall stress that since most of the theorems in the general part of Zhuk's algorithm that are proved for any subalgebras were used in the algorithm only for solution sets \cite{zhuk2020proof}, whenever possible, we restrict ourselves upward to solution sets to some CSP instances over $\Gamma_{\mathcal{A}}$.

For definitions, we use the $\Sigma^{\mathscr{B}}_0$-3COMP axiom scheme. We can consider any $n$-ary relation $R$ on $A^n$ as a third-order object – a class of maps $\mathscr{R}$ from $[n]$ to $[A,A,...,A]$. Analogously, any $R\leq D_0\times...\times D_{n-1}$ is a class of maps from $[n]$ to $[D_0,D_1,...,D_{n-1}]$. In terms of digraphs, 
\begin{equation}
\mathscr{R}(H)\implies MAP(V_\mathcal{X}, n,V_{\ddot{\mathcal{A}}},\langle n,l\rangle,H), 
\end{equation}\label{kqwhf;48r;3ihert}
which is $\Sigma^{\mathscr{B}}_0$-formula, and
\begin{equation}\label{kqwhf;48r;3ihert}
  \mathscr{D}_0\times...\times \mathscr{D}_{n-1}(H)\iff MAP(V_\mathcal{X}, n,V_{\ddot{\mathcal{A}}},\langle n,l\rangle,H),
\end{equation}
which is already $\Sigma^{1,b}_0$-formula. To make from $ \mathscr{D}_0\times...\times \mathscr{D}_{n-1}$ an algebra, we define a third-order object representing a basic $m$-ary function $\mathscr{F}_{\Omega_0,...,\Omega_{n-1}}$ on $\mathscr{D}_0\times...\times \mathscr{D}_{n-1}$ (again, $\Sigma^{1,b}_0$-formula):
\begin{equation}\label{l;alsd'jqf;hd}
\begin{gathered}
   \mathscr{F}_{\Omega_0,...,\Omega_{n-1}}(H_1,...,H_m,H) \iff \forall i<n,\exists a^i_1,...,a^i_m,a^i<l,\,\Omega_i(a^i_1,...,a^i_m)=a^i\wedge\\
       \hspace{0pt}\wedge H_1(i)=\langle i,a^i_1\rangle\wedge...\wedge H_m(i)=\langle i,a^i_m\rangle\wedge H(i)=\langle i,a^i\rangle.
\end{gathered}
\end{equation}
If for all $D_i$ there is the same operation $\Omega$, we denote this class by $\mathscr{F}_{\Omega}$. Let us consider in this section all $D_i$ being subalgebras of $\mathbb{A}=(A,\Omega).$ For subuniverses, we require $\mathscr{R}(H)$ to be closed under $\mathscr{F}_{\Omega}$ by the definition. In fact, we can express this requirement remaining in the second-order setting. For any $n$, we introduce a string function $\omega$ of $m$ maps from $[n]$ to $[A,A,...,A]$ returning a new map $H$ by its bit-definition for all $i<n, a<l$:
\begin{equation}\label{alalkdhjhfr}
  \begin{gathered}
     \omega(H_1,...,H_m)(\langle i,\langle i,a\rangle\rangle)\iff \exists a_1,...,a_m<l,\,\Omega(a_1,...,a_m)=a\wedge\\
       \hspace {0pt}\wedge H_1(i)=\langle i,a_1\rangle\wedge...\wedge H_m(i)=\langle i,a_m\rangle.
  \end{gathered}
\end{equation}
Note that $\omega$ is an actual function, not a set of sets, and it is based on the fixed set $\Omega$. In the same fashion, we can introduce a string function $usepol_{k}$ for any $k$-ary polymorphism $F$ on $\mathbb{A}$ by its bit-definition for all $i<n, a<l$:
\begin{equation}
  \begin{gathered}
     usepol_{k}(F,H_1,...,H_k)(\langle i,\langle i,a\rangle\rangle)\iff \exists a_1,...,a_k<l,\,F(a_1,...,a_k)=a\wedge\\
       \hspace {0pt}\wedge H_1(i)=\langle i,a_1\rangle\wedge...\wedge H_k(i)=\langle i,a_k\rangle.
  \end{gathered}
\end{equation}
We will denote such functions restricted by $\mathscr{R}$ by $\omega^{\mathscr{R}}$ and $usepol_{k}^{\mathscr{R}}$. Thus, for any class $\mathscr{R}$ representing subalgebra on $A^n$, and any maps $H_1,...,H_m$:
\begin{equation}
  \begin{gathered}
     \mathscr{R}(H_1)\wedge...\wedge \mathscr{R}(H_m)\implies \mathscr{R}(\omega(H_1,...,H_m)).
  \end{gathered}
\end{equation}
To consider a projection of subalgebra $\mathscr{R}$ to some subset of coordinates $i_1,...,i_s, s<n$, we introduce a partial map from $[n]$ to $(D_0,...,D_{n-1})$:
\begin{equation}\label{akllsoei657}
\begin{gathered}
   \hspace{0pt}\mathscr{R}^{i_1,...,i_s}(H)\iff \bigwedge_{i\in \{i_1,...,i_s\}} \exists ! a\in D_i,\, H(i)=\langle i,a\rangle \wedge\\
   \hspace{0pt}\bigwedge_{i\notin \{i_1,...,i_s\}} \forall a<l,\,\neg H(i) = \langle i,a \rangle \wedge\\
   \exists H'<\langle n,\langle n,l \rangle \rangle,\,\mathscr{R}(H')\wedge \bigwedge_{i\in \{i_1,...,i_s\}} H(i)=H'(i).
\end{gathered}
\end{equation}
There are $2^n=\sum^{n}_{s=0}\binom{n}{s}$ such different classes, but we do not need to define them all; we will define the required occasionally. Note that $\omega$ and $usepol_{k}$ are well-defined for such partial maps for $i\in \{i_1,...,i_s\}$.

The solution set to the instance $\Theta=(\mathcal{X},\ddot{\mathcal{A}})$ of CSP over $\Gamma_{\mathcal{A}}$, $\ddot{\mathcal{A}}=( D_0,...,D_{n-1},E_{\ddot{\mathcal{A}}})$, is a set of homomorphisms $\{\mathcal{X}\to\ddot{\mathcal{A}}\}=\{H_1,H_2,...,H_s\}$. Let us denote it by $\mathscr{R}_{\Theta}$. Note that the definition is a $\Sigma^{1,b}_0$-formula.
\begin{equation}
\mathscr{R}_{\Theta}(H)\iff \ddot{HOM}( \mathcal{X},\ddot{\mathcal{A}},H).
\end{equation}
In these terms, the product $\mathscr{D}_0\times...\times \mathscr{D}_{n-1}$ can be considered as $\mathscr{R}_{\Theta_{null}}$. The projection $ \mathscr{R}_{\Theta}^{i_1,...,i_s}$ of the solution set to some subset of coordinates is defined analogously to (\ref{akllsoei657}), we call $H\in \mathscr{R}_{\Theta}^{i_1,...,i_s}$ a partial homomorphism from $\mathcal{X}$ to $\ddot{\mathcal{A}}$. Note that while $\mathscr{R}_{\Theta}$ is a $\Sigma^{1,b}_0$-formula, $\mathscr{R}_{\Theta}^{i_1,...,i_s}$ is a $\Sigma^{1,b}_1$-formula.

\begin{lemmach}\label{lslsdkleijf}
  For any $k>0$, $V^1$ proves that for any CSP instance $\Theta$, any $k$-ary operation $F\in Pol_k(F,A,\Gamma_{\mathcal{A}})$, and any $k$ homomorphisms $H_1,...,H_k$ from $\mathcal{X}$ to $\ddot{\mathcal{A}}$ (and for any $i\in \{i_1,...,i_s\}$) a map $H=usepol_{n,k}(F,H_1,...,H_k)$ 
  is again a homomorphism from $\mathcal{X}$ to $\ddot{\mathcal{A}}$ (a partial homomorphism from $\mathcal{X}$ to $\ddot{\mathcal{A}}$). 
\end{lemmach}
\begin{proof}
  Recall that any polymorphism preserves all relations from $\Gamma_{\mathcal{A}}$. Every relation $E_{\mathcal{A}}^{ij}$ (set of edges from $D_i$ to $D_j$) is a subalgebra of $D_i\times D_j$ (since it is compatible with $\Omega$). The proof then goes by contradiction: suppose that there is an edge $(x_i,x_j)\in E_{\mathcal{X}}$ (with $i,j\in \{i_1,...,i_s\}$) such that $H$ does not map it to an edge in $\ddot{\mathcal{A}}$. Since all homomorphisms (partial for $i\in \{i_1,...,i_s\}$) $H_1,...,H_k$ map $(x_i,x_j)$ to some edge in $E^{ij}_{\ddot{\mathcal{A}}}$, it is possible only if $F$ does not preserve the relation $E^{ij}_{\ddot{\mathcal{A}}}$. 
\end{proof}

\begin{corollary}
  $V^1$ proves that a solution set $\mathscr{R}_{\Theta}$ and a projection $ \mathscr{R}_{\Theta}^{i_1,...,i_s}$ for any subset of coordinates $\{i_1,...,i_s\}$ for a CSP instance $\Theta=(\mathcal{X},\ddot{\mathcal{A}})$ on $n$ variables are subuniverses of $A^n$ and $A^{ \{0,1,...,n\}\backslash \{i_1,...,i_s\}}$ respectively. 
\end{corollary}

We say that subuniverse $\mathscr{R}$ is subdirect if
\begin{equation}\label{HOMINSTAG}
 \begin{gathered}
 subDSSInst(\mathscr{R})\iff \forall i<n\forall a\in D_i, \exists H < \langle n,\langle n,l\rangle\rangle,H\in \mathscr{R}\wedge H(i)=\langle i,a \rangle.
 \end{gathered}
\end{equation}
Note that this is a $\Sigma^{\mathscr{B}}_0$-formula. If we consider solution set $\mathscr{R}_{\Theta}$, then the definition becomes a $\Sigma^{1,b}_1$-formula:
\begin{equation}\label{HOrhrhrhrrhAG}
 \begin{gathered}
 subDSSInst(\mathscr{R}_{\Theta})\iff \forall i<n,\forall a\in D_i, \\
 \exists H < \langle n,\langle n,l\rangle\rangle,\ddot{HOM}(\mathcal{X},\ddot{\mathcal{A}},H)\wedge H(i)=\langle i,a \rangle.
 \end{gathered}
\end{equation}
Whenever possible, we refer to $\mathscr{R}_{\Theta}$ as a set of homomorphisms $\{\mathcal{X}\to\ddot{\mathcal{A}},\}=\{H_1,H_2,...,$ $H_s\}$, i.e. we use $\forall H\leq \langle n,\langle n,i\rangle\rangle,\, \ddot{HOM}(\mathcal{X},\ddot{\mathcal{A}},H)$ since this allows us to avoid third-sorted objects.

\begin{remark}
Note that no subalgebra of $A^n$ is a solution set to some CSP instance over $\Gamma_{\mathcal{A}}$. It is an instance of a larger language $\Gamma'$ containing $\Gamma$ and closed under $pp$-definitions.
\end{remark}

\subsection{Congruence and congruence on products}\label{a;kjdgh;qeirgy;gh;w}
A maximal congruence on an algebra $(D,\Omega)$ can be defined by the following $\Sigma^{1,b}_0$-formula:
\begin{equation}
  \begin{gathered}
     maxCong_m(D,\Omega,\sigma)\iff Cong_m(D,\Omega,\sigma) \wedge \exists a,b\in D,\, \neg\sigma( a,b )\wedge\\
     \hspace{0pt}\wedge[\bigwedge_{\Sigma_{\mathcal{D},i}} ( \exists a,b\in D,\, \neg\Sigma_{\mathcal{D},i}( a,b )\rightarrow \exists a,b\in D,\,\sigma(a,b)\wedge \neg \Sigma_{\mathcal{D},i}(a,b))].
  \end{gathered}
\end{equation}
The general definition of a maximal congruence $\sigma$ for any (not fixed) algebra $\mathbb{B}$ of size $n$ is $\Pi^{1,b}_1$:
\begin{equation}
  \begin{gathered}
     maxCong_m(B,\Omega,\sigma)\iff Cong_m(B,\Omega,\sigma) \wedge \exists a,b\in B,\, \neg\sigma( a,b )\wedge\\
     \hspace {0pt}\wedge[\forall \sigma'<\langle n,n\rangle,\, (Cong_m(B,\Omega,\sigma')\wedge \exists a,b\in B,\, \neg\sigma'( a,b ))\rightarrow\\
     \hspace {0pt}\rightarrow \exists a,b\in B,\,\sigma(a,b)\wedge \neg \sigma'(a,b)].
  \end{gathered}
\end{equation}
Analogously, we can define a minimal congruence $\sigma$, by relation $minCong_m(D,\Omega,\sigma)$. Recall that each block of a factor set, denoted by $D/\sigma$, is represented by its minimum element (it exists by the $\Sigma^{1,b}_0$-MIN principle). Therefore, we also think of the factorized object $D/\sigma$ as a set of numbers. When we consider any congruence $\sigma$ on $D$, we do not need to claim the existence of sets $D/\sigma$ and $\Omega/\sigma$ - there is a simple algorithm to construct them, and the construction is unique. First, we define the following string function
\begin{equation}
  \begin{gathered}
     factorset(D,\sigma)(a)=D/\sigma(a)\iff a<|D|\wedge a\in D\wedge\\
     \hspace {0pt}\wedge (\forall a'\in D, \sigma(a,a')\rightarrow a\leq a').
  \end{gathered}
\end{equation}
To represent an element we define a number function $rep(a/\sigma,D,\sigma)$
\begin{equation}
  \begin{gathered}
     a=rep(a/\sigma,D,\sigma)\iff \sigma(a,a/\sigma)\wedge factorset(D,\sigma)(a).
  \end{gathered}
\end{equation}
Finally, we can define a string function returning $\Omega/\sigma$ using a bit-defining axiom:
\begin{equation}
  \begin{gathered}
     factor\omega(D,\Omega,\sigma)(b) = \Omega/\sigma(b)\iff \exists a_1...\exists a_m\exists c \in D/\sigma,\,b=\langle a_1,...,a_m,c\rangle\wedge\\
     \hspace {0pt}\exists a_1/\sigma...\exists a_m/\sigma\exists c/\sigma \in D,\, c=rep(c/\sigma,D,\sigma)\wedge\bigwedge_{i<m} a_i=rep(a_i/\sigma,D,\sigma)\wedge \\
 \hspace {0pt}\wedge\Omega(a_1/\sigma,...,a_{m}/\sigma, c/\sigma).
  \end{gathered}
\end{equation}
The following two claims follow straightforwardly from the definitions of congruence and WNU operation.
\begin{claimm}
Consider an algebra $\mathbb{D}=(D,\Omega_D)$, its subuniverse $B$ and a congruence $\sigma$ on $D$. Then $V^0$ proves that $\sigma$ restricted to $B$ is a congruence on $B$.
\end{claimm}

\begin{claimm}\label{alalskdu65f}
Consider an algebra $\mathbb{D}=(D,\Omega_D)$ with $\Omega$ being a special WNU operation, and a congruence $\sigma$ on $D$. Then $V^0$ proves that for all $a\in D$, a congruence block $[a]/\sigma$ is a subuniverse of $D$.
\end{claimm}

For any congruence $\sigma$ on algebra $\mathbb{D}=(D,\Omega)$, for factor algebra $\mathbb{D}/\sigma$ we will define the quotient set of relation $\Gamma_{\mathcal{D}}/\sigma$ as follows:
\begin{equation}\label{allalask75575764}
\begin{gathered}
 \hspace{0pt}\Gamma^{1}_{\mathcal{D}}/\sigma(j,a)\iff \forall a/\sigma\in D,\, Rep_m(a,a/\sigma,D/\sigma,D,\Omega,\sigma)\wedge\Gamma^{1}_{\mathcal{D}}(j,a/\sigma)\\
 \hspace {0pt}\Gamma^{2}_{\mathcal{D}}/\sigma(i,a,b)\iff \forall a/\sigma,b/\sigma\in D,\,\Gamma^{2}_{\mathcal{D}}(i,a/\sigma,b/\sigma)\wedge\\
 \hspace {0pt}\wedge Rep_m(a,a/\sigma,D/\sigma,D,\Omega,\sigma)\wedge Rep_m(b,b/\sigma,D/\sigma,D,\Omega,\sigma).
\end{gathered}
\end{equation}
Note that for some $i,j$, $\Gamma^{1}_{\mathcal{D},j}/\sigma$ and $\Gamma^{2}_{\mathcal{D},i}/\sigma$ are empty sets. We will use it in the definition of PC subuniverses. The formulas (\ref{allalask75575764}) follow from log-space reduction from CSP($\mathbb{D}/\sigma$) to CSP($\mathbb{D}$), see \cite{Brady2022NotesOC}. We want to stress it directly here, not to repeat it many times. The relation signatures of the structures corresponding to $D$ and $D/\sigma$ differ, and the relation $R\in \Gamma_{\mathcal{D}}/\sigma$ lifts to the relation $R'\in\Gamma_{\mathcal{D}}$ by the rule $\bar{a}\in R'\iff \bar{a}/\sigma\in R$. Thus, for any binary or unary relation preserved by $\Omega/\sigma$ on $D/\sigma$ its corresponding lifted relation is preserved by $\Omega$ on $D$. That is, we already have all such relations in $\Gamma_{\mathcal{A}}$. Moreover, for any binary relation $R$ on $D_i/\sigma_i\times D_j/\sigma_j$ preserved by $\Omega/\sigma=(\Omega/\sigma_i,\Omega/\sigma_j)$ its lifted relation on $D_i\times D_j$ is preserved by $\Omega$ under the same rule.

We define a binary relation $\mathscr{C_\sigma}$ on $D_0\times...\times D_{n-1}$ as a third-order object for all maps from $[n]$ to $(D_0,...,D_{n-1})$, $\mathscr{C_\sigma}(H_1,H_2)$. For $\mathscr{C_\sigma}$ being compatible with $\Omega$, we require that for any $H_1,...,H_m,$ $H_1',...,$ $H_m'$:
\begin{equation}\label{alksjdfgy7}
  \begin{gathered}
    \mathscr{C_\sigma}(H_1,H_1')\wedge...\wedge\mathscr{C_\sigma}(H_m,H_m')\implies \mathscr{C_\sigma}(\omega(H_1,...,H_m),\omega(,H_1',...,H_m')).
  \end{gathered}
\end{equation}
For $\mathscr{C_\sigma}$ being a congruence, we additionally require that for any three maps $H_1,H_2,H_3$,
\begin{equation}\label{skksidjh}
  \begin{gathered}
     \mathscr{C_\sigma}(H_1,H_1)\wedge (\mathscr{C_\sigma}(H_1,H_2)\leftrightarrow \mathscr{C_\sigma}(H_2,H_1))\wedge\\
     \hspace {0pt}\wedge (\mathscr{C_\sigma}(H_1,H_2)\wedge \mathscr{C_\sigma}(H_2,H_3)\rightarrow \mathscr{C_\sigma}(H_1,H_3)).
  \end{gathered}
\end{equation}
We can restrict $\mathscr{C_\sigma}$ to any subuniverse $\mathscr{R}$ (we will call $\mathscr{C}^{\mathscr{R}}_\sigma$ a congruence restricted to $\mathscr{R}$) by requiring for all $H,H'$,
\begin{equation}
  \mathscr{C}^{\mathscr{R}}_{\sigma}(H,H')\implies \mathscr{R}(H')\wedge \mathscr{R}(H').
\end{equation}

Now we return to second-order congruences and extend them to third-order objects. The next three relations are expressed by $\Sigma^{1,b}_0$-formulas. For any congruence $\sigma_i$ on $D_i$ we say that two maps $H_1,H_2$ are in the same equivalence block on $D_0\times...\times D_{n-1}$ if 
\begin{equation}\label{slslooeojduk}
  \begin{gathered}
 1EqClass(i,H_1,H_2,D_i,\sigma_i)\iff \forall a_{i_1},a_{i_2}<l,\\
     \hspace {0pt}H_1(i)=\langle i,a_{i_1}\rangle\wedge H_2(i)=\langle i,a_{i_2}\rangle\rightarrow \sigma_i(a_{i_1},a_{i_2}).
  \end{gathered}
\end{equation}
Then for any congruence $\sigma_i$ on $D_i$ we define an extended relation $\mathscr{C}_{\sigma_i^{ext}}$ as follows:
\begin{equation}\label{slslooeojduk}
  \begin{gathered}
     \mathscr{C}_{\sigma_i^{ext}}(H_1,H_2)\iff 1EqClass(i,H_1,H_2,D_i,\sigma_i).
  \end{gathered}
\end{equation}
Analogously, for any $\sigma_0,...,\sigma_{n-1}$ where each $\sigma_i$ is a congruence on $D_i$, we define a relation $\mathscr{C}_{\cap_{n}\sigma^{ext}_i}$ on $D_0\times...\times D_{n-1}$ as follows: 
\begin{equation}\label{laallaskskjd}
  \begin{gathered}
    \mathscr{C}_{\cap_{n}\sigma^{ext}_i}(H_1,H_2)\iff \forall i<n,\, 1EqClass(i,H_1,H_2,D_i,\sigma_i).
  \end{gathered}
\end{equation}
Notice that some congruences $\sigma_i$ can be $\nabla_{D_i}$ or $\Delta_{D_i}$. Obviously, for any three maps $H_1,H_2,H_3$ the relation $1EqClass$ is reflexive, symmetric, and transitive. Compatibility can be proved easily again. Consider $2m$ maps $H_1,...,H_m$ and $H_1',...,H_m'$ such that for every $j<m$, $1EqClass(i,H_j,H_j',D_i,\sigma_i)$. Then, due to defining equation (\ref{alalkdhjhfr}) of $\omega$, 
$$1EqClass(i,\omega(H_1,...,H_m),\omega(H_1',...,H_2'),D_i,\sigma_i).$$ 
Note that in (\ref{laallaskskjd}) we define a third-order object using only its second-order properties. Thus, we have proved the following claims.

\begin{claimm}
  Consider $D_0\times...\times D_{n-1}$, and binary relations $\sigma_0,...,\sigma_{n-1}$ where $\sigma_i$ is a congruence of $D_i$ for every $i$. Then $V^0$ proves that any $\mathscr{C}_{\sigma_i^{ext}}$ and $\mathscr{C}_{\cap_{n}\sigma^{ext}_i}$ are congruences on $\mathscr{D}_0\times...\times \mathscr{D}_{n-1}$.
\end{claimm}

\begin{claimm}
  Consider $\mathscr{R}_{\Theta}\leq D_0\times...\times D_{n-1}$, and binary relations $\sigma_0,...,\sigma_{n-1}$ where $\sigma_i$ is a congruence of $D_i$ for every $i$. Then $V^0$ proves that any $\mathscr{C}_{\sigma_i^{ext}}$ and $\mathscr{C}_{\cap_{n}\sigma^{ext}_i}$ restricted to $\mathscr{R}_{\Theta}$ are congruences on $\mathscr{R}_{\Theta}$.
\end{claimm}

We get the following lemma with Claim \ref{alalskdu65f}. 
\begin{lemmach}
Consider $\mathscr{R}_{\Theta}\leq D_0\times...\times D_{n-1}$, and binary relations $\sigma_0,...,\sigma_{n-1}$ where $\sigma_i$ is a congruence of $D_i$ for every $i$. Suppose that $E_i$ is a congruence block of $\sigma_i$ for all $i$. Then $V^0$ proves that $( E_0\times...\times E_{n-1})\cap \mathscr{R}_{\Theta}$ is a subuniverse of $\mathscr{R}_{\Theta}$.
\end{lemmach}

We need to define factor sets and factor operations for third-order objects. We first show how to define them for solution set $\mathscr{R}_{\Theta}$ and extended congruence $\mathscr{C}_{\sigma_i^{ext}}$ and the intersection of extended congruences $\mathscr{C}_{\cap_{n}\sigma^{ext}_i}$. To be compatible with the definition of factor sets for second-order objects, we need to choose not the existing map of $\mathscr{R}_{\Theta}$, but that sending $i$ to $a_i=rep(a_i/\sigma_i,D_i,\sigma_i)$ for every $\sigma_i$. 
\begin{equation}
  \begin{gathered}
 \hspace {0pt}factorset(\mathscr{R}_{\Theta}, \mathscr{C}_{\cap_{n}\sigma^{ext}_i})(H) = \mathscr{R}_{\Theta}/\mathscr{C}_{\cap_{n}\sigma^{ext}_i}(H)\iff \exists H'\leq\langle n,\langle n,l\rangle\rangle,H'\in \mathscr{R}_{\Theta}\wedge\\
     \wedge \mathscr{C}_{\cap_{n}\sigma^{ext}_i}(H,H')\wedge \\
     \wedge \forall i<n,\exists a_i\in D_i,\,H'(i)=\langle i,a_i/\sigma_i\rangle\wedge H(i)=\langle i,rep(a_i/\sigma_i,D_i,\sigma_i)\rangle.
  \end{gathered}
\end{equation}
Note that the value of the function is a third-order object, but its definition is again essentially second-order. To define factor set for $\mathscr{C}_{\sigma^{ext}_i}$, consider $\mathscr{C}_{\cap_{n}\sigma^{ext}_i}$ where for each $j\neq i$, $\sigma_j$ is $\nabla_{D_j}$. To represent an element $H'$ we define a string function $rep(H',\mathscr{R}_{\Theta},\mathscr{C}_{\cap_{n}\sigma^{ext}_i})$:
\begin{equation}
  \begin{gathered}
     H=rep(H',\mathscr{R}_{\Theta},\mathscr{C}_{\cap_{n}\sigma^{ext}_i})\iff \mathscr{C}_{\cap_{n}\sigma^{ext}_i}(H,H')\wedge factorset(\mathscr{R}_{\Theta},\mathscr{C}_{\cap_{n}\sigma^{ext}_i})(H).
  \end{gathered}
\end{equation}
Finally, we define third-order valued function $factor\omega$:
\begin{equation}
  \begin{gathered}
     factor\omega(\mathscr{R}_{\Theta},\mathscr{F}_{\Omega},\mathscr{C}_{\cap_{n}\sigma^{ext}_i})(B) = \mathscr{F}_{\Omega/\cap_{n}\sigma^{ext}_i}(B)\iff \exists H_1...\exists H_m \in \mathscr{R}_{\Theta}/\mathscr{C}_{\cap_{n}\sigma^{ext}_i},\\
     \exists H \in \mathscr{R}_{\Theta}/\mathscr{C}_{\cap_{n}\sigma^{ext}_i},\,B=\langle H_1,...,H_m,H\rangle\wedge\\
     \hspace {0pt}\wedge \exists H'_1...\exists H'_m\exists H' \in \mathscr{R}_{\Theta}\wedge H=rep(H',\mathscr{R}_{\Theta},\mathscr{C}_{\sigma^{ext}_i})\wedge\bigwedge_{i<m} H_i=rep(H'_i/\sigma,D,\sigma)\wedge \\
 \hspace {0pt}\mathscr{F}_{\Omega/\sigma_0,...,\Omega/\sigma_{n-1}}(H_1,...,H_m,H).
  \end{gathered}
\end{equation}
As a factor algebra we consider a pair of classes $(\mathscr{R}_{\Theta}/\mathscr{C}_{\cap_{n}\sigma^{ext}_i}, \mathscr{F}_{\Omega/\cap_{n}\sigma^{ext}_i})$. 

Now, to define a factor set for the general third-order subalgebra $\mathscr{R}$ and the congruence relation $\mathscr{C}_{\sigma}$, we need to choose a representative of a congruence block. It can be done by choosing the minimum string (in the sense of (\ref{laalalskksskdjdjdjfh})) that represents maps from the block. The rest are defined analogously.

\subsection{Homomorphism and isomorphism between second and third order objects}

We say that there exists a homomorphism between two subalgebras $(B,\Omega_B)$, $(C,\Omega_C)$ of algebra $\mathbb{A}$ if 
\begin{equation}
\begin{gathered}
 \hspace{0pt}HOM_{alg}(B,\Omega_B,C,\Omega_C)\iff \exists H<\langle l,l\rangle, MAP(B,l,C,l,H)\wedge \\
 \hspace {0pt}\wedge \forall b_1,...,b_m,b\in B,
\Omega_B(b_1,...,b_m,b)\leftrightarrow\Omega_C(H(b_1),...,H(b_m),H(b)).
\end{gathered}
\end{equation}
The image and kernel of $B$ under $H$ can be returned by string-valued functions defined as follows:
\begin{equation}
\begin{gathered}
   img(B,H)(i)\iff i\in C\wedge \exists j\in B,\, H(j)=i,\\
   \hspace {0pt}ker(H)(i,j)\iff i,j\in B\wedge H(i)=H(j).
\end{gathered}
\end{equation}
We can easily formalize embedding, epimorphism, and isomorphism: 
\begin{equation}
\begin{gathered}
 \hspace {0pt}ISO_{alg}(B,\Omega_B,C,\Omega_C)\iff \exists H<\langle l,n\rangle, HOM_{alg}(B,\Omega_B,C,\Omega_C)\wedge \\
 \hspace{0pt}\wedge \forall i_1,i_2\in B,\,(H(i_1)=H(i_2)\rightarrow i_1=i_2)\wedge \forall j\in C,\exists i\in B, H(i)=j.
\end{gathered}
\end{equation}
Now, we define the relation of being isomorphic between third-order objects and second-order objects. This relation assumes the existence of a third-order object, a class of maps. We say that $\mathscr{M}$ is a \emph{well-defined map} between a class $\mathscr{R}$ and a set $D$ if
\begin{equation}
 \begin{gathered}
 MAP^{3,2}(\mathscr{R},D,\mathscr{M}) \iff \forall H\in \mathscr{R} \exists a\in D,\,\mathscr{M}(H,a) \wedge \\
  \hspace {0pt}
\forall H\in \mathscr{R} \,\forall a,b \in D\, (\mathscr{M}(H,a)\wedge \mathscr{M}(H,b) \to a=b).
 \end{gathered}
\end{equation} 
We say that $\mathscr{H}$ is a homomorphism from a class $(\mathscr{R},\mathscr{F})$ to an algebra $(D,\Omega)$ if 
\begin{equation}
\begin{gathered}
 \hspace {0pt}HOM_{alg}^{3,2}(\mathscr{R},\mathscr{F}, D,\Omega,\mathscr{H})\iff MAP^{3,2}(\mathscr{R},D,\mathscr{H})\wedge \\
 \hspace{0pt}\wedge \forall H_1,...,H_m,H\in \mathscr{R},
\mathscr{F}(H_1,...,H_m,H)\leftrightarrow\Omega(\mathscr{M}(H_1),...,\mathscr{M}(H_m),\mathscr{M}(H)),
\end{gathered}
\end{equation}
and that the class $(\mathscr{R},\mathscr{F})$ is isomorphic to the set $(D,\Omega)$ if
\begin{equation}
\begin{gathered}
 \hspace{0pt}ISO_{alg}^{3,2}(\mathscr{R},\mathscr{F}, D,\Omega)\iff \exists \mathscr{H},\, HOM^{3,2}(\mathscr{R},\mathscr{F}, D,\Omega,\mathscr{H})\wedge \forall H_1,H_2\in \mathscr{R},\\
 \hspace {0pt}(\mathscr{H}(H_1)=\mathscr{H}(H_2)\rightarrow H_1=H_2)\wedge \forall a\in D,\exists H\in \mathscr{R}, \mathscr{H}(H)=a.
\end{gathered}
\end{equation}
Analogously, we can define a map $MAP^{2,3}$, a homomorphism $HOM_{alg}^{2,3}$, and an isomorphism $ISO_{alg}^{2,3}$ from a set to a class. Finally, we define an isomorphism between third-order objects. We say that $\mathscr{M}$ is a \emph{well-defined map} between a class $\mathscr{R}$ and a class $\mathscr{R}'$ if
\begin{equation}
 \begin{gathered}
 MAP^{3,3}(\mathscr{R},\mathscr{R}',\mathscr{M}) \iff \forall H\in \mathscr{R} \exists H'\in \mathscr{R},\,\mathscr{M}(H,H') \wedge \\
  \hspace {0pt}
\forall H\in \mathscr{R} \,\forall H_1,H_2 \in \mathscr{R}'\, (\mathscr{M}(H,H_1)\wedge \mathscr{M}(H,H_2) \to H_1=H_2).
 \end{gathered}
\end{equation} 
We say that $\mathscr{H}$ is a homomorphism from a class $(\mathscr{R},\mathscr{F})$ to a class $(\mathscr{R}',\mathscr{F}')$ if 
\begin{equation}
\begin{gathered}
 \hspace {0pt}HOM_{alg}^{3,3}(\mathscr{R},\mathscr{F}, \mathscr{R}',\mathscr{F}',\mathscr{H})\iff MAP^{3,3}(\mathscr{R},\mathscr{R}',\mathscr{H})\wedge \\
 \hspace{0pt}\wedge \forall H_1,...,H_m,H\in \mathscr{R},
\mathscr{F}(H_1,...,H_m,H)\leftrightarrow\mathscr{F}'(\mathscr{M}(H_1),...,\mathscr{M}(H_m),\mathscr{M}(H)),
\end{gathered}
\end{equation}
and, finally,
\begin{equation}
\begin{gathered}
 \hspace{0pt}ISO_{alg}^{3,3}(\mathscr{R},\mathscr{F}, \mathscr{R}',\mathscr{F}')\iff \exists \mathscr{H},\, HOM^{3,3}(\mathscr{R},\mathscr{F}, \mathscr{R}',\mathscr{F}',\mathscr{H})\wedge \forall H_1,H_2\in \mathscr{R},\\
 \hspace{0pt}(\mathscr{H}(H_1)=\mathscr{H}(H_2)\rightarrow H_1=H_2)\wedge \forall H'\in \mathscr{R}',\exists H\in \mathscr{R}, \mathscr{H}(H)=H'.
\end{gathered}
\end{equation}

For every domain $D_i$ and any of its subuniverse $B_i$, we define its extension $\mathscr{B}^{ext}_i$ to third-order object, as a set of maps from $[n]$ to $[D_0,...,D_{n-1}]$ such that it contains all maps sending $i$ to elements of $B_i$:
\begin{equation}
  \begin{gathered}
     \mathscr{B}^{ext}_i(H)\iff \exists a_i\in B_i,\,H(i)=\langle i,a_i\rangle.
  \end{gathered}
\end{equation}

\subsection{Auxiliary definitions from Zhuk's algorithm}

Some notions, which were used in Zhuk's algorithm mainly in relation to constraints, we give both for binary relations and $n$-ary relations. 

\subsubsection{Crucial Instance}
We say that $i$-th variable of a constraint $C_j=(y_1,...,y_k;R)$ is \emph{dummy} if $R$ does not depend on $y_i$. For an instance $\Theta$ a constraint $C$ is called \emph{crucial} in $D^{(\bot)} = (D^{(\bot)}_0,...,D^{(\bot)}_{n-1})$, where $D^{(\bot)}_i\subseteq D_i$ for each $i$, if it does not have dummy variables, $\Theta$ has no solutions in $D^{(\bot)}$, but the replacement of $C$ by all weaker constraints gives an instance with a solution in $D^{(\bot)}$. A CSP instance $\Theta$ is crucial in $D^{(\bot)}$ if every constraint of $\Theta$ is crucial in $D^{(\bot)}$. In this section, we will formalize this notion. 

\begin{definition}[Reduction of the domain set]\label{kakakjd8j74r}
For an instance $\Theta=(\mathcal{X},\ddot{\mathcal{A}})$ with domain set $D=(D_0,...,D_{n-1})$ we say that a set $D^{(\bot)}=(D^{(\bot)}_0,...,D^{(\bot)}_{n-1})$ is a reduction of $D$ if $D^{(\bot)}_i$ is a subuniverse of $D_i$ for every $i$.
\begin{equation}
  \begin{gathered}
 Red(D^{(\bot)},D)\iff \forall i\leq n,\,subTA(D^{(\bot)}_i,D_i).
  \end{gathered}
\end{equation}
In the definition, we can additionally require that equal domains be reduced to equal domains, i.e.
\begin{equation}\label{ajdkjfhuei4}
  \forall i\forall j,\, D_i=D_j \rightarrow D^{(\bot)}_i=D^{(\bot)}_j.
\end{equation}
We shall use it later considering different modifications of the instance to avoid abuse of the notation. 
\end{definition}

\begin{definition}[Instance after reduction]
For an instance $\Theta=(\mathcal{X},\ddot{\mathcal{A}})$, we need to define an instance $\Theta^{(\bot)}=(\mathcal{X}^{(\bot)},\ddot{\mathcal{A}}^{(\bot)})$ after the reduction of a domain set of a target digraph $\ddot{\mathcal{A}}=( V_{\ddot{\mathcal{A}}},E_{\ddot{\mathcal{A}}})$ from $D=(D_0,...,D_{n-1})$ to $D^{(\bot)}=(D^{(\bot)}_0,...,D^{(\bot)}_{n-1})$. Since there is a unique way to construct a reduction of an instance, we actually can define a string function (using a bit-defining axiom) that returns a reduced instance:
\begin{equation}
  \begin{gathered}
   \hspace{0pt}redinst(\Theta,D^{(\bot)})(\mathcal{X}^{(\bot)},\ddot{\mathcal{A}}^{(\bot)}) = \Theta^{(\bot)}(\mathcal{X}^{(\bot)},\ddot{\mathcal{A}}^{(\bot)})\iff Red(D^{(\bot)},D)\wedge \\
    \hspace {0pt}\wedge (\mathcal{X}^{(\bot)}=\mathcal{X})\wedge\\
   \wedge(\forall i,j<n,\forall a,b<l,E^{ij}_{\ddot{\mathcal{A}}^{(\bot)}}(a,b ) \leftrightarrow E^{ij}_{\ddot{\mathcal{A}}}(a,b ) \wedge a\in D^{(\bot)}_i\wedge b\in D^{(\bot)}_j).
  \end{gathered}
\end{equation}
We say that $D^{(\bot)}$ is a $1$-consistent reduction if the instance $\Theta^{(\bot)}$ is $1$-consistent, $1C(\Theta^{(\bot)})$.
\end{definition}
Sometimes, when we work with a nonlinked instance, we need to produce its \emph{linked component}, i.e. elements that can be connected by a path in the instance. To this end, we need to define the notion of being \emph{linked} for two elements in $a\in D_i,b\in D_j$. We say that there is a path from $i$ to $j$ in the input digraph $\mathcal{X}$ if there exists a path $\mathcal{P}_t$ of some length $t$ that can be homomorphically mapped to $\mathcal{X}$ such that $H(0)=i$ and $H(t)=j$:
\begin{equation}
 \begin{gathered}
  Path(i,j,\mathcal{X}) \iff \exists t<n, \exists V_{\mathcal{P}_t}=t, \exists E_{\mathcal{P}_t}\leq 4t^2,\,PATH(V_{\mathcal{P}_t},E_{\mathcal{P}_t})\wedge\\  \hspace{0pt}\wedge\exists H\leq \langle t,n \rangle, HOM(\mathcal{P}_t,\mathcal{X},H)\wedge(H(0,i)\wedge H(t,j)).
 \end{gathered}
\end{equation}
We say that the path $\mathcal{P}_t$ connects $i$ and $j$. Also, we can encode what it means to be linked for two elements $a\in D_i,b\in D_j$. In words, there must exist a path $\mathcal{P}_t$ of some length $t$ connecting $i,j$ with homomorphism $H$ such that there exists a homomorphism $H'$ from $\mathcal{P}_t$ to $\ddot{\mathcal{A}}$ sending $0$ to $\langle i,a\rangle$ and $t$ to $\langle j,b\rangle$, and for every element $p<t$, $H(p)=k$ implies that $H(p)=\langle k,c\rangle$ for some $c\in D_k$. We can express this by the $\Sigma^{1,b}_1$-formula.
\begin{equation}\label{JJKTRE6}
\begin{gathered}
 \hspace{0pt}Linked(a,b,i,j,\Theta) \iff \exists t<nl, \exists V_{\mathcal{P}_t}=t, \exists E_{\mathcal{P}_t}\leq 4t^2,\\
 \exists H\leq \langle t,n \rangle,\,PATH(V_{\mathcal{P}_t},E_{\mathcal{P}_t})\wedge HOM(\mathcal{P}_t,\mathcal{X},H)\wedge(H(0,i)\wedge H(t,j))\wedge\\
 \hspace{0pt}\wedge\exists H'\leq \langle t,\langle t,l\rangle\rangle,\ddot{HOM}(\mathcal{P}_t,\ddot{\mathcal{A}},H')\wedge\\
 \hspace{0pt}\wedge(\forall k<n,p<t,\,(H(p,k)\rightarrow \exists c\in D_k,\,H'(p)=\langle k,c\rangle))\\
 \hspace{0pt}\wedge H'(0)=\langle i,a\rangle \wedge H'(t)= \langle j,b\rangle.
\end{gathered}
\end{equation}
\begin{notation}
Sometimes we will write
$\exists \mathcal{P}_t< \langle n,4n^2\rangle, Path(i,j,\mathcal{X},\mathcal{P}_t)$ and $\exists \mathcal{P}_t< \langle nl,$ $4(nl)^2\rangle,$ $ Linked(a,b,i,j,\Theta,\mathcal{P}_t)$ to omit repetitions. Note that while $Path(i,j,\mathcal{X})$ and $Linked(a,b,i,j,\Theta)$ are $\Sigma^{1,b}_1$-formulas, $Path(i,j,\mathcal{X},\mathcal{P}_t)$ and $Linked(a,b,i,j,\Theta,\mathcal{P}_t)$ are $\Sigma^{1,b}_0$-formulas.
\end{notation}
We have formalized in  \cite{gaysin2023proof} that $Linked(a,b,i,i,\Theta)$ is a congruence relation on $D_i$, and that for a non-fragmented instance, this congruence provides a partition into linked components. That is, each linked component can be viewed as the same CSP instance on smaller domains.

\begin{definition}[Linked component]
 We define a string function $linkcomp(\Theta,D_i,a)$ that produces a linked reduction of the domain set based on an element $a$ in the domain $D_i$. 
\begin{equation}
  \begin{gathered}
    linkcomp(\Theta,D_i,a)(j,b) = V^{link,i,a}_{\ddot{\mathcal{A}}}(j,b)\iff \exists \mathcal{P}_t< ( nl,4(nl)^2),\\
    Linked(a,b,i,j,\Theta,\mathcal{P}_t).
  \end{gathered}
\end{equation}
Then a $\Sigma^{1,b}_1$-function 
$$redinst(\Theta, linkcomp(\Theta,D_i,a))$$
produces a linked reduction of instance $\Theta$, which contains the element $a$ in domain $D_i$.

\end{definition}

\begin{definition}[Dummy variable]

 A variable $x_i$ of an edge $(x_i,x_j)\in E_{\mathcal{X}}$ is dummy if for every $b\in D_j$ such that there exists $a\in D_i$, $E^{ij}_{\ddot{\mathcal{A}}}(a,b)$, there is an edge $(a',b)\in E^{ij}_{\ddot{\mathcal{A}}}$ for every $a'\in D_i$.
 \begin{equation}
   \begin{gathered}
 Dum_2(E^{ij}_{\ddot{\mathcal{A}}},i)\iff\forall b\in D_j, (\exists a\in D_i, E^{ij}_{\ddot{\mathcal{A}}}(a,b) \rightarrow \forall a'\in D_i, E^{ij}_{\ddot{\mathcal{A}}}(a',b)).
   \end{gathered}
 \end{equation}
 Note that for a $1$-consistent CSP instance this means that $E^{ij}_{\ddot{\mathcal{A}}}$ is a full relation:
 \begin{equation}
   FullRel(E^{ij}_{\ddot{\mathcal{A}}}) \iff \forall a\in D_i,\forall b\in D_j,E^{ij}_{\ddot{\mathcal{A}}}(a,b).
 \end{equation}
We also introduce the notion of being a dummy variable for a solution set $\mathscr{R}_{\Theta}$. We say that a variable $x_i$ is dummy if the following $\Pi^{1,b}_2$-relation holds:
\begin{equation}
  \begin{gathered}
     Dum(\mathscr{R}_{\Theta}, i)\iff \forall H \leq \langle n,\langle n,l\rangle\rangle,\, (\ddot{HOM}(\mathcal{X},\ddot{\mathcal{A}},H)\rightarrow \forall a\in D_i, \\
     \exists H' \leq \langle n,\langle n,l\rangle\rangle,
     \hspace {0pt}\ddot{HOM}(\mathcal{X},\ddot{\mathcal{A}},H')\wedge H'(i) = \langle i, a\rangle\wedge \forall j\neq i<n,\, H'(j) = H(j)  ).
   \end{gathered}
\end{equation}
 
\end{definition}

\begin{definition}[Weaker constraint]
  For a binary constraint $E^{ij}_{\ddot{\mathcal{A}}}$ there are only two types of weaker constraint: domains $D_i$, $D_j$, which are weaker constraints of less arity (which we never increase), and all binary constraints from the list $\Gamma_{\mathcal{A}}$ containing $E^{ij}_{\ddot{\mathcal{A}}}$, including the full relation on $D_i\times D_j$ (as if we remove a constraint at all). We say that $E$ is a \emph{weaker constraint} than $E^{ij}_{\ddot{\mathcal{A}}}$ if
  \begin{equation}
    \begin{gathered}
 \hspace{0pt}Weaker(E,E^{ij}_{\ddot{\mathcal{A}}})\iff Pol_{m,2}(\Omega, A,E)\wedge (\forall a,b<l, E(a,b)\rightarrow \\
 \hspace {0pt}\rightarrow a\in D_i\wedge b\in D_j) \wedge \big[FullRel(E) \vee\\
 \hspace {0pt}\vee\big((\forall a\in D_i,\forall b\in D_j,\,E^{ij}_{\ddot{\mathcal{A}}}(a,b)\rightarrow E(a,b))\wedge(\exists a\in D_i,\\
 \hspace {0pt}\exists b\in D_j,\,E(a,b)\wedge \neg E^{ij}_{\ddot{\mathcal{A}}}(a,b))\big)\big].
    \end{gathered}
  \end{equation}
  \end{definition}
Note that for any constraint $E^{ij}_{\ddot{\mathcal{A}}}$ there exists at least one weaker constraint (namely the full relation). Any time we weaken a constraint $E^{ij}_{\ddot{\mathcal{A}}}$ we replace it with all weaker constraints simultaneously. That is, we consider an intersection of all weaker constraints. But since in the list $\Gamma_{\mathcal{A}}$ we have all $pp$-definable binary relations invariant under $\Omega$, there exists $k<2^{l^2}$ such that $\Gamma^2_{\mathcal{A},k}$ is that intersection. A problem here arises when the intersection of all weaker constraints of a constraint is the constraint itself: it just means that there are several incomparable intersections of weaker constraints. In this case, we can choose one of them arbitrarily, and we will choose the one with the smallest $k<2^{l^2}$. We first define a string function that returns the list of such intersections.
\begin{equation}
  \begin{gathered}
     weakerlist(E^{ij}_{\ddot{\mathcal{A}}})(k)\iff Weaker(\Gamma^2_{\mathcal{A},k},E^{ij}_{\ddot{\mathcal{A}}})\wedge \forall g\neq k<2^{l^2},\\
     \hspace {0pt}\neg (Weaker(\Gamma^2_{\mathcal{A},g},E^{ij}_{\ddot{\mathcal{A}}})\wedge Weaker(\Gamma^2_{\mathcal{A},k},\Gamma^2_{\mathcal{A},g})).
  \end{gathered}
\end{equation}

\begin{definition}[Weakening of a constraint]
    We define a string function that returns the first intersection from the list $weakerlist$. We will call it the \emph{weakening of the constraint} $E^{ij}_{\ddot{\mathcal{A}}}$ and denote by $E^{ij}_{\ddot{\mathcal{A}},w}$:
\begin{equation}
  \begin{gathered}
     weakening(E^{ij}_{\ddot{\mathcal{A}}})(a,b)=E^{ij}_{\ddot{\mathcal{A}},w}(a,b)\iff \exists i< 2^{l^2},\,(weakerlist(E^{ij}_{\ddot{\mathcal{A}}})(i)\wedge \\
     \hspace {0pt}\wedge\Gamma^2_{\mathcal{A},i}(a,b))\wedge \forall j<i,\,\neg weakerlist(E^{ij}_{\ddot{\mathcal{A}}})(j).
  \end{gathered}
\end{equation}
\end{definition}

The function is well-defined due to the Number Minimization axiom $\Sigma^{1,b}_0$-MIN and since the list $weakerlist(E^{ij}_{\ddot{\mathcal{A}}})$ is never empty. Thus, we can uniquely define the instance after weakening a constraint $E^{ij}_{\ddot{\mathcal{A}}}$.
\begin{definition}[Instance after weakening]
For an instance $\Theta=(\mathcal{X},\ddot{\mathcal{A}})$, the instance $\Theta_{w^{ij}}=(\mathcal{X}_{w^{ij}},\ddot{\mathcal{A}}_{w^{ij}})$ after the weakening of a constraint $E^{ij}_{\ddot{\mathcal{A}}}$ is defined by the following string function:
\begin{equation}
  \begin{gathered}
 \hspace{0pt}weakinst(\Theta,E^{ij}_{\ddot{\mathcal{A}}})(\mathcal{X}_{w^{ij}},\ddot{\mathcal{A}}_{w^{ij}}) = \Theta_{w^{ij}}(\mathcal{X}_{w^{ij}},\ddot{\mathcal{A}}_{w^{ij}})\iff (D_{w^{ji}}=D)\wedge \\
 \wedge (V_{\mathcal{X}_{w^{ij}}} = V_{\mathcal{X}})\wedge(\forall t\neq i<n,\forall r\neq j<n, (E_{\mathcal{X}_{w^{ij}}}(t,r)\leftrightarrow E_{\mathcal{X}}(t,r))\wedge \\
 \wedge (E^{tr}_{\ddot{\mathcal{A}}_{w^{ij}}} = E^{tr}_{\ddot{\mathcal{A}}}))\wedge\\ \wedge E^{ij}_{\ddot{\mathcal{A}}_{w^{ij}}}=weakening(E^{ij}_{\ddot{\mathcal{A}}}) \wedge (FullRel(E^{ij}_{\ddot{\mathcal{A}}_{w^{ij}}}) \leftrightarrow \neg E_{\mathcal{X}}(i,j)).
  \end{gathered}
\end{equation}
Note that the last line ensures that if the only weaker binary relation to $E^{ij}_{\ddot{\mathcal{A}}}$ is a full relation, then we remove an edge from $\mathcal{X}$. Finally, we are ready to define a crucial instance. 
\end{definition}

\begin{definition}[Crucial instance]\label{a'a''as;s;s;dlfkjg}
Let $D_i^{(\bot)}\subseteq D_i$ for every $i$, and let $D^{(\bot)}$ be a reduction of $D$. A constraint $E^{ij}_{\ddot{\mathcal{A}}}$ of instance $\Theta$ is called crucial in $D^{(\bot)}$ if
\begin{equation}\label{dhghdhdfyg}
  \begin{gathered}
 CrucConst(E^{ij}_{\ddot{\mathcal{A}}},\Theta,D^{(\bot)})\iff\neg Dum_2(E^{ij}_{\ddot{\mathcal{A}}},i)\wedge \neg Dum_2(E^{ij}_{\ddot{\mathcal{A}}},j)\wedge \\
 \hspace{0pt}\neg \ddot{HOM}(\mathcal{X}^{(\bot)},\ddot{\mathcal{A}}^{(\bot)})\wedge \ddot{HOM}(\mathcal{X}_{w^{ij}}^{(\bot)},\ddot{\mathcal{A}}_{w^{ij}}^{(\bot)}).
  \end{gathered}
\end{equation}
We say that a CSP instance $\Theta=(\mathcal{X},\ddot{\mathcal{A}})$ is crucial in $D^{(\bot)}$ if
\begin{equation}\label{akaksijdhghf}
  \begin{gathered}
 CrucInst(\Theta,D^{(\bot)})\iff \forall j,i<n, E_{\mathcal{X}}(i,j)\rightarrow CrucConst(E^{ij}_{\ddot{\mathcal{A}}}, \Theta,D^{(\bot)}).
  \end{gathered}
\end{equation}
\end{definition}
Note that all formulas used in the definitions of this section except Definition \ref{a'a''as;s;s;dlfkjg}, are $\Sigma^{1,b}_0$. Formulas (\ref{dhghdhdfyg}) and (\ref{akaksijdhghf}) are from the class $\mathfrak{B}(\Sigma^{1,b}_1)$, Boolean combinations of $\Sigma^{1,b}_1$-formulas.

\subsubsection{Covering and expanded covering}
We can consider a CSP instance $\Theta$ on $n$ variables as a set of constraints of the form $E^{ij}_{\ddot{\mathcal{A}}}$ for all $i,j<n$ (we do not consider domains here as constraints). 

\begin{definition}[Tree-instance]
We say that an instance $\Theta$ is a \emph{tree-formula} if there is no path $z_1 - C_1 - z_2 - ... - z_{l-1} - C_l - z_l$ such that $l\leq 3$, $z_1=z_l$, and all the constraints $C_1,C_2,...,C_l$ are different. Since in our setting for any $i,j<n$ we suppose that there is only one constraint relation $E^{ij}_{\ddot{\mathcal{A}}}$ (we can do this because we have any intersection of any invariant relations in our list $\Gamma_{\mathcal{A}}$), an instance $\Theta=(\mathcal{X},\ddot{\mathcal{A}})$ is a tree-formula if it does not contain cycles. It can be expressed by the following $\Pi^{1,b}_1$-formula:
\begin{equation}\label{DEFINITIONOFCYCLECONSISTENCY}
  \begin{gathered}
     \hspace {0pt}TreeInst(\mathcal{X}, \ddot{\mathcal{A}})\iff \forall t<n^2, \forall V_{\mathcal{C}_t}=t, \forall E_{\mathcal{C}_t}\leq 4t^2, \forall H<\langle t,n \rangle,\\   
 \hspace{0pt}CYCLE(V_{\mathcal{C}_t},E_{\mathcal{C}_t})\wedge HOM (\mathcal{C}_t,\mathcal{X},H)\rightarrow \exists i_1\neq j_1<t,\exists i_2\neq j_2<t,\exists k_1,k_2<n,\\
     \hspace {0pt}E_{\mathcal{C}_t}(i_1,i_2)\wedge E_{\mathcal{C}_t}(j_1,j_2)\wedge H(i_1,k_1)\wedge H(i_2,k_2)\wedge H(j_1,k_1)\wedge H(j_2,k_2).
  \end{gathered}
\end{equation}
That is, for any cycle $\mathcal{C}_t$, any homomorphism from $\mathcal{C}_t$ to $\mathcal{X}$ must glue at least two different edges of $\mathcal{C}_t$.
\end{definition}

\begin{definition}[Subinstance]
  For instance $\Theta = (\mathcal{X}, \ddot{\mathcal{A}})$ we call $\Theta'=(\mathcal{X}',\ddot{\mathcal{A}})$ a \emph{subinstance} of $\Theta$ if $\Theta'$ is a subset of the variables together with some subset of constraints from $\Theta$ that only involve these variables, i.e.:
\begin{equation}
\begin{gathered}
   subInst(\Theta', \Theta)\iff \ddot{\mathcal{A}}'=\ddot{\mathcal{A}}\wedge V_{\mathcal{X'}}\subseteq V_{\mathcal{X}}\wedge E_{\mathcal{X'}}\subseteq E_{\mathcal{X}}\wedge\\
   \hspace {0pt}(E_{\mathcal{X'}}(x_1,x_2)\rightarrow x_1,x_2\in V_{\mathcal{X'}}).
\end{gathered}
\end{equation}
That is, the target digraph with domains $\ddot{\mathcal{A}}$ does not change, the set of vertices $V_{\mathcal{X'}}$ is a subset of $V_{\mathcal{X}}$, and the set of constraints $E_{\mathcal{X'}}$ is a subset of $E_{\mathcal{X}}$ defined only on $V_{\mathcal{X'}}$. The solution to such a subinstance is a partial homomorphism. 

Consider instance $\Theta$ as a set of constraints $\{E^{ij}_{\ddot{\mathcal{A}}}:i,j<n\}$. Then consider a subset $\Theta'$ of such constraints, a subinstance of $\Theta$. We need to define uniquely a subinstance $\Theta\backslash\Theta'$ as a string function:
\begin{equation}
  \begin{gathered}
     \hspace {0pt}dif(\Theta,\Theta')(\mathcal{X}_{\Theta\backslash\Theta'},\ddot{\mathcal{A}}_{\Theta\backslash\Theta'})=\Theta\backslash\Theta'(\mathcal{X}_{\Theta\backslash\Theta'},\ddot{\mathcal{A}}_{\Theta\backslash\Theta'})\iff \ddot{\mathcal{A}}_{\Theta\backslash\Theta'}=\ddot{\mathcal{A}}\wedge \\
     \hspace {0pt}\wedge \forall i,j<n,\,E_{\mathcal{X}\backslash\mathcal{X'}}(i,j)\leftrightarrow \,E_{\mathcal{X}}(i,j)\wedge \neg E_{\mathcal{X}'}(i,j)\wedge \\
     \wedge V_{\mathcal{X}\backslash\mathcal{X}'}\subseteq V_{\mathcal{X}}\wedge (\forall i<n,V_{\mathcal{X}\backslash\mathcal{X'}}(i)\leftrightarrow\exists j<n,\,\neg E_{\mathcal{X}'}(i,j)\wedge \neg E_{\mathcal{X}'}(i,j)\wedge \\
     \wedge( E_{\mathcal{X}}(i,j)\vee E_{\mathcal{X}}(i,j)). \\
  \end{gathered}
\end{equation}
Note that $\Theta'$ and $\Theta\backslash\Theta'$ can share common variables, so the third line in the formula places to $V_{\mathcal{X}\backslash\mathcal{X}'}$ only variables that are involved in some constraint not in $\Theta'$. We also lose all the variables that are not involved in any constraint, neither in $\Theta'$ nor in $\Theta$.

\end{definition}

For the rest part of this section and sometimes further when we talk about (expanded) covering and substitutions, we will use labels for vertex sets instead of elements. For any instance $\Theta$ with a vertex set $V_{\mathcal{X}}$ we can introduce as many labels as we want using two-dimensional strings $Y, Z, W$, and the function $seq(i,Y)=y_i$. They are bounded on the first coordinate by the number of vertices and on the second coordinate by some reasonable number of labels. Let us denote this bound for $n$ variables by $b_n$. We will use $y_i,z_j,w_k<b_n$ in the formulas when appropriate. When we write $\forall i<n,\,R(y_i)$, this is an abbreviation for
$$\forall i<n,\, R(seq(i,Y)).
$$
The representation of the entire vertex set for a digraph $\mathcal{X}$ is $V_{\mathcal{X}}(i,x_i)$, and the representation of the set of vertices of a digraph $\ddot{\mathcal{A}}$ is $V(\langle seq(i,V_{\mathcal{X}}), a\rangle)$, which does not differ much from our usual representations. 
For an instance $\Lambda$ and two sets of variables $z_1,...,z_k$ and $y_1,...,y_k$ by $\Lambda^{y_1,...,y_k}_{z_1,...,z_k}$ we denote the instance obtained from instance $\Lambda$ by replacing every variable $z_i$ by $y_i$. This can be expressed by a $\Sigma^{1,b}_0$ string function $$substitute(\Lambda,Y,Z)(\mathcal{X}_{\Lambda^{y_1,...,y_k}_{z_1,...,z_k}},\ddot{\mathcal{A}}_{\Lambda^{y_1,...,y_k}_{z_1,...,z_k}}) = \Lambda^{y_1,...,y_k}_{z_1,...,z_k}(\mathcal{X}_{\Lambda^{y_1,...,y_k}_{z_1,...,z_k}},\ddot{\mathcal{A}}_{\Lambda^{y_1,...,y_k}_{z_1,...,z_k}}),$$
which definition is rather tedious than interesting, so we do not present it here. We also need to define a union of two sets of constraints, i.e. a union of two instances $\Theta_{\mathcal{X}} = (\mathcal{X},\ddot{\mathcal{A}})$ with $x_0,...x_{n-1}$ variables, $\ddot{\mathcal{A}}=( V_{\ddot{\mathcal{A}}},E_{\ddot{\mathcal{A}}})$, and $\Theta_{\mathcal{Y}}=(\mathcal{Y},\ddot{\mathcal{B}})$ with $y_0,...,y_{m-1}$ variables, $\ddot{\mathcal{B}}=( V_{\ddot{\mathcal{B}}},E_{\ddot{\mathcal{B}}})$. The problem here is that they can share common variables (that are labeled by the same number). To be safe and to easily track the number of variables, we just copy the common variables $x_i,y_j$, and set $E^{x_iy_j}_{\ddot{\mathcal{A}}\cup \ddot{\mathcal{B}}}$ to be equality relation (we have it since in our list $\Gamma_{\mathcal{A}}$ we have all relations $pp$-definable from $\Gamma$). To copy variables without collisions, we first define a number function. 
\begin{equation}\label{ala9s7d5tdhgdfr}
  \begin{gathered}
     maxlabel(V_{\mathcal{X}}) = s \iff \forall i<n,\, seq(i,V_{\mathcal{X}})\leq s\wedge \exists i<n,\,seq(i,V_{\mathcal{X}}) = s,
  \end{gathered}
\end{equation}
and use a label $z_i=y_i+maxlabel(V_{\mathcal{X}})$ for every $i<m$ in the following definition. Then we define a string function $uni$ on two arguments by its bit-definition: 
\begin{equation}
  \begin{gathered}
     \hspace {0pt}uni(\Theta_{\mathcal{X}},\Theta_{\mathcal{Y}})(\mathcal{X}\cup \mathcal{Y},\ddot{\mathcal{A}}\cup\ddot{\mathcal{B}} ) = \Theta_{\mathcal{X}}\cup\Theta_{\mathcal{Y}}(\mathcal{X}\cup \mathcal{Y},\ddot{\mathcal{A}}\cup\ddot{\mathcal{B}} )\iff \\
     \hspace {0pt}\forall i<n,\,V_{\mathcal{X}\cup \mathcal{Y}}(i,x_i)\wedge \forall i<m,\,V_{\mathcal{X}\cup \mathcal{Y}}(n+i,z_i)\wedge\\
      \hspace {0pt}\wedge \forall i,j<n,\, E_{\mathcal{X}\cup \mathcal{Y}}(x_i,x_j)\leftrightarrow E_{\mathcal{X}}(x_i,x_j) \wedge \forall i,j<m,\,E_{\mathcal{X}\cup \mathcal{Y}}(z_i,z_j)\leftrightarrow E_{\mathcal{Y}}(y_i,y_j)\wedge\\
      \hspace {0pt}\wedge\forall i<n,\forall a<l,\,V_{\ddot{\mathcal{A}}\cup\ddot{\mathcal{B}}}(x_i,a)\leftrightarrow V_{\ddot{\mathcal{A}}}(x_i,a)\wedge \\
      \wedge \forall i<m,\forall a<l,\,V_{\ddot{\mathcal{A}}\cup\ddot{\mathcal{B}}}(z_i,a)\leftrightarrow V_{\ddot{\mathcal{B}}}(y_i,a)\wedge \\
     \wedge \forall i,j<n, \forall a,b<l,\,E^{x_ix_j}_{\ddot{\mathcal{A}}\cup\ddot{\mathcal{B}}}(a,b)\leftrightarrow E^{x_ix_j}_{\ddot{\mathcal{A}}}(a,b)\wedge \\
     \wedge \forall i,j<m, \forall a,b<l,\,E^{z_iz_j}_{\ddot{\mathcal{A}}\cup\ddot{\mathcal{B}}}(a,b)\leftrightarrow E^{y_iy_j}_{\ddot{\mathcal{B}}}(a,b)\wedge\\
     \hspace {0pt}\wedge \forall i<n\forall j<m,\,x_i=y_j\rightarrow \forall a\in D_{x_i},\, E^{x_iz_j}_{\ddot{\mathcal{A}}\cup\ddot{\mathcal{B}}}(a,a).
  \end{gathered}
\end{equation}
Due to this definition, function $uni$ is not commutative, but $\Theta_{\mathcal{X}}\cup\Theta_{\mathcal{Y}}$ and $\Theta_{\mathcal{Y}}\cup\Theta_{\mathcal{X}}$ are obviously isomorphic. We can iteratively define $\Theta_{\mathcal{X}_1}\cup\Theta_{\mathcal{X}_2}\cup...\cup \Theta_{\mathcal{X}_n} = (\Theta_{\mathcal{X}_1}\cup\Theta_{\mathcal{X}_2}\cup...\cup \Theta_{\mathcal{X}_{n-1}})\cup \Theta_{\mathcal{X}_n}$.

A $pp$-formula $\exists y_0,...,y_{k-1}\Theta'(x_0,...,x_{n-1})$, where $y_0,...,y_{k-1}$ are the only variables that occur in $\Theta'$ except $x_0,...,$ $x_{n-1}$, is called a \emph{subconstraint} of $\Theta$ if $\Theta'\subseteq \Theta$ and $\Theta'$ and $\Theta\backslash\Theta'$ do not have common variables except for $x_0,...,x_{n-1}$. We can consider $\Theta'$ as a subinstance that involves variables $x_0,...,x_{n-1},y_0,...,y_{k-1}$, and $\Theta$ as an instance on variables $x_0,...,x_{n-1},y_0,...,y_{k-1}, z_0,...,z_{s-1}$. Constraints involving variables $y_0,...,y_{k-1}$ occur only in $\Theta'$, and constraints involving $z_0,...,z_{s-1}$ occur only in $\Theta\backslash\Theta'$. We code common variables by a set $X$. Then we can define a subconstraint in the following way:
\begin{equation}
  \begin{gathered}
     subConst(\Theta,\Theta', X)\iff subInst(\Theta', \Theta)\wedge \forall i,j,k<(n+k+s),\ \\
     \hspace {0pt}E_{\mathcal{X}'}(i,j)\wedge E_{\mathcal{X}\backslash\mathcal{X}'}(j,k)\rightarrow \exists s<(n+k+s), j = X(s,x_s).
  \end{gathered}
\end{equation}
Here, for brevity's sake, we abbreviate by $E_{\mathcal{X}}(i,j)\wedge E_{\mathcal{X}\backslash\mathcal{X}'}(j,k)$ all four combinations of non-symmetric constraints. Then $\exists y_0,...,y_{k-1}\Theta'(x_0,...,x_{n-1})$ defines a projection of solution set to the CSP instance $\Theta'$ on the coordinates $x_0,...,x_{n-1}$, $\mathscr{R}^{x_0,...,x_{n-1}}_{\Theta'}$.

We further define technical notions of covering and expanded covering. We define them for languages with at most binary relations and, for the general definition, refer the reader to \cite{zhuk2020proof}. 
\begin{definition}[Covering]
For an instance $\Theta_{\mathcal{X}} = (\mathcal{X},\ddot{\mathcal{A}})$ with $x_0,...x_{n-1}$ variables, $\ddot{\mathcal{A}}=( V_{\ddot{\mathcal{A}}},E_{\ddot{\mathcal{A}}})$, we say that an instance $\Theta_{\mathcal{Y}}=(\mathcal{Y},\ddot{\mathcal{B}})$ with $y_0,...,y_{m-1}$ variables and $\ddot{\mathcal{B}}=( V_{\ddot{\mathcal{B}}},E_{\ddot{\mathcal{B}}})$ is a covering of $\Theta$ if the following $\Sigma^{1,b}_1$-relation holds:
\begin{equation}
\begin{gathered}
 \hspace {0pt}Cov(\Theta_{\mathcal{Y}},\Theta_{\mathcal{X}})\iff \exists H<\langle b_m,b_n \rangle, HOM(\mathcal{Y},\mathcal{X})\wedge \\
 \wedge\forall i<m,\, H(y_i)=x_j\rightarrow D_{y_i}=D_{x_j}\wedge\\
 \wedge \forall i,j<m,\, E_{\mathcal{Y}}(y_i,y_j) \wedge H(y_i)=x_k\wedge H(y_j)=x_p\rightarrow \forall a\in D_{y_i}, \forall b\in D_{y_j},\\
 E^{y_iy_j}_{\ddot{\mathcal{B}}}(a,b)\leftrightarrow E^{x_kx_p}_{\ddot{\mathcal{A}}}(a,b)\wedge \forall i<m,\forall j<n,\,y_i=x_j\rightarrow H(y_j)=x_i.
\end{gathered}
\end{equation}
\end{definition}
That is, for our purpose, a covering is another instance with different $\mathcal{X'}$ and $\ddot{\mathcal{A'}}$ such that
\begin{enumerate}
  \item The domain of every variable $y_i$ in $\Theta_{\mathcal{Y}}$ is equal to the domain of $H(y_i)$ in $\Theta_{\mathcal{X}}$.
  \item There is a homomorphism from $\mathcal{Y}$ to $\mathcal{X}$ (for any constraint $(y_i,y_j;E_{\ddot{\mathcal{B}}})$) of $\Theta_{\mathcal{Y}}$, $(H(y_i),H(y_j);E_{\ddot{\mathcal{A}}})$ is a constraint of $\Theta_{\mathcal{X}}$ such that $E_{\ddot{\mathcal{A}}}$ and $E_{\ddot{\mathcal{B}}}$ here are the same relation but for different variables.
  \item If a variable $y$ appears in both $\Theta_{\mathcal{X}}$ and $\Theta_{\mathcal{Y}}$, we just assume that $H(y)=y$.
\end{enumerate}

\begin{definition}[Expanded covering]
We say that $\Theta_{\mathcal{Y}}$ is an expanded covering if
\begin{equation}
\begin{gathered}
 \hspace {0pt}ExpCov(\Theta_{\mathcal{Y}},\Theta_{\mathcal{X}})\iff \exists H<\langle b_m,b_n\rangle, HOM(\mathcal{Y},\mathcal{X})\wedge\\
 \hspace {0pt}\wedge\forall i<n,\, H(y_i)=x_j\rightarrow D_{y_i}=D_{x_j}\wedge \forall i,j<m, \\
 \big((E_{\mathcal{Y}}(y_i,y_j)\wedge H(y_i)=x_k\neq H(y_j)=x_p\rightarrow \\
 \rightarrow (\forall a\in D_{x_k}, \forall b\in D_{x_p}, E^{x_kx_p}_{\ddot{\mathcal{A}}}(a,b)\rightarrow E^{y_iy_j}_{\ddot{\mathcal{B}}}(a,b)))\wedge \\
 \hspace {0pt}\wedge((E_{\mathcal{Y}}(y_i,y_j)\wedge H(y_i)=H(y_j)=x_k\rightarrow \forall a\in D_{y_i},\,E^{y_iy_j}_{\ddot{\mathcal{B}}}(a,a))\big).
\end{gathered}
\end{equation}

\end{definition}

That is, an expanded covering is another instance with different $\mathcal{X'}$ and $\ddot{\mathcal{A'}}$ such that:
\begin{enumerate}
  \item The domain of every variable $y_i$ in $\Theta_{\mathcal{Y}}$ is equal to the domain of $H(y_i)$ in $\Theta_{\mathcal{X}}$.
  \item There is a homomorphism from $\mathcal{Y}$ to $\mathcal{X}$, but in this case:
  \begin{itemize}
  \item $\mathcal{X}$ can 'have loops'. When $H(y_i)=H(y_j)$, then we need for any $a$ in $D_{y_i}=D_{y_j}$, $(a,a)\in E_{\ddot{\mathcal{B}}}$;
  \item When $H(y_i)\neq H(y_j)$, then $E_{\mathcal{X}}(H(y_i),H(y_j))$ is an edge but $E^{y_iy_j}_{\ddot{\mathcal{B}}}$ is weaker or equivalent to $E^{H(y_i),H(y_j)}_{\ddot{\mathcal{A}}}$ (in our case it is always a richer relation of the same arity, more edges between $D_{y_i},D_{y_j}$ in $\Theta_{\mathcal{Y}}$ than between $D_{H(y_i)}, D_{H(y_j)}$) in $\Theta_{\mathcal{X}}$.
  \end{itemize}
  \item If a variable $y$ appears in both $\Theta_{\mathcal{X}}$ and $\Theta_{\mathcal{Y}}$, we just assume that $H(y)=y$.
\end{enumerate}

Then it is obvious that:
\begin{enumerate}
\item Any time we replace some constraints with weaker constraints, we get an expanded covering of the original instance: we remove some edges from $\mathcal{X}$ and add some edges to $\ddot{\mathcal{A}}$.
\item Any solution to the original instance can be naturally expanded to a solution to a covering (expanded covering): consider a homomorphism $H$ from $\mathcal{X}$ to $\ddot{\mathcal{A}}$, and a homomorphism $H'$ from $\mathcal{Y}$ to $\mathcal{X}$ and then construct $H\circ H'$ (and it will be a homomorphism from $\mathcal{Y}$ to $\ddot{\mathcal{B}}$).
\item The union (union of all constraints) of two coverings (expanded coverings) is also a covering (expanded covering): consider digraphs $\mathcal{Y}_1\cup \mathcal{Y}_2$ and $\ddot{\mathcal{B}}_1\cup \ddot{\mathcal{B}}_2$.
\item A covering (expanded covering) of a covering (expanded covering) is a covering (expanded covering): consider a superposition of homomorphism.
\item Suppose $\Theta_{\mathcal{Y}}$ is a covering (expanded covering) of a $1$-consistent instance and $\Theta_{\mathcal{Y}}$ is a tree-formula. Then the solution set to $\Theta_{\mathcal{Y}}$ is subdirect (there are no cycles in $\mathcal{Y}$).
\end{enumerate}
The following lemma can be easily proved (see \cite{zhuk2020proof}).
\begin{lemmach}[Lemma 6.1, \cite{zhuk2020proof}]
  Suppose $\Theta_{\mathcal{X}}$ is a cycle-consistent irreducible CSP instance and $\Theta_{\mathcal{Y}}$ is an expanded covering. Then $\Theta_{\mathcal{Y}}$ is cycle-consistent and irreducible.
\end{lemmach}

\subsubsection{Relations and properties}
A binary relation $R$ is called \emph{critical} if it cannot be represented as an intersection of other binary relations on $D_i\times D_j$ and it has no dummy variables. Since in our list $\Gamma_{\mathcal{A}}$ there is any invariant binary relation on $D_i\times D_j$, we define $Critical_2(R)$ as follows:
\begin{equation}
  \begin{gathered}
     Critical_2(R)\iff \neg Dum_2(R,i)\wedge \neg Dum_2(R,j)\wedge \exists a\in D_i,\exists b\in D_j,\,\forall k<2^{l^2},\\
     \hspace {0pt}R\subsetneq\Gamma^2_{\mathcal{A},k}\rightarrow(\Gamma^2_{\mathcal{A},k}(a,b)\wedge \neg R(a,b)).
  \end{gathered}
\end{equation}
For a critical binary relation $R$, the minimal relation $R'$ such that $R\subsetneq R'$ is called the \emph{cover of} $R$:
\begin{equation}
\begin{gathered}
  Cover_2(R',R)\iff Critical_2(R)\wedge R'=weakening(R).
\end{gathered}
\end{equation}
Further notions we will consider in connection to both binary and $n$-ary relations, so we will define them both for constraints and solution sets $\mathscr{R}_{\Theta}$. We use constant subscripts to highlight the difference between the definitions, but we do not use $n$ in subscripts for higher arity since the definitions do not depend on variable $n$. For a congruence $\sigma$ on $D_i$ we say that the $i$th variable of a unary relation $E\leq D_i$ and a binary relation $R\leq D_i\times D_j$ is \emph{stable} under $\sigma$ if
  \begin{equation}
  \begin{gathered}
   \hspace{0pt}Stable_1(E,\sigma)\iff \forall a,a'\in D_i,\,E(a)\wedge \sigma(a,a')\rightarrow E(a');\\
     Stable_2(R,i,\sigma)\iff \forall a,a'\in D_i,\forall b\in D_j, R(a,b)\wedge \sigma(a,a')\rightarrow R(a',b).
  \end{gathered}
  \end{equation}
\begin{remark}
  Note that a unary relation $E$ stable under some congruence $\sigma$ on $D$ is just a union of that congruence blocks, it does not have to be a subuniverse of $D$. A binary relation $R$ such that both its variables on $D$ are stable under $\sigma$ is a full relation between some set of congruence blocks on the first variable and some (not necessarily the same) set of congruence blocks on the second variable. A congruence $\sigma\subseteq D\times D$ is stable under itself, in the sense that all elements from one congruence block on the first coordinate are connected with all elements from the same block on the second coordinate.
\end{remark}
We say that the $i$th variable of the solution set $\mathscr{R}_{\Theta}\leq D_{0}\times ...\times D_{n-1}$ is stable under congruence $\sigma$ on $D_i$ if
\begin{equation}
  \begin{gathered}
     Stable(\mathscr{R}_{\Theta}, i,\sigma)\iff\forall H, H'\leq \langle n,\langle n,l\rangle\rangle, \forall a_i, a_i'\in D_i,\big(\sigma(a_i,a_i')\wedge(\forall j\neq i,\\
      \hspace {0pt}H(j)=H'(j))\wedge H(i)=\langle i,a_i \rangle\wedge H'(i)=\langle i,a'_i \rangle\wedge\ddot{HOM}(\mathcal{X},\ddot{\mathcal{A}},H)\big)\rightarrow\\
     \hspace {0pt}\rightarrow \ddot{HOM}(\mathcal{X},\ddot{\mathcal{A}},H').
  \end{gathered}
\end{equation}
Note that this is $\Pi^{1,b}_1$-formula. If every variable of $R$ or $\mathscr{R}_{\Theta}$ is stable under $\sigma$ we say that $R$ or $\mathscr{R}_{\Theta}$ is \emph{stable} under $\sigma$ and write $Stable(\mathscr{R}_{\Theta},\sigma)$. We say that a binary relation $R\leq D_i\times D_j$ \emph{has a parallelogram property} if
\begin{equation}
  ParlPr_2(R)\iff\forall a,c\in D_i,\forall b,d\in D_j, R(a,b)\wedge R(c,b)\wedge R(c,d)\rightarrow R(a,d).
\end{equation}
A relation has the parallelogram property if any way of grouping its coordinates into two groups gives a binary relation with the parallelogram property. That is, we say that a solution set $\mathscr{R}_{\Theta}\leq D_{0}\times ...\times D_{n-1}$ to some CSP instance $\Theta$ has a \emph{parallelogram property} if the following $\Pi^{1,b}_2$-relation holds:
\begin{equation}
  \begin{gathered}
 \hspace {0pt}ParlPr(\mathscr{R}_{\Theta})\iff \forall V_1,V_2<n,\,(\forall i<n,\, V_1(i)\leftrightarrow \neg V_2(i))\wedge \\
 \wedge \forall H_1,H_2,H_3\leq \langle n,\langle n,l\rangle\rangle,\\
     \hspace {0pt}\big( \ddot{HOM}(\mathcal{X},\ddot{\mathcal{A}},H_1)\wedge \ddot{HOM}(\mathcal{X},\ddot{\mathcal{A}},H_2)\wedge \ddot{HOM}(\mathcal{X},\ddot{\mathcal{A}},H_3)\wedge\\
     \wedge (\forall i\in V_1,\, H_3(i)=H_2(i)\wedge \forall j\in V_2,\, H_3(j)=H_1(j))\big)\rightarrow \\
     \rightarrow \exists H_4\leq \langle i,\langle i,a\rangle\rangle,\, \ddot{HOM}(\mathcal{X},\ddot{\mathcal{A}},H_4)\wedge\\
     \hspace {0pt}\wedge (\forall i\in V_1,\, H_4(i)=H_1(i)\wedge \forall j\in V_2,\, H_4(j)=H_2(j)).
  \end{gathered}
\end{equation}
For a binary relation $R\leq D_i\times D_j$ by $Con_2^{(R,i)}$ we denote the following relation:
\begin{equation}
\begin{gathered}
 Con_2^{(R,i)}(a,a')\iff \exists b\in D_j, R(a,b)\wedge R(a',b)\\
 Con_2^{(R,j)}(b,b')\iff \exists a\in D_i, R(a,b)\wedge R(a,b').
\end{gathered}
\end{equation}
For a constraint $C=(x_i,x_j;R)$ we will denote $Con_2^{(R,i)}$ by $Con_2^{(C,i)}$. For a set of constraints $\Theta$ we denote by $Con_2^{(\Theta,i)}$ the set of all $Con_2^{(C,i)}$. In the case of a CSP instance $\Theta$, for any $i<n$ this set is of the form
\begin{equation}
\begin{gathered}
   \forall j<n,\forall a,b<l,\,Con_2^{(\Theta,i)}(0,j,a,b)\iff E_{\mathcal{X}}(i,j)\wedge Con_2^{(E_{\mathcal{X}}(i,j),i)}(a,b),\\
   \forall j<n,\forall a,b<l,\,Con_2^{(\Theta,i)}(j,0,a,b) \iff E_{\mathcal{X}}(j,i)\wedge Con_2^{(E_{\mathcal{X}}(j,i),i)}(a,b).
\end{gathered}
\end{equation}
and its size is bounded by $\langle n,n,l,l\rangle $. We say that the $i$th variable of the binary relation $R$ is \emph{rectangular} if 
\begin{equation}
\begin{gathered}
   RectPr_2(R,i)\iff \forall a,a'\in D_i,\forall b\in D_j,\\
   \hspace {0pt}(Con_2^{(R,i)}(a,a')\wedge R(a,b)\rightarrow R(a',b)).
\end{gathered}
\end{equation}
For a solution set $\mathscr{R}_{\Theta}\leq D_{0}\times ...\times D_{n-1}$ to some CSP instance $\Theta$ by $Con^{([\mathscr{R}_{\Theta}],i)}$ we define the binary relation
\begin{equation}
  \begin{gathered}
 \hspace{3pt}Con^{([\mathscr{R}_{\Theta}],i)}(a,a')\iff \exists H_1,H_2\leq \langle n,\langle n,l\rangle\rangle,\, H_1(i)=a\wedge H_2(i)=a'\wedge\\
      \wedge\forall j\neq i<n,\,H_1(j)=H_2(j) \wedge\ddot{HOM}(\mathcal{X},\ddot{\mathcal{A}},H_1)\wedge \ddot{HOM}(\mathcal{X},\ddot{\mathcal{A}},H_2).
  \end{gathered}
\end{equation}
We say that the $i$th variable of the solution set $R$ is \emph{rectangular} if
\begin{equation}
\begin{gathered}
 \hspace {0pt}RectPr(\mathscr{R}_{\Theta},i)\iff \forall a,a'\in D_i,\, \forall H_1\leq \langle n,\langle n,l\rangle\rangle,\\
   \ddot{HOM}(\mathcal{X},\ddot{\mathcal{A}},H_1)\wedge H_1(i)=\langle i,a\rangle \wedge Con_n^{([\mathscr{R}_{\Theta}],i)}(a,a') \rightarrow \exists H_2\leq \langle n,\langle n,l\rangle\rangle,\\
    \hspace {0pt}\ddot{HOM}(\mathcal{X},\ddot{\mathcal{A}},H_2)\wedge H_2(i)=\langle i,a'\rangle \wedge(\forall j\neq i<n,\, H_1(j)=H_2(j)).
\end{gathered}
\end{equation}
Finally, we say that the solution set $\mathscr{R}_{\Theta}$ to $\Theta$ is \emph{rectangular} if all its variables are rectangular:
\begin{equation}
  RectInst(\mathscr{R}_{\Theta})\iff \forall i<n,\,RectPr(\mathscr{R}_{\Theta},i).
\end{equation}
Note that $RectPr(\mathscr{R}_{\Theta},i)$ is $\Sigma^{1,b}_2$-formula.
\begin{remark}
 Note that the parallelogram property implies rectangularity, and if $i$th coordinate of the relation $R$ is rectangular, then $Con^{([\mathscr{R}_{\Theta}],i)}$ is a congruence.
\end{remark}
A binary relation $R\leq D_i\times D_j$ is called \emph{essential} if it cannot be represented as a conjunction of relations with smaller arities. A pair $(a_i,a_j)\in D_i\times D_j$ is called \emph{essential for $R$} if
\begin{equation}
\begin{gathered}
 EssPair(a_i,a_j,R)\iff \neg R(a_i,a_j)\wedge \exists b_i\in D_i,\exists b_j\in D_j,\, \\
 \hspace {0pt}R(a_i,b_j)\wedge R(b_i,a_j).
\end{gathered}
\end{equation}
It is known \cite{unknown} that for a relation $R$ being an essential is equivalent to having an essential pair. Thus, we can define an essential binary relation $R$ as follows:
\begin{equation}
\begin{gathered}
 EssRel_2(R)\iff \exists a_i\in D_i,\exists a_j\in D_j,\,EssPair(a_i,a_j,R).
\end{gathered}
\end{equation}
For a solution set $\mathscr{R}_{\Theta}$ we define an essential tuple by the following $\Sigma^{1,b}_1$-formula: 
\begin{equation}
\begin{gathered}
 EssTuple(H,\mathscr{R}_{\Theta})\iff \neg \ddot{HOM}(\mathcal{X},\ddot{\mathcal{A}},H)\wedge \forall i<n,\exists b<l,\exists H'\leq \langle n,\langle n,l\rangle\rangle,\\
 \ddot{HOM}(\mathcal{X},\ddot{\mathcal{A}},H')\wedge H'(i) = \langle i, b\rangle\wedge \forall j\neq i<n,\, H'(j) = H(j). 
\end{gathered}
\end{equation}
Thus, $\mathscr{R}_{\Theta}$ is essential if there exists an essential tuple.
\begin{equation}
  \begin{gathered}
     EssRel(\mathscr{R}_{\Theta})\iff \exists H\leq \langle n,\langle n,l\rangle\rangle,\,EssTuple(H,\mathscr{R}_{\Theta}).
  \end{gathered}
\end{equation}
We say that a relation $\mathscr{R}\leq D_{0}\times ...\times D_{n-1}$ is $(C_0,...,C_{n-1})$-essential if $\mathscr{R}\cap (C_0,...,C_{n-1})=\emptyset$, but for every $i\leq k$, $\mathscr{R}\cap (C_0,...,C_{i-1},D_i,C_{i+1},...,C_{n-1})\neq \emptyset$. We can formalize the tuple $(C_0,...,C_{n-1})$ as usual, by one set $C(i,a)\iff C_i(a)$.
\begin{equation}
  \begin{gathered}
 \hspace {0pt}EssRel(\mathscr{R},C)\iff \neg (\exists H\in \mathscr{R},\,\forall i<n, \exists c_i\in C_i,\,\,H(i)=\langle i,c_i\rangle)\wedge \\
     \forall i<n,\exists H\in \mathscr{R},\exists a_i\in D_i\backslash C_i,\,H(i)=\langle i,a_i\rangle\wedge \\
     \wedge \forall j\neq i, j<n,\,\exists c_j\in C_j,\,\,H(j)=\langle j,c_j\rangle.
  \end{gathered}
\end{equation}
This is $\Sigma^{\mathscr{B}}_0$-formula, but if we restrict ourselves to solution sets, we get a Boolean combination of $\Sigma^{1,b}_1$ and $\Pi^{1,b}_1$ formulas. 

Finally, to define a \emph{key relation}, we first present a \emph{unary vector function that preserves the relation}. Suppose $R\leq D_0\times...\times D_{s-1}$ and define a tuple $\Psi=(\psi_0,...,\psi_{s-1})$, where $\psi_i:D_i\to D_i$, is called a unary-vector function. We say that $\psi$ preserves $R$ if $(\psi_0(a_0),...,\psi_{s-1}(a_{s-1}))\in R$ for every $(a_0,...,a_{s-1})\in R$. We say that $R$ is a \emph{key relation} if there exists a tuple $(b_0,...,b_{s-1})\notin R$ such that for any tuple $(c_0,...,c_{s-1})\notin R$ there exists a vector function $\Psi$ which preserves $R$ and gives $\psi_i(c_i)=b_i$ for any $i<s$. For a binary relation $R\leq D_i\times D_j$ there is a pair of unary functions $\psi_i,\psi_j$, represented by two-dimensional sets, such that:\begin{equation}
  \begin{gathered}
     VecFun_2(R,\psi_i,\psi_j)\iff MAP(D_i,l,D_i,l,\psi_i)\wedge MAP(D_j,l,D_j,l,\psi_j)\wedge\\
     \hspace{0pt}\wedge\forall a_i,b_i\in D_i,\forall a_j,b_j\in D_j,\,R(a_i,a_j)\wedge \psi_i(a_i,b_i)\wedge \psi_j(a_j,b_j)\rightarrow R(b_i,b_j). 
  \end{gathered}
\end{equation}
Obviously, both $\psi_i,\psi_j$ are polymorphisms. We say that a binary relation $R$ is a \emph{key relation} if there exists a tuple $(b_i,b_j)\notin R$ such that for every $(c_i,c_j)\notin R$ there exists a unary vector function represented by sets $\psi_i,\psi_j$ that preserves $R$ and $\psi(c_i,b_i)$ and $\psi_j(c_j,b_j)$:
\begin{equation}
  \begin{gathered}
     \hspace{0pt}KeyRel_2(R)\iff \exists b_i,b_j<l,\forall c_i,c_j<l\,\\
     \neg R(b_i,b_j)\wedge \neg R(c_i,c_j)\rightarrow \bigvee_{\psi_i,\psi_j < l^2} \,VecFun_2(R,\psi_i,\psi_j)\wedge\psi_i(c_i,b_i)\wedge \psi_j(c_j,b_j).
  \end{gathered}
\end{equation}
Note that already for binary relations it would be $\Sigma^{1,b}_1$-formula if we do not fix the algebra $\mathbb{A}$. But in our case, we can go through all possible endomorphisms on $D_i,D_j$. For a solution set $\mathscr{R}_{\Theta}\leq D_{0}\times ...\times D_{n-1}$ we can represent a unary vector function as a three-dimensional set $\Psi(i,a,b)$, where each $\Psi_i$ represents a function from $D_i$ to $D_i$. Consider the following $\Pi^{1,b}_1$-formula:
\begin{equation}\label{laalks34523}
  \begin{gathered}
 \hspace{0pt}VecFun(\mathscr{R}_{\Theta},\Psi) \iff \forall i<n,\,MAP(D_i,l,D_i,l,\Psi_i)\wedge\\
 \wedge \forall H,H'\leq \langle n,\langle n,l\rangle\rangle,\ddot{HOM}(\mathcal{X},\ddot{\mathcal{A}},H)\wedge\\
     \wedge (\forall i<n,\forall a_i,b_i\in D_i,\,H(i)=\langle i,a_i\rangle\wedge \Psi_i(a_i,b_i)\rightarrow H'(i)=\langle i,b_i\rangle  )\rightarrow \\
     \rightarrow\ddot{HOM}(\mathcal{X},\ddot{\mathcal{A}},H').
  \end{gathered}
\end{equation} 
Then for a solution set $\mathscr{R}_{\Theta}\leq D_{0}\times ...\times D_{n-1}$ we have the following $\Sigma^{1,b}_4$-formula:
\begin{equation}
  \begin{gathered}
 \hspace {0pt}KeyRel(\mathscr{R}_{\Theta})\iff \exists H\leq \langle n,\langle n,l\rangle\rangle, \forall H'\leq \langle n,\langle n,l\rangle\rangle,\,\neg \ddot{HOM}(\mathcal{X},\ddot{\mathcal{A}},H)\wedge \\
 \wedge \neg \ddot{HOM}(\mathcal{X},\ddot{\mathcal{A}},H')\rightarrow\exists \Psi\leq \langle n,l,l\rangle,\,VecFun(\mathscr{R}_{\Theta},\Psi)\wedge \\
 \wedge (\forall i<n,\forall a_i,b_i<l,\,\,H'(i)=\langle i,a_i\rangle\wedge \Psi_i(a_i,b_i) \rightarrow H(i)=\langle i,b_i\rangle).
  \end{gathered}
\end{equation}

\subsubsection{Bridges and connectivity}

\begin{definition}[Irreducible congruence] 
We say that a congruence $\sigma$ on $D_i$ is \emph{irreducible} if it is proper, and it cannot be represented as an intersection of other binary relations stable under $\sigma$.
\begin{equation}
  \begin{gathered}
  \hspace{0pt}irCong_m(\sigma,D_i)\iff pCong_m(D_i,\Omega,\sigma)\wedge\exists a,b\in D_i, \forall j<2^{l^2},\\
     \hspace {0pt}\sigma\subsetneq\Gamma^2_{\mathcal{D},j}\wedge Stable_2(\Gamma^2_{\mathcal{D},j},\sigma)\rightarrow(\Gamma^2_{\mathcal{D},j}(a,b)\wedge \neg \sigma(a,b)).
  \end{gathered}
\end{equation}
We denote the set of all irreducible congruences on $D$ by $\Sigma^{ir}_{\mathcal{D}}$. For an irreducible congruence $\sigma$ on set $D$ by $\sigma^*$ is denoted the minimal binary relation $\sigma\subsetneq\sigma^*$ stable under $\sigma$. We can define a string function 
\begin{equation}
  \begin{gathered}
     \cdot^*(\sigma)(a,b)=\sigma^*(a,b)\iff \bigvee_{\sigma'\leq \langle l,l\rangle} Stable_2(\sigma',\sigma)\wedge \sigma\subsetneq\sigma'\wedge\forall j<2^{l^2},\\
    \hspace {0pt}Stable_2 (\Gamma^2_{\mathcal{D},j},\sigma)\wedge \sigma\subsetneq \Gamma^2_{\mathcal{D},j}\rightarrow \sigma'\subseteq \Gamma^2_{\mathcal{D},j}\wedge \sigma(a,b).
  \end{gathered}
\end{equation}
\end{definition}

\begin{remark}
Any congruence $\sigma'$ containing $\sigma$ is stable under $\sigma$, but a binary relation stable under $\sigma$ does not need to be a congruence. 
\end{remark}
\begin{definition}[Bridge]
  For two domains $D_i,D_j$ and congruences on them $\sigma_i,\sigma_j$ respectively, we say that a $4$-ary relation $\rho\subseteq D_i^2\times D^2_j$ is a \emph{bridge from $\sigma_i$ to $\sigma_j$} if the first two variables of $\rho$ are stable under $\sigma_1$ and the last two variables of $\rho$ are stable under $\sigma_2$, $\sigma_1\subsetneq pr_{1,2}(\rho)$ and $\sigma_2\subsetneq pr_{3,4}(\rho)$, and $(a_1,a_2,a_3,a_4)\in \rho$ implies 
$$(a_1,a_2)\in \sigma_1\iff (a_3,a_4)\in \sigma_2.
$$
We can define it by $\Sigma^{1,b}_0$-formula:
    \begin{equation}
    \begin{gathered}
 \hspace {0pt}Bridge(\rho,\sigma_i,\sigma_j)\iff (\exists a,a'\in D_i,\exists b,b'\in D_j,\\
 \hspace{0pt}pr_{1,2}(\rho)(a,a')\wedge \neg\sigma_i(a,a')\wedge pr_{3,4}(\rho)(b,b')\wedge \neg\sigma_j(b,b'))\wedge\\
 \wedge Stable_2(\rho,1,\sigma_i)\wedge Stable_2(\rho,2,\sigma_i)\wedge Stable_2(\rho,3,\sigma_j)\wedge Stable_2(\rho,4,\sigma_j)\wedge \\
 \hspace {0pt}\wedge (\forall a,a'\in D_i,\forall b,b'\in D_j,\ \rho(a,a',b,b')\rightarrow (\sigma_i(a,a')\leftrightarrow \sigma_j(b,b'))).
 \end{gathered}
  \end{equation}
In words, the projection of $\rho$ to the first two coordinates strictly contains $\sigma_i$ and is a full relation between some set of congruence blocks on the first coordinate and some set of blocks on the second coordinate, and the same for the projection of $\rho$ to the last two coordinates, and the first two coordinates contain elements from one congruence block of $\sigma_i$ if and only if the last two coordinates also contain elements from one congruence block of $\sigma_j$. A bridge $\rho\subseteq D^4$ is called \emph{reflexive} if $(a,a,a,a)\in \rho$ for every $a\in D$. For a bridge $\rho$, denote by $\tilde{\rho}$ the binary relation defined by $\tilde{\rho}(x,y)=\rho(x,x,y,y)$, we define it as a string function:
\begin{equation}
  \tilde{\cdot}(\rho)(x,y) = \tilde{\rho}(x,y)\iff \rho(x,x,y,y).
\end{equation}
A reflexive bridge $\rho$ from an irreducible congruence $\sigma_i$ to an irreducible congruence $\sigma_j$ is
called \emph{optimal} if there is no a reflexive bridge $\rho'$ from $\sigma_i$ to $\sigma_j$ such that $\tilde{\rho}\subsetneq \tilde{\rho}'$, i.e. a bridge that contains more congruence blocks than $\rho$. 
\begin{equation}
  \begin{gathered}
     OptBridge(\rho,\sigma_i,\sigma_j)\iff irCong_m(\sigma_i,D)\wedge irCong_m(\sigma_j,D)\wedge \\
     \wedge\neg (\bigvee_{\rho'\leq (4l)^{2^4}} Bridge(\rho',\sigma_i,\sigma_j)\wedge \forall a\in D,\,\rho'(a,a,a,a)\wedge \tilde{\rho}\subsetneq \tilde{\rho}').
  \end{gathered}
\end{equation}
If $\rho$ is optimal, then $\tilde{\rho}$ is a congruence. For an irreducible congruence $\sigma$, define a string function $opt$ as 
\begin{equation}
  opt(\sigma)(x,y) \iff \bigvee_{\rho\leq (4l)^{2^4}} OptBridge(\rho,\sigma,\sigma)\wedge \tilde{\rho}(x,y).
\end{equation}
It returns the congruence $\tilde{\rho}$ for an optimal bridge $\rho$ from $\sigma$ to $\sigma$, which is well-defined since we can compose two reflexive bridges. For a set of irreducible congruences $\Sigma^{ir}_{\mathcal{D}}$, we define a string function $optset$ that returns the set of $opt(\sigma)$ for all $\sigma\in \Sigma^{ir}_{\mathcal{D}}$:
\begin{equation}
  \begin{gathered}
     optset(\Sigma^{ir}_{\mathcal{D}})(i,a,b)\iff \Sigma^{ir}_{\mathcal{D},i}\neq \emptyset\wedge opt(\Sigma^{ir}_{\mathcal{D},i})(a,b). 
  \end{gathered}
\end{equation}
\end{definition}

We say that two congruences $\sigma_i,\sigma_j$ on a set $D$ are \emph{adjacent} if there exists a reflexive bridge from $\sigma_i$ to $\sigma_j$. Since we consider only finite and fixed set of binary constraints $\Gamma^2_{\mathcal{A}}$, including the set of all congruences on $A$ and all its subuniverses, we know in advance all bridges for all congruences, the list denoted by $\Xi$:
\begin{equation}
  \begin{gathered}
     Adj(\sigma_i,\sigma_j)\iff \bigvee_{\rho\,\in \,\Xi}Bridge(\rho,\sigma_i,\sigma_j)\wedge \forall a\in D,\,\rho(a,a,a,a).
  \end{gathered}
\end{equation}
Note that $Adj(\sigma_i,\sigma_j)$ is $\Sigma^{1,b}_0$-formula. We say that two rectangular constraints $C_1,C_2$ are \emph{adjacent} in a common variable $x$ if $Con_2^{(C_1,x)}$ and $Con_2^{(C_2,x)}$ are adjacent.
A formula is called \emph{connected} if every constraint in the formula is critical and rectangular, and the graph, whose vertices are constraints and edges are adjacent constraints, is connected. To define connectivity, recall that there is a path from $i$ to $j$ in the input digraph $\mathcal{X}$ if there exists a path $\mathcal{P}_t$ of some length $t$ that can be homomorphically mapped to $\mathcal{X}$ such that $H(0)=i$ and $H(t)=j$. For an instance $\Theta$, we define the following $\Sigma^{1,b}_1$-relation of being connected:
\begin{equation}
  \begin{gathered}
     Connected(\Theta)\iff\forall i,j<n,\,E_{\mathcal{X}}(i,j)\rightarrow Critical_2(E^{ij}_{\ddot{\mathcal{A}}})\wedge \\
     \wedge RectPr_2(E^{ij}_{\ddot{\mathcal{A}}},i)\wedge RectPr_2(E^{ij}_{\ddot{\mathcal{A}}},j)\wedge \\
     \hspace {0pt}\wedge \forall i,j,k,s<n,\,E_{\mathcal{X}}(i,j)\wedge E_{\mathcal{X}}(k,s) \rightarrow \exists \mathcal{P}_t< \langle n,4n^2\rangle,\exists H\leq \langle t,n \rangle,\,PATH(V_{\mathcal{P}_t}) \wedge\\
 \hspace {0pt}\wedge HOM(\mathcal{P}_t,\mathcal{X},H)\wedge(H(0,i)\wedge H(t,s)) \wedge \\
 \hspace {0pt}\wedge \forall p\leq t-2,\, Adj(Con_2^{(E_{\mathcal{X}}(H(p),H(p+1)), H(p+1))},Con_2^{(E_{\mathcal{X}}(H(p+1),H(p+2)),H(p+1))}),
  \end{gathered}
\end{equation}
where by $E_{\mathcal{X}}(H(p),H(p+1)))$ and $E_{\mathcal{X}}(H(p+1),H(p+2))$ we abbreviate all four combinations of non-symmetric constraints.

\subsection{One-of-four subuniverses}\label{alalo98}
In this section, we will define $4$ different subuniverses for an algebra $\mathbb{D}=(D,\Omega)$. For $D$ being a subuniverse for a fixed algebra $\mathbb{A}=(A,\Omega)$, all these definitions are $\Sigma^{1,b}_0$-formulas. The part of the content of the section repeats the formalization previously introduced in \cite{gaysin2023proof}. We cannot omit it here due to the crucial importance of these concepts for the understanding of the paper.

\subsubsection{Binary absorption subuniverse}
For the fixed algebra $\mathbb{A}=(A,\Omega)$, Due to Theorem \ref{fjjduh87} $Clone(\Omega)=Pol(\Gamma_{\mathcal{A}})$. Thus, for any binary term operation $T$ over $A$ the condition $T\in Clone(\Omega)$ can be encoded as: 
\begin{equation}
T\in Clone(\Omega) \iff Pol_{2}(T,A, \Gamma_{\mathcal{A}}).
\end{equation}
For any three sets $D,B,T$ the following $\Sigma^{1,b}_0$-definable relation indicates that the subset $B$ absorbs $D$ with binary operation $T$.
\begin{equation}
  \begin{gathered}
     BAsubS(B,D,T)\iff subS(B,D)\wedge \forall a\in D, \forall b\in B, \exists c_1,c_2\in B,\\
 \hspace {0pt}T(a,b)=c_1\wedge T(b,a)=c_2.
  \end{gathered}
\end{equation}
If we want to define a subuniverse, then 
\begin{equation}
  \begin{gathered}
     BAsubU(B,D,T,\Omega)\iff SwNU(\Omega,B)\wedge Pol_2(T,D,\Gamma_{\mathcal{A}})\wedge\\
     \hspace {0pt}\wedge BAsubS(B,D,T).
  \end{gathered}
\end{equation}
Recall that a binary absorbing subuniverse can be trivial, i.e. $B=D$.

\subsubsection{Central subuniverse}
To define a central subuniverse $C$ of an algebra $\mathbb{A}=(A,\Omega)$ we need to encode a set $Sg$ for the subset $X=\{\{a\}\times C,C\times \{a\}\}$ of $A^2$ for any $a\in A$. Recall that $Sg(X)$ can be constructed by the closure operator 
\begin{equation}
\begin{gathered}
 \hspace{0pt}Cl(X)=X\cup\{\Omega(a_1,...,a_m): a_1,...,a_m\in X\}\\
 \hspace{0pt}\forall t\geq 0, Cl^0(X)=X, Cl^{t+1}(X) = Cl(Cl^t(X)).
\end{gathered}
\end{equation}
Since $\mathbb{A}$ is finite of size $l$ and $|X|=2|C|$, we do not need more than $(l^2-2|C|)$ applications of the closure operator $Cl$ because at any application we either add to the set at least one element or, after some $t$, $Cl^t(X)=Cl^{t+r}(X)$ for any $r$. Not to depend on $C$, let us choose the value $l^2$. Thus, for any set $X\leq \langle l,l\rangle$, we will iteratively define the following set $Cl^{l^2}_X$ up to $l^2$:
\begin{equation}
\begin{gathered}
 \hspace{0pt}\forall b,c<l,\,Cl^{0}_X(b,c)\iff X(b,c)\wedge\\
 \hspace{0pt}\wedge \forall 0<t<l^2, \forall b,c<l,\, Cl^{t}_X(b,c)\iff Cl^{t-1}_X(b,c)\vee\\
 \vee\exists b_1,...,b_m,c_1,...,c_m\in A,Cl^{t-1}_X(b_1,c_1)\wedge...\wedge Cl^{t-1}_X(b_m,c_m)\wedge\\
 \hspace {0pt}\wedge\Omega(b_1,...,b_m)=b\wedge \Omega(c_1,...,c_m)=c.
\end{gathered}
\end{equation}
The existence of this set follows from $\Sigma^{1,b}_1$-induction. A central subuniverse must be an absorbing subuniverse, namely, a ternary absorbing subuniverse \cite{zhuk2020strong}. For any three sets $D,C,S$ the following $\Sigma^{1,b}_0$-definable relation ($D$ and $C$ are bounded by $l$) expresses that the subset $C$ of $D$ is central under ternary term operation $S$:
\begin{equation}
  \begin{gathered}
 \hspace{0pt}CRsubS(C,D,S)\iff subS(C,D)\wedge\forall c_1,c_2\in C, \forall a\in D, \exists c'_1,c'_2,c'_3\in C,\\
 \hspace{0pt}S(c_1,c_2,a)=c'_1\wedge S(c_1,a,c_2)=c'_2\wedge S(a,c_1,c_2)=c'_3\wedge\\
 \hspace {0pt}\wedge\bigwedge_{a\in D\backslash C} \bigwedge_{X<\langle l,l\rangle} ((X(a,c)\wedge X(c,a)\leftrightarrow c\in C) \rightarrow \neg Cl^{l^2}_X(a,a)).
  \end{gathered}
\end{equation}
Note that for not fixed algebra $\mathbb{B}=(B,\Omega)$, this relation is $\Pi^{1,b}_1$ since the size of $B$ would not be bounded, and therefore we could not use large conjunction. If we want to define a subuniverse, then 
\begin{equation}
  \begin{gathered}
 \hspace{0pt}CRsubU(C,D,S,\Omega)\iff SwNU(\Omega,C)\wedge Pol_2(S,D,\Gamma_{\mathcal{A}})\wedge\\
 \hspace {0pt}\wedge CRsubS(C,D,S).
  \end{gathered}
\end{equation}
Recall that a central subuniverse can be trivial, i.e. $C=D$.

\subsubsection{PC subuniverse}\label{'a'a'as;ldfkjg}
Polynomially complete algebras are necessarily simple \cite{ISTINGER1979103}, i.e. they have no non-trivial congruence relations. The \emph{ternary discriminator function} is the function $t$ defined by the identities
\[  
t(x,y,z) = 
   \begin{cases}
     z, \,x=y, \\ 
     x, \,x\neq y.
   \end{cases}
\]
Then Theorem \ref{slsl88yhdurh} gives a necessary and sufficient condition of polynomial completeness. 
\begin{theorem}[\cite{Brady2022NotesOC}]\label{slsl88yhdurh}
A finite algebra is polynomially complete if and only if it has the ternary discriminator as a polynomial operation.
\end{theorem}
The clone of all polynomials over $\mathbb{D}$, $Polynom(\mathbb{D})$ is defined as the clone generated by $\Omega$ and all constants on $D$, i.e. nullary operations:
\begin{equation}
Polynom(\mathbb{D}) = Clone(\Omega,a_1,...,a_{|D|}).
\end{equation}
Constants as nullary operations with constant values, composed with $0$-many $n$-ary operations are $n$-ary operations with constant values. Thus, to be preserved by all constant operations, any unary relation has to contain the entire set $D$, and any binary relation has to contain the diagonal relation $\Delta_{D}$. For the algebra $\mathbb{D}=(D,\Omega)$ denote by $\Gamma^{diag}_{\mathcal{D}} = (\Gamma^{1,diag}_{\mathcal{D}}, \Gamma^{2,diag}_{\mathcal{D}})$ the pair of sets such that
\begin{equation}\label{'''''''''''l}
\begin{gathered}
 \hspace {0pt}\Gamma^{1,diag}_{\mathcal{D}}(j,a)\iff \Gamma^{1}_{\mathcal{D}}(j,a)\wedge (\forall b\in A, \Gamma^{1}_{\mathcal{D}}(j,b))\\
 \Gamma^{2,diag}_{\mathcal{D}}(i,a,b)\iff \Gamma^{2}_{\mathcal{D}}(i,a,b)\wedge (\forall c\in A, \Gamma^{2}_{\mathcal{D}}(j,c,c)).
\end{gathered}
\end{equation}
Note that for some $i,j$, $\Gamma^{1,diag}_{\mathcal{D},j}$ and $\Gamma^{2,diag}_{\mathcal{D},i}$ are empty sets. An $n$-ary operation $P$ on algebra $\mathbb{D}$ is a polynomial operation if it is a polymorphism for relations from $\Gamma^{diag}_{\mathcal{D}}$, i.e.
\begin{equation}
P\in Polynom(\mathbb{D})\iff Pol_{n}(P,D,\Gamma^{diag}_{\mathcal{D}}).
\end{equation}
For any two sets $D$ and $P$ the following $\Sigma^{1,b}_0$-definable relation claims that $P$ is a ternary discriminator on $D$.
\begin{equation}
  \begin{gathered}
     \hspace{0pt}PCD(D,P)\iff \forall a,b,c\in D, \\
     (a=b\wedge P(a,b,c)=c)\vee(a\neq b\wedge P(a,b,c)=a).
  \end{gathered}
\end{equation}
To formalize a PC subuniverse we need the following definition.
\begin{definition}[Polynomially complete algebra] We say that an algebra $\mathbb{D}=(D,\Omega)$ is polynomially complete if
\begin{equation}
  \begin{gathered}
 PCA(D,\Omega)\iff \bigvee_{P\in \Pi^{3}_{\mathcal{D}}} Pol_{3}(P,D,\Gamma^{diag}_{\mathcal{D}})\wedge PCD(D,P).
  \end{gathered}
\end{equation}
\end{definition}

\begin{definition}[Polynomially complete algebra without a non-trivial binary absorbing or central subuniverse] We say that an algebra $\mathbb{D}=(D,\Omega)$ is an algebra \emph{without a non-trivial binary absorbing or central subuniverse} if it satisfies the following $\Sigma^{1,b}_0$-definable relation:
\begin{equation}
  \begin{gathered}
 \hspace {0pt}subTA_{\neg BACR}(D,\Omega)\iff subTA(D,A,\Omega)\wedge \\
 \wedge\bigwedge_{B<l}\bigwedge_{T<(3l)^{2^3}} PsubS(B)\wedge Pol_2(T,D, \Gamma_{\mathcal{D}})\rightarrow\neg BAsubU(B,D,T)\wedge\\
 \hspace {0pt}\wedge\bigwedge_{C<l}\bigwedge_{S<(4l)^{2^4}} PsubS(C)\wedge Pol_3(S,D, \Gamma_{\mathcal{D}})\rightarrow\neg CRsubU(C,D,S).
  \end{gathered}
\end{equation}
Note that for a not fixed algebra $\mathbb{B}=(B,\Omega)$, this relation would be $\Pi^{1,b}_2$ since the size of $B$ would not be bounded, and therefore we could not use big conjunctions for sets and $\neg CRsubU(C,D,S)$ would be $\Sigma^{1,b}_1$-formula. We say that an algebra is polynomially complete algebra without a non-trivial binary absorbing or central subuniverse if
\begin{equation}\label{aaaaaaaaa444444}
  \begin{gathered}
 \hspace{0pt}PCA_{\neg BACR}(D,\Omega)\iff PCA(D,\Omega)\wedge subTA_{\neg BACR}(D,\Omega).
  \end{gathered}
\end{equation}
\end{definition}
\begin{definition}[PC congruence]
We say that a set $\sigma<\langle l,l\rangle$ is a PC congruence on algebra $\mathbb{D}=(D,\Omega)$ of size bounded by $l$ if
\begin{equation}
  \begin{gathered}
 PCCong_m(D,\Omega,\sigma)\iff PCA_{\neg BACR}(D/\sigma,\Omega/\sigma).
  \end{gathered}
\end{equation}
\end{definition}
Note that in this definition we apply notions from (\ref{'''''''''''l}) to (\ref{aaaaaaaaa444444}) to algebra $(D/\sigma,$ $\Omega/\sigma)$ and relations from $\Gamma_{\mathcal{D}}/\sigma$, recall (\ref{allalask75575764}).

Recall that for algebra $\mathbb{A}=(A,\Omega)$ we denoted the set of all congruences on $\mathbb{A}$ and all its subuniverses by $\Sigma_{\mathcal{A}}$. Using this list we can define the list of congruence on $\mathbb{A}$ and all its subuniverses of any type, for example:
\begin{equation}
  \begin{gathered}
   \Sigma^{max}_{\mathcal{A}}(i,a,b)\iff \Sigma_{\mathcal{A}}(i,a,b)\wedge maxCong_m(A,\Omega,\Sigma_{\mathcal{A},i});\\
     \Sigma^{PC}_{\mathcal{D}}(i,a,b)\iff \Sigma_{\mathcal{D}}(i,a,b)\wedge PCCong_m(D,\Omega,\Sigma_{\mathcal{D},i}).
  \end{gathered}
\end{equation}
In these definitions we do not enumerate elements in the lists from the beginning, we thin out the existing lists $\Sigma_{\mathcal{A}}$, $\Sigma_{\mathcal{D}}$. That is, for some $i<2^{l^2}$ the new lists can be empty. Then we say that $\sigma$ is an intersection of all PC congruences on $\mathbb{D}$ if it satisfies the following $\Sigma^{1,b}_0$ relation:
\begin{equation}
  \begin{gathered}
 CongPC(D,\Omega,\sigma)\iff Cong_m(D,\Omega,\sigma)\wedge (\forall i<2^{l^2},\Sigma^{PC}_{{\mathcal{D}},i}\neq\emptyset\rightarrow \sigma\subseteq\Sigma^{PC}_{{\mathcal{D}},i})\wedge\\
 \hspace {0pt}\wedge \bigwedge_{\sigma'<\langle l,l\rangle}(Cong_m(D,\Omega,\sigma')\wedge \sigma\subsetneq \sigma')\rightarrow\exists j<2^{l^2}, \Sigma^{PC}_{{\mathcal{D}},j}\neq\emptyset ,\exists a,b\in D,\\
 \hspace {0pt}\sigma'(a,b)\wedge \neg \Sigma^{PC}_{{\mathcal{D}},j}(a,b).
  \end{gathered}
\end{equation}

A subuniverse $E\subseteq D$ is called a \emph{PC subuniverse} if $E=E_0\cap...\cap E_{s-1}$ where each $E_i$ is an equivalence class of some PC congruence.
\begin{definition}[PC subuniverse, I definition]
For an algebra $\mathbb{D}=(D,\Omega)$, $E$ is a PC subuniverse if
\begin{equation}
  \begin{gathered}
 \hspace {0pt}PCsubU(E,D,\Omega)\iff (E=\emptyset)\vee (E=D)\vee\\
 \hspace {0pt}\vee\big((\exists j<2^{l^2},\forall a,b\in E, \Sigma^{PC}_{{\mathcal{D}}}(j,a,b))\wedge\\
 \wedge (\exists i< 2^{l^2},\forall a\in E,\forall b\in D,\Sigma_{\mathcal{D}}(i,a,b)\leftrightarrow b\in E)\big).
  \end{gathered}
\end{equation}
\end{definition}
A PC subuniverse can be empty or full ($E=D$). The condition in the second line ensures that the entire $E$ is inside exactly one block of any number of PC congruences (since we do not restrict the number of different $j$'s in any way) and the condition in the third line ensures that $E$ is indeed a block of some congruence (not necessarily PC congruence since due to the maximality, intersection of any number of PC congruences is not a PC congruence).

We give a second definition of a PC subuniverse in this section straightaway. The following lemma is proved in \cite{zhuk2020proof}.
\begin{lemmach}[Lemma 7.13.1,   \cite{zhuk2020proof}]\label{alskjdiojkgfkj}
  Suppose that $\sigma_1,...,\sigma_k$ are all PC congruences on $A$. Put $A_i=A/\sigma_i$, and define $\psi:A\to A_1\times...\times A_k$ by $\psi(a)=(a/\sigma_1,...,a/\sigma_k)$. Then
\begin{enumerate}
  \item $\psi$ is surjective, hence $A/\cap_{i}\sigma_i\cong A_1\times...\times A_k$;
  \item the PC subuniverses are the sets of the form $\psi^{-1}(S)$ where $S\subseteq A_1\times...\times A_k$ is a relation definable by unary constraints of the form $x_i=a_i$;
  \item for each non-empty PC subuniverse $B$ of $A$ there is a congruence $\theta$ of $A$ such that $B$ is an equivalence class of $\theta$ and $A/\theta$ is isomorphic to a product of PC algebras having no non-trivial binary absorbing subuniverse or center. That is, $A/\theta\cong A_{j_1}\times... \times A_{j_s}$ where each $A_{j_i}$ is a PC algebra that has no non-trivial binary absorbing subuniverse or center.
\end{enumerate}
\end{lemmach}
Since for a fixed algebra $\mathbb{A}=(A,\Omega)$ and all its subalgebras $D$ we know the list of all PC congruences, we do not need to prove this lemma, the algorithm can just check it. Then, since the algebra $\mathbb{A}$ and all its subalgebras are bounded by size $l$, the maximal possible number of quotients in the direct product $D_0\times ...\times D_{k-1}$ is $s=log_2l$.
\begin{definition}[PC subuniverse, II definition]\label{';s;sdllfjjg}
For an algebra $\mathbb{D}=(D,\Omega)$, $s=log_2l$, $E$ is a PC subuniverse if
\begin{equation}
  \begin{gathered}
 \hspace{0pt}PCsubU(E,D,\Omega)\iff (E=\emptyset)\vee (E=D)\vee\big((\exists i< 2^{l^2},\forall a\in E,\forall b\in D,\\
 \Sigma_{\mathcal{D}}(i,a,b)\leftrightarrow b\in E)\wedge\\
 \hspace {0pt}\wedge \bigvee_{\substack{ (\sigma_0\in \Sigma^{PC}_{D}})}\bigvee_{(H\in \mathrm{M}_{D,\Sigma_{D,i},\sigma_0})} ISO_{alg}(D/\Sigma_{D,i},\Omega/\Sigma_{D,i},D/\sigma_0,\Omega/\sigma_0,H) \vee \\
 \hspace {0pt}\vee \bigvee_{\substack{ (\sigma_0,\sigma_1\in \Sigma^{PC}_{D}})}\bigvee_{(H\in \mathrm{M}_{D,\Sigma_{D,i},\sigma_0,\sigma_1})} ISO_{alg}(D/\Sigma_{D,i},\Omega/\Sigma_{D,i},D/\sigma_0\times D/\sigma_1,\\
 \Omega/\sigma_0\cap \sigma_1,H) \vee ...\\
 \hspace{0pt}...\vee \bigvee_{\substack{ (\sigma_0,...,\sigma_{s-1}\in \Sigma^{PC}_{D}})}\bigvee_{(H\in \mathrm{M}_{D,\Sigma_{D,i},\sigma_0,...,\sigma_{s-1}})} ISO_{alg}(D/\Sigma_{D,i},\Omega/\Sigma_{D,i},D/\sigma_0\times...\\
 ...\times D/\sigma_{s-1},  \Omega/\cap_{i}\sigma_i ,H)\big).
  \end{gathered}
\end{equation}
\end{definition}

\subsubsection{Linear subuniverse}
Since the algebra $\mathbb{A}$ and all its subalgebras are bounded by size $l$, the maximal possible number of prime fields in the prime product $\mathbb{Z}_{p_0}\times ...\times \mathbb{Z}_{p_{k-1}}$ is $s=log_2l$. We formalize a finite abelian group $\mathbb{Z}_p=(Z_p,0,+,-)$, where $p$ is prime, as a set $Z_p$, $|Z_p|=p\wedge \forall a<p, Z_p(a)$, and $+_{(mod\,p)}$, $-_{(mod\,p)}$. For any direct product up to $k\leq log_2l$ abelian groups $\mathbb{Z}_{p_0}\times ...\times \mathbb{Z}_{p_{k-1}}$ we define a set $Z_{p_0}\times...\times Z_{p_{k-1}}\leq (p_0+...+p_{k-1}+1)^{2^k}$ such that for all $a_0<p_0,...,a_{k-1}<p_{k-1}$, $(a_0,...,a_{k-1})\in Z_{p_0}\times...\times Z_{p_{k-1}}$ and
\begin{equation}
  \begin{gathered}
     \hspace {0pt}\forall a_0,b_0<p_0,...,\forall a_{k-1},b_{k-1}<p_{k-1},\\
     +((a_0,...,a_{k-1}),(b_0,...,b_{k-1}))=(a_0+_{(mod\,p_0)} b_0,...,a_{k-1}+_{(mod\,p_{k-1})} b_{k-1}),\\
     -((a_0,...,a_{k-1}),(b_0,...,b_{k-1}))=(a_0-_{(mod\,p_0)} b_1,...,a_{k-1}-_{(mod\,p_{k-1})} b_{k-1}).
  \end{gathered}
\end{equation}
We will denote elements $(a_0,...,a_{k-1})$ of $Z_{p_0}\times...\times Z_{p_{k-1}}$ by $\bar{a}^k$, and will omit index $_{(mod\,p)}$ when it does not lead to confusion. Also, we allow the use of trivial algebras (with one element $0$) in a product, so $Prime'(p)\iff Prime(p)\vee p=1$.

\begin{definition}[Linear algebra of size at most $|A|$]\label{a'a'a's;ldkfijgt} For an algebra $\mathbb{D}=(D,\Omega)$, $s=log_2l$, we say that it is a \emph{linear algebra} if
\begin{equation}
  \begin{gathered}
       LinA(D,\Omega)\iff \bigvee_{\substack{p_0\leq l \\ Prime'(p_0)}}\bigvee_{H\in \mathrm{M}_{D,p_0}} ISO_{alg}(D,\Omega,Z_{p_0},\bar{a}^1_1+...+\bar{a}^1_m,H)\vee\\
 \hspace {0pt}\vee\bigvee_{\substack{ p_0\cdot p_1\leq l\\ Prime'(p_0),Prime'(p_1)}} \bigvee_{H\in \mathrm{M}_{D,p_{0},p_{1}}} ISO_{alg}(D,\Omega,Z_{p_0}\times Z_{p_1},\bar{a}^2_1+...+\bar{a}^2_m,H)\vee...\\
 \hspace {0pt}...\vee\bigvee_{\substack{p_0\cdot...\cdot p_{s-1}\leq l \\ Prime'(p_0),...,Prime'(p_k)\\
s=log_2l}} \bigvee_{H\in \mathrm{M}_{D,p_{0}, ...,p_{s-1}}} ISO_{alg}(D,\Omega,Z_{p_0}\times...\times Z_{p_{s-1}},\\
\bar{a}^s_1+...+\bar{a}^s_m,H).
  \end{gathered}
\end{equation}
\end{definition}

\begin{definition}[Linear congruence]
We say that a set $\sigma<\langle l,l\rangle$ is a \emph{linear congruence} on algebra $\mathbb{D}=(D,\Omega)$ if
\begin{equation}
  \begin{gathered}
 LinCong_m(D,\Omega,\sigma)\iff LinA(D/\sigma,\Omega/\sigma).
  \end{gathered}
\end{equation}
\end{definition}
We can also check that any linear congruence of algebra $\mathbb{A}$ (or its subalgebras) bounded by size $l$ is a linear congruence for any subalgebra of $\mathbb{A}$ (or their subalgebras). Let us define the set of all linear congruences on $\mathbb{D}$ as:
\begin{equation}
  \begin{gathered}
   \Sigma^{lin}_{\mathcal{D}}(i,a,b)\iff \Sigma_{\mathcal{D}}(i,a,b)\wedge LinCong_m(D,\Omega,\Sigma_{\mathcal{D},i}).
  \end{gathered}
\end{equation}
Then we say that $\sigma$ is \emph{the} minimal linear congruence (an intersection of all linear congruences) on $D$ if
\begin{equation}
  \begin{gathered}
 CongLin(D,\Omega,\sigma)\iff \exists i<2^{l^2}, \sigma = \Sigma^{lin}_{{\mathcal{D}},i}\wedge \forall j<2^{l^2},\Sigma^{lin}_{{\mathcal{D}},j}\neq\emptyset\rightarrow \sigma\subseteq\Sigma^{lin}_{{\mathcal{D}},j}.
  \end{gathered}
\end{equation}
Note that the definition of $CongLin$ differs from the definition of $CongPC$ since any intersection of linear congruences is again a linear congruence. 
A subuniverse $L\subseteq D$ is called a \emph{linear subuniverse} if it is stable under $CongLin$:
\begin{equation}
  \begin{gathered}
  LNsubU(L,D,\Omega)\iff SwNU(\Omega,L)\wedge\bigwedge_{\sigma\leq \langle l,l\rangle} CongLin(D,\Omega,\sigma) \rightarrow \\
  \rightarrow Stable_1(L,\sigma). 
  \end{gathered}
\end{equation}
\begin{remark}
A linear subuniverse is a union of classes of $CongLin$. However, not every union of such classes needs to be a subuniverse. For example, for a linear algebra $(D,\Omega)$ that is isomorphic to $(\mathbb{Z}_p,x_1+...+x_m)$, and a minimal linear congruence $\Delta$ every element of $\mathbb{Z}_p$ is a subuniverse (since $\Omega$ is idempotent), but not any other proper subset of $\mathbb{Z}_p$ is a subuniverse. From here, we get that there are no non-trivial congruences on $(D,\Omega)$ (every congruence block must be a subuniverse).
\end{remark}

\subsubsection{One-of-four and minimal subuniverse}
All the following formulas in this section are $\Sigma^{1,b}_0$ (they would not if $\mathbb{A}$ is not fixed). We say that $B$ is \emph{one-of-four} subuniverse of $D$ if
\begin{equation}
  \begin{gathered}
     1of4subU(B,D,\Omega)\iff PCsubU(B,D,\Omega)\vee LNsubU(B,D,\Omega)\vee\\
     \hspace {0pt}\vee\bigvee_{T<(3l)^8} BAsubU(B,D,T,\Omega)\vee\bigvee_{S<(4l)^{16}}CRsubU(B,D,S,\Omega).\\
  \end{gathered}
\end{equation}
We say that a subuniverse is minimal if it is non-trivial and inclusion minimal (does not contain any other subuniverse of the same type). For example,
\begin{equation}
  \begin{gathered}
     minBAsubU(B,D,T,\Omega)\iff BAsubU(B,D,T,\Omega)\wedge \bigwedge_{B'<l}\bigwedge_{T'<(3l)^8}B'\subsetneq B \rightarrow\\
     \hspace {0pt}\rightarrow\neg BAsubU(B',D,T',\Omega);\\
     \hspace {0pt}minLNsubU(B,D,\Omega)\iff LNsubU(B,D,\Omega)\wedge \bigwedge_{B'<l} B'\subsetneq B \rightarrow\\
     \hspace {0pt}\rightarrow\neg LNsubU(B',D,\Omega).
    \end{gathered}
\end{equation}
For linear and PC subuniverses we also will use the fact that a minimal linear/ PC subuniverse is a block of $CongLin$/ $CongPC$. We denote a block $B$ of a congruence $\sigma$ as
\begin{equation}
  Block(B,D,\sigma)\iff \forall a\in B,\forall b\in D, \sigma(a,b)\leftrightarrow b\in B.
\end{equation}
Then
\begin{equation}
  \begin{gathered}
     minPCsubU^B(B,D,\Omega)\iff \bigwedge_{\sigma<(2l)^4}CongPC(D,\Omega, \sigma)\rightarrow\\
     \hspace {0pt}\rightarrow Block(B,D,\sigma);\\
     minLNsubU^B(B,D,\Omega)\iff \bigwedge_{\sigma<(2l)^4}CongLin(D,\Omega, \sigma)\rightarrow \\
     \hspace {0pt}\rightarrow Block(B,D,\sigma).\\
    \end{gathered}
\end{equation}

\subsection{Reductions}
Note that all further definitions for all types of reductions and strategies are $\Sigma^{1,b}_0$-formulas.
\begin{definition}[Different types of reductions]
For an instance $\Theta=(\mathcal{X},\ddot{\mathcal{A}})$ with domain set $D=(D_0,...,D_{n-1})$ we say that a set $D^{(\bot)}=(D^{(\bot)}_0,...,D^{(\bot)}_{n-1})$ is an absorbing reduction of $D$ if there exists a term operation $T$ such that $D^{(\bot)}_i$ is a binary absorbing subuniverse of $D_i$ with the term operation $T$ for every $i$:
\begin{equation}
  \begin{gathered}
 BARed(D^{(\bot)},D)\iff Red(D^{(\bot)},D)\wedge \bigvee_{T<(3l)^{8}}Pol_2(T,D, \Gamma_{\mathcal{D}})\wedge\\
 \hspace {0pt}\wedge \forall i<n,\,BAsubU(D^{(\bot)}_i,D_i,T,\Omega).
  \end{gathered}
\end{equation}
We say that $D^{(\bot)}=(D^{(\bot)}_0,...,D^{(\bot)}_{n-1})$ is a central reduction if $D_i$ is a central subuniverse for every $i$:
\begin{equation}
  \begin{gathered}
 CRRed(D^{(\bot)},D)\iff Red(D^{(\bot)},D)\wedge \forall i<n, \\
 \hspace{0pt}\bigvee_{S_i<(4l)^{16}} Pol_3(S_i,D, \Gamma_{\mathcal{D}})\wedge CRsub(D^{(\bot)}_i,D_i,S_i).
  \end{gathered}
\end{equation}
We say that $D^{(\bot)}=(D^{(\bot)}_0,...,D^{(\bot)}_{n-1})$ is a PC reduction if 
\begin{equation}
  \begin{gathered}
 PCRed(D^{(\bot)},D)\iff Red(D^{(\bot)},D)\wedge \forall i<n,\\
 \hspace {0pt}PCsubU(D^{(\bot)}_i,D_i,\Omega)\wedge subTA_{\neg BACR}(D_i,\Omega).
  \end{gathered}
\end{equation}
We say that $D^{(\bot)}=(D^{(\bot)}_0,...,D^{(\bot)}_{n-1})$ is a linear reduction if 
\begin{equation}
  \begin{gathered}
 LNRed(D^{(\bot)},D)\iff Red(D^{(\bot)},D)\wedge \forall i<n,\\
 \hspace {0pt}LNsubU(D^{(\bot)}_i,D_i,\Omega)\wedge subTA_{\neg BACR}(D_i,\Omega).
  \end{gathered}
\end{equation}
In an obvious way, we can define a minimal absorbing/ central/ PC/ linear reduction, a non-linear reduction $nonLNRed(D^{(\bot)},D)$ and one-of-four reduction $1of4Red(D^{(\bot)},D)$.

\begin{remark}
  A CSP instance $\Theta=(\mathcal{X},\ddot{\mathcal{A}})$ is a set 
  $$\Theta(\langle\, \,\underbrace{\langle n,\langle n,n\rangle \rangle}_{\mathcal{X}},\underbrace{\langle \langle n,l\rangle, \langle\langle n,l\rangle,\langle n,l\rangle \rangle \rangle}_{\ddot{\mathcal{A}}} \, \,   \rangle)$$
 Let us denote the length of $\Theta$ as a number function $instsize(n,l)=|\Theta|$.
\end{remark}

\end{definition}

A \emph{strategy} for a CSP instance $\Theta = (\mathcal{X},\ddot{\mathcal{A}})$ with domain set $D$ is a sequence of reductions $D^{(0)},...,D^{(s)}$, where $D^{(j)}=(D^{(j)}_0,...,D^{(j)}_{n-1})$, such that $D^{(0)}=D$ and $D^{(j)}$ is a one-of-four $1$-consistent reduction of $\Theta^{(j-1)}$ for every $j\geq 1$. A strategy is called \emph{minimal} if every reduction in the sequence is minimal. Since after any reduction we decrease at least one domain by at least one element, to represent the entire strategy it is enough to consider a set (matrix) with $nl$ rows, each row representing a reduction of the CSP instance. We need to track both domain reductions (the set $V_{\ddot{\mathcal{A}}}$, or $D$) and restrictions of the constraint relations (the set $E_{\ddot{\mathcal{A}}}$). An input digraph $\mathcal{X}$ does not change, but for consistency, we will track it as well. 

\begin{definition}[A strategy for a CSP instance] A strategy for an instance $\Theta$ up to $s\leq nl$ step is a set $\Theta_{Str}< \langle nl,instsize(n,l)\rangle$ such that:
\begin{equation}
  \begin{gathered}
     Strategy(\Theta,\Theta_{Str},s)\iff \Theta_{Str}^0=\Theta\wedge \forall 1\leq j\leq s,  1C(\Theta^{(j)}_{Str})\wedge \\
     \wedge redinst(\Theta^{(j-1)},D^{(j)})=\Theta^{(j)}_{Str}\wedge \forall 1\leq j\leq s,\, 1of4Red(D^{(j)}_{Str},D).
  \end{gathered}
\end{equation}
A strategy is called minimal if
\begin{equation}
  \begin{gathered}
     minStrategy(\Theta,\Theta_{Str},s)\iff Strategy(\Theta,\Theta_{Str},s)\wedge \forall 1\leq j\leq s,\,\\  \hspace {0pt}min1of4Red(D^{(j)}_{Str},D).
  \end{gathered}
\end{equation}

\end{definition}
When we want to consider the domain strategy separately, we will refer to it as $D_{Str}<\langle nl,\langle n,l\rangle\rangle$, each $j$th row representing $D^{(j)}$.

\subsection{Three universal algebra axiom schemes}\label{===-06476357234609}
In this section, we formalize three universal algebra axiom schemes, reflecting the "only if" implications of Theorems \ref{fhfhfhydj}, \ref{llldju67yr}, and \ref{fjfjhgd} in \cite{zhuk2020proof} (for the soundness we do not need the "if" implication). These axiom schemes were already introduced in \cite{gaysin2023proof}, and we echo the formalization because of its importance for the understanding of the paper. The schemes are formulated for a constraint language $\Gamma_{\mathcal{A}}$ over the set $A$ of size $l$, fixed algebra $\mathbb{A}=(A,\Omega)$ with $\Omega$ being a special $m$-ary WNU operation. They consist of finitely many $\forall \Sigma^{1,b}_2$-formulas (for all possible subuniverses of $\mathbb{A}$). 

BA$_\mathcal{A}$-axiom scheme reflects that if $\Theta$ is a cycle-consistent irreducible CSP instance, and $B$ is a non-trivial binary absorbing subuniverse of $D_i$, then $\Theta$ has a solution only if $\Theta$ has a solution with $x_i\in B$ (Theorem \ref{fhfhfhydj} in \cite{zhuk2020proof}): 
\begin{equation}
\begin{gathered}
 \hspace {0pt}BA_{\mathcal{A},B,D}=_{def}\forall\mathcal{X}=(V_{\mathcal{X}},E_{\mathcal{X}}),\forall\ddot{\mathcal{A}}=(V_{\ddot{\mathcal{A}}},E_{\ddot{\mathcal{A}}},D_0,...,D_{n-1}),\\
 \hspace {0pt}\big(B\subsetneq D\wedge SwNU_m(\Omega,D)\wedge SwNU_m(\Omega,B)\wedge\\
 \hspace{0pt}\wedge\exists T <(3l)^{2^3}, Pol_{2}(T,D, \Gamma_{\mathcal{A}})\wedge BAsubS(B,D,T) \wedge\\
 \hspace{0pt}\wedge Inst(\Theta,\Gamma_{\mathcal{A}})\wedge CCInst(\mathcal{X},\ddot{\mathcal{A}})\wedge IRDInst(\mathcal{X},\ddot{\mathcal{A}})\wedge \\
 \hspace{0pt}\exists i<n, D_i=D\wedge\\ 
 \hspace{0pt}\ddot{HOM}(\mathcal{X},\ddot{\mathcal{A}})\big)\implies \ddot{HOM}(\mathcal{X},\ddot{\mathcal{A}}=(V_{\ddot{\mathcal{A}}},E_{\ddot{\mathcal{A}}},D_0,...,B,...,D_{n-1})).
\end{gathered}
\end{equation}
CR$_\mathcal{A}$-axiom scheme states that if $\Theta$ is a cycle-consistent irreducible CSP instance, and $C$ is a non-trivial central subuniverse of $D_i$, then $\Theta$ has a solution only if $\Theta$ has a solution with $x_i\in C$ (Theorem \ref{llldju67yr} in \cite{zhuk2020proof}):
\begin{equation}
\begin{gathered}
 \hspace{0pt}CR_{\mathcal{A},D,C}=_{def}\forall \mathcal{X}=(V_{\mathcal{X}},E_{\mathcal{X}}),\forall\ddot{\mathcal{A}}=(V_{\ddot{\mathcal{A}}},E_{\ddot{\mathcal{A}}},D_0,...,D_{n-1}),\\
 \hspace{0pt}\big(C\subsetneq D\wedge SwNU_m(\Omega,D)\wedge SwNU_m(\Omega,C)\wedge\\
   \hspace {0pt}\ \exists S<(4l)^{2^4},\,Pol_3(S, D,\Gamma_{\mathcal{A}})\wedge CRsubS(C,D,S) \wedge \\
 \hspace {0pt}\wedge Inst(\Theta,\Gamma_{\mathcal{A}})\wedge CCInst(\mathcal{X},\ddot{\mathcal{A}})\wedge IRDInst(\mathcal{X},\ddot{\mathcal{A}})\wedge \\
 \hspace{0pt}\exists i<n, D_i=D\wedge\\ 
 \hspace {0pt}\ddot{HOM}(\mathcal{X},\ddot{\mathcal{A}})\big)\implies \ddot{HOM}(\mathcal{X},\ddot{\mathcal{A}}=(V_{\ddot{\mathcal{A}}},E_{\ddot{\mathcal{A}}},D_0,...,C,...,D_{n-1})).
\end{gathered}
\end{equation}

Finally, PC$_\mathcal{A}$-axioms says that
if $\Theta$ is a cycle-consistent irreducible CSP instance, there does not exist a non-trivial binary absorbing subuniverse or a non-trivial center on $D_j$ for every $j$, $(D_i,\Omega)/\sigma_i$ is a polynomially complete algebra, and $E$ is an equivalence class of $\sigma_i$, then $\Theta$ has a solution only if $\Theta$ has a solution with $x_i\in E$ (Theorem \ref{fjfjhgd} in \cite{zhuk2020proof}):
\begin{equation}
\begin{gathered}
 \hspace{0pt}PC_{\mathcal{A},D,E}=_{def}\forall \mathcal{X}=(V_{\mathcal{X}},E_{\mathcal{X}}),\forall\ddot{\mathcal{A}}=(V_{\ddot{\mathcal{A}}},E_{\ddot{\mathcal{A}}},D_0,...,D_{n-1}),\\ 
 \hspace {0pt}\big([\forall j<n,\forall B<l,\forall T<(3l)^{2^3}, Pol_2(T,D_j, \Gamma_{\mathcal{A}})\rightarrow\neg BAsubS(B,D_j,T)\wedge\\
 \hspace {0pt}\wedge\forall j<n,\forall C<l, \forall S<(4l)^{2^4}, Pol_3(S,D_j, \Gamma_{\mathcal{A}})\rightarrow\neg CRsubS(C,D_j,S)]\\
 \hspace {0pt}\wedge \exists \sigma<\langle l,l\rangle, \exists D/\sigma<l,\exists\Omega/\sigma<(ml)^{2^{m+1}}, FA_m(D/\sigma,\Omega/\sigma,D,\Omega,\sigma)\wedge\\
 \hspace {0pt}\wedge \exists P<(4l)^{2^4},\, Pol_{3}(P,D/\sigma,\Gamma^{diag}_{\mathcal{D}}/\sigma)\wedge PCD(D/\sigma,P)\wedge \\
 \hspace{0pt}SwNU_m(\Omega,D)\wedge E\subsetneq D\wedge (\forall a\in E,\forall b\in D, \sigma(a,b)\leftrightarrow b\in E)\wedge\\
 \hspace {0pt}\wedge Inst(\Theta,\Gamma_{\mathcal{A}})\wedge CCInst(\mathcal{X},\ddot{\mathcal{A}})\wedge IRDInst(\mathcal{X},\ddot{\mathcal{A}})\wedge \\
 \hspace{0pt}\exists i<n, D_i=D\wedge\\ 
 \hspace {0pt}\ddot{HOM}(\mathcal{X},\ddot{\mathcal{A}})\big)\implies \ddot{HOM}(\mathcal{X},\ddot{\mathcal{A}}=(V_{\ddot{\mathcal{A}}},E_{\ddot{\mathcal{A}}},D_0,...,E,...,D_{n-1})).
\end{gathered}
\end{equation}

\section{Formalization of a proof of three axiom schemes}\label{'a'tteyrghwgdfh}
In this section, we consider solution sets and not their projections to subsets of coordinates. The proofs for both cases do not essentially differ, and we can repeat the same reasoning for any projection. However, since projection increases the complexity of a formula from $\Sigma^{1,b}_0$ to $\Sigma^{1,b}_1$, in any proof using the induction or comprehension axioms of level $i$ for a set of solutions, we need to exchange the level for $i+1$ for a projection to a subset of coordinates.

\subsection{Formalization of some auxiliary lemmas and theorems}
We will present the formalization of several selected statements and their proofs used in the proof of the soundness of the algorithm. We selected those that genuinely represent various types of arguments encountered in \cite{zhuk2020proof}, \cite{zhuk2020strong}. It should be sufficiently clear that other statements of a similar nature can be formalized analogously.

\subsubsection{Properties of a binary absorbing subuniverse on $\mathbb{A}^n$}

We say that a solution set to a CSP instance $\Theta$ over $\Gamma_{\mathcal{A}}$ on $n$ variables, $\mathscr{R}_{\Theta}\leq D_0\times...\times D_{n-1}$ is a binary absorbing subuniverse of $D_0\times...\times D_{n-1}$ if there exists a binary term operation $T\in Pol(\Gamma_{\mathcal{A}})$ such that for any two maps $H,H':[n]\to(D_0,...,D_{n-1})$ where $H\notin \mathscr{R}_{\Theta}$ and $H'\in \mathscr{R}_{\Theta}$, the maps $usepol_2(T,H,H')$ and $usepol_2(T,H',H)$ are in $\mathscr{R}_{\Theta}$. An analogous definition can be formulated for any projection of the set of solutions $\mathscr{R}^{i_1,...,i_s}_{\Theta}$. 

\begin{lemmach}[Lemma 7.1, \cite{zhuk2020proof}]\label{alalalur6456} Suppose that $\mathscr{R}_{\Theta}$ is defined by a $pp$-formula $\Theta(x_0,...,x_{n-1})$ and $\Theta'$ is obtained from $\Theta$ by replacing some constraint relations $\rho_1,...,\rho_s$ by constraint relations $\rho'_1,...,\rho'_s$ such that $\rho_k'$
absorbs $\rho_k$ with a term operation $T$ for every $k$. Then $V^0$ proves that the relation $\mathscr{R}_{\Theta'}$ defined by
$\Theta'(x_0,...,x_{n-1})$ absorbs $\mathscr{R}_{\Theta}$ with the term operation $T$.
\end{lemmach}
\begin{proof}
  Let us consider two CSP instances $\Theta = (\mathcal{X},\ddot{\mathcal{A}})$ and $\Theta' = (\mathcal{X}',\ddot{\mathcal{A}}')$, where $\mathcal{X}'=\mathcal{X}$ (the analogous reasoning can be applied to projections). Suppose that there exists a binary term $T\in Pol(\Gamma_{\mathcal{A}})$ such that for each $i<n$, $D'_{i}\subseteq D_i$ binary absorbs $D_i$ and for all $i,j<n$ with $E_{\mathcal{X}}(i,j)$, $E^{ij}_{\ddot{\mathcal{A}}'}\subseteq E^{ij}_{\ddot{\mathcal{A}}}$ binary absorbs $E^{ij}_{\ddot{\mathcal{A}}}$:
\begin{equation}\label{djdjjduuuhs6745ed}
  \begin{gathered}
     \hspace {0pt}\forall a\in D'_i,\forall b\in D_i,\exists c_1,c_2\in D'_i,\,T(a,b)=c_1\wedge T(b,a)=c_2\wedge\\
     \hspace {0pt}\forall a_1,b_1<l,\forall a_2,b_2<l,\,(E^{ij}_{\ddot{\mathcal{A}}}(a_1,b_1)\wedge E^{ij}_{\ddot{\mathcal{A}}'}(a_2,b_2))\rightarrow\\
     \hspace {0pt}\rightarrow\exists a_3,a_4<l,\exists b_3,b_4<l,\,E^{ij}_{\ddot{\mathcal{A}}'}(a_3,b_3)\wedge E^{ij}_{\ddot{\mathcal{A}}'}(a_4,b_4)\wedge\\
     \wedge T(a_1,a_2)=a_3\wedge T(a_2,a_1)=a_4\wedge T(b_1,b_2)=b_3\wedge T(b_2,b_1)=b_4.
  \end{gathered}
\end{equation}
Note that for some $i,j<n$, $D'_i$ and $E^{ij}_{\ddot{\mathcal{A}}'}$ could be equal to $D_i$ and $E^{ij}_{\ddot{\mathcal{A}}}$. If $\mathscr{R}_{\Theta'}$ or/ and $\mathscr{R}_{\Theta}$ are empty, we are done ($\mathscr{R}_{\Theta'}$ is an empty subuniverse). Suppose that both instances have solutions. Every solution to the instance $\Theta'$ is a solution to the instance $\Theta$. Consider any two solutions to $\Theta$ and $\Theta'$, homomorphisms $H$ and $H'$, respectively. Consider two maps $H_1=usepol_2(T,H,H')$ and $H_2=usepol_2(T,H',H)$. We need to prove that these maps are homomorphisms from $\mathcal{X}'$ to $\ddot{\mathcal{A}}'$. Suppose that $H_1$ (or $ H_2$) is not a homomorphism. Then there exists an edge in $\mathcal{X}'$, $E_{\mathcal{X}'}(i,j)$ such that $H_1$ fails to map it to an edge in $\ddot{\mathcal{A}}'$. But this contradicts (\ref{djdjjduuuhs6745ed}). 
\end{proof}

\begin{corollary}[Corollary 6.1.3, \cite{zhuk2020strong}]\label{==-=-=-=-=-=-=-=-}
Suppose $\mathscr{R}_{\Theta}\leq D_0\times...\times D_{n-1}$ and $B_i$ is an absorbing subuniverse in $A_i$ with a term $T$ for every $i$. Then $V^0$ proves that $(B_0\times...\times B_{n-1})\cap \mathscr{R}_{\Theta}$ is an absorbing subuniverse of $\mathscr{R}_{\Theta}$ with the term $T$.
\end{corollary}

\begin{corollary}[Corollary 7.1.2, \cite{zhuk2020proof}]\label{akjshdutjgjf}
Suppose that $\mathscr{R}_{\Theta}\leq D_0\times...\times D_{n-1}$ is a relation such that $pr_0(\mathscr{R}_{\Theta})=D_0$ and $B=pr_0((B_0\times...\times B_{n-1})\cap \mathscr{R}_{\Theta})$, where $B_i$ is an absorbing subuniverse in $D_i$ with a term $T$ for every $i$. Then $V^1$ proves that $B$ is an absorbing subuniverse in $D_0$ with the term $T$.
\end{corollary}
\begin{proof}
  Consider $\mathscr{R}_{\Theta}$ as a solution set to some CSP instance $\Theta=(\mathcal{X},\ddot{\mathcal{A}})$. Since every $B_i$ is an absorbing subuniverse of $D_i$, $(B_0\times...\times B_{n-1})\cap \mathscr{R}_{\Theta}$ is a solution set for an instance $\Theta'$ defined similarly to (\ref{djdjjduuuhs6745ed}) with a domain set $D'=\{B_0,...,B_{n-1}\}$ (for all $i,j<n$, $E^{ij}_{\ddot{\mathcal{A}}'}=E^{ij}_{\ddot{\mathcal{A}}}$). Then from $pr_0(\mathscr{R}_{\Theta})=D_0$ it follows that $B$ is an absorbing subuniverse of $D_0$.
\end{proof}

\begin{lemmach}[Lemma 7.3, \cite{zhuk2020proof}]
  Suppose that $\mathscr{R}_{\Theta}$ is a non-trivial absorbing subuniverse of $D_0\times...\times D_{n-1}$. Then $V^2$ proves that for some $i$ there exists a non-trivial absorbing subuniverse $B_i$ in $D_i$ with the same term.
\end{lemmach}
\begin{proof}
The lemma is proved by induction on the arity of $\mathscr{R}_{\Theta}$. $\mathscr{R}_{\Theta}$ is a non-trivial binary absorbing subuniverse of $D_0\times...\times D_{n-1}$ if there exists a solution to $\Theta_{null}$ that is not a solution to $\Theta$, and there exists a binary term $T$ such that for any two homomorphisms $H$ from $\mathcal{X}$ to $\ddot{\mathcal{A}}$ and $H_{null}$ from $\mathcal{X}_{null}$ to $\ddot{\mathcal{A}}_{null}$, both maps $usepol_2(T,H,H_{null})$ and $usepol_2(T,H_{null},H)$ are homomorphisms from $\mathcal{X}$ to $\ddot{\mathcal{A}}$.

Consider the following $\mathfrak{B}(\Sigma^{1,b}_1)$-formula $\phi(t)$. Here, as fixed parameters, we use $\Gamma_{\mathcal{A}}$ and $\mathbb{A}=(A,\Omega)$. Induction goes on the size of the instance, $V_{\mathcal{X}} = V_{\mathcal{X}_{null}}=t$. Witnesses ($\forall$ quantification) in $\Sigma^{1,b}_2$-induction corresponding to $t$ are sets $E_{\mathcal{X}}, E_{\mathcal{X}_{null}}$, $\ddot{\mathcal{A}}$ with the set of vertices of size $\langle t,l\rangle$.
\begin{equation}
  \begin{gathered}
     \hspace {0pt}\phi(t):=\hspace{0pt} \big(V_{\mathcal{X}} = V_{\mathcal{X}_{null}}=t\wedge DG(\Theta)\wedge DG_{null}(\Theta_{nul}) \wedge \\
     \wedge Inst(\Theta,\Gamma_{\mathcal{A}})\wedge Inst(\Theta_{null},\Gamma_{\mathcal{A}})\big) \wedge \\
 \wedge \big(\bigvee_{T\in \Pi^2_{\mathcal{A}}}\forall H, H_{null}\leq \langle t,\langle t,l\rangle\rangle,\,\ddot{HOM}(\mathcal{X},\ddot{\mathcal{A}},H)\wedge\ddot{HOM}(\mathcal{X}_{null},\ddot{\mathcal{A}}_{null},H_{null})\rightarrow\\
     \hspace {0pt}\rightarrow \ddot{HOM}(\mathcal{X},\ddot{\mathcal{A}},usepol_2(T,H,H_{null})) \wedge\ddot{HOM}(\mathcal{X},\ddot{\mathcal{A}},usepol_2(T,H_{null},H)\big)\\
     \wedge\big(\exists H'\leq \langle t,\langle t,l\rangle\rangle,\,\ddot{HOM}(\mathcal{X}_{null},\ddot{\mathcal{A}}_{null},H')\wedge \neg \ddot{HOM}(\mathcal{X},\ddot{\mathcal{A}},H')\big)\implies\\
     \hspace {0pt}\implies\exists i<t,\bigvee_{B<l} B_i\subsetneq A_i\wedge BAsubU(B_i,A_i,T,\Omega).
  \end{gathered}
\end{equation}
The expression in the first brackets says that $\Theta$ and $\Theta_{null}$ are CSP instances over $\Gamma_{\mathcal{A}}$ on $t$ variables, and $\Theta_{null}$ is an empty instance. This is obviously true for $t=1$. Suppose that it is true for $t=s$ and consider $t=s+1$. If the projection of $\mathscr{R}_{\Theta}$ on the $s+1$ coordinate is not $D_{s+1}$, then $\mathscr{R}_{\Theta}^{s+1}$ is a binary absorbing subuniverse due to the definition of a non-trivial absorbing subuniverse. Otherwise, choose any element $a\in D_{s+1}$ such that $\mathscr{R}_{\Theta}$ does not contain all homomorphisms sending $s+1$ to $a$ and consider the new $s$-ary relation $\mathscr{R}_{\Theta'}=\{(a_0,...,a_s)| (a_0,...,a_s,a)\in \mathscr{R}_{\Theta}\}$. Due to Lemma \ref{alalalur6456}, it is a non-trivial binary absorbing subuniverse for $D_0\times...\times D_{s}$ ($T$ is idempotent). But note that $\mathscr{R}_{\Theta'}$ is also a solution set to a specific CSP instance $\Theta'$ on $s$ variables: we just remove from $E_{\mathcal{X}}$ of $\Theta$ all edges adjacent to $x_{s+1}$ and for all $j<s+1$ restrict $E^{j(s+1)}_{\ddot{\mathcal{A}}}$ and $E^{(s+1)j}_{\ddot{\mathcal{A}}}$ to pairs ending and starting with $a$. Thus, we can apply the induction hypothesis. 
\end{proof}

The following lemma from \cite{krajicek_1995} gives us a more precise theory, which we will not define here. 
\begin{lemmach}[\cite{krajicek_1995}]
  For all $i\geq 1$, the theory $T^i_2$ proves the induction scheme IND for $\mathcal{B}(\Sigma^{b}_i)$-formulas.
\end{lemmach}

\begin{lemmach}[Lemma 7.5, \cite{zhuk2020proof}]
  Suppose $D^{(1)}$ is an absorbing reduction of a CSP instance $\Theta$ and a relation $\mathscr{R}_{\Theta}\leq D_{0}\times...\times D_{n-1}$ is subdirect. Then $V^0$ proves that $\mathscr{R}_{\Theta}^{(1)}$ is not empty.
\end{lemmach}
\begin{proof}
  Suppose the opposite. Then $\mathscr{R}_{\Theta}\cap D^{(1)}_{0}\times...\times D^{(1)}_{n-1}=\emptyset$, where for each $i<n$, $D^{(1)}_i$ is a binary absorbing subuniverse of $D_i$ with the term $T$. This means that for every homomorphism $H_{i_1}\in \mathscr{R}_{\Theta}$ there exists $i<n$ such that $H_{i_1}(i)=\langle i,a\rangle$, where $a\in D_i\backslash D^{(1)}_{i}$. Since $\mathscr{R}_{\Theta}$ is subdirect, for any $i<n$ and for any $b\in D^{(1)}_{i}$ there exists a homomorphism $H_{i_2}$ such that $H_{i_2}(i)=\langle i,b\rangle$. Composing these homomorphisms, since $T$ is a polymorphism, we again get a homomorphism $H_{i_3}=usepol_2(T,H_{i_1},H_{i_2})$ such that $H_{i_3}(i)=\langle i,c\rangle$ for some $c\in D^{(1)}_{i}$. Consider any $j\neq i <n$ such that $H_{i_3}(j)=\langle j,d\rangle$ with $d\notin D^{(1)}_{j}$. Again, since $\mathscr{R}_{\Theta}$ is subdirect, there must be a homomorphism $H_{i_4}$ such that $H_{i_4}(j)=\langle j,e\rangle$ for some $e\in D^{(1)}_{j}$. We compose these two homomorphisms and get $H_{i_5} =usepol_2(T,H_{i_3},H_{i_4})$ such that $H_{i_5}(i)=\langle i,f\rangle$ and $H_{i_5}(j)=\langle j,g\rangle$ with $f\in D^{(1)}_{i}$ and $g\in D^{(1)}_{j}$. Applying this procedure at most $n$ times, we will get a homomorphism in $\mathscr{R}_{\Theta}^{(1)}$, a contradiction.
\end{proof}

\subsubsection{Properties of a central subuniverse on $\mathbb{A}^n$}
We say that a solution set to a CSP instance $\Theta$ over $\Gamma_{\mathcal{A}}$ on $n$ variables, $\mathscr{R}_{\Theta}\leq D_0\times...\times D_{n-1}$ is a central absorbing subuniverse of $D_0\times...\times D_{n-1}$ if there exists a ternary term operation $S\in Pol(\Gamma_{\mathcal{A}})$ such that $\mathscr{R}_{\Theta}$ absorbs $D_0\times...\times D_{n-1}$ with $S$ and for any map $H\notin \mathscr{R}_{\Theta}$ such that for all $i<n$, $H(i)=\langle i,a_i\rangle$, the following conditions hold. Construct two new CSP instances, $\Theta_L$ and $\Theta_R$ as follows: 
\begin{itemize}
  \item We double the number of variables, $D_{L}=D_R=\{D_0,...,D_{2n-1}\}$ with $D_i=D_{n+i}$ for every $i<n$;
  \item For the instance $\Theta_L$ we copy digraph $\mathcal{X}$ for the first $n$ variables, and we join it with a path $\mathcal{P}_n$ of length $n$, namely $E_{\mathcal{P}_n}(x_n-1,x_n),...,E_{\mathcal{P}_n}(x_{2n-2},x_{2n-1})$. We copy $\ddot{\mathcal{A}}$ for the first $n$ variables and set the constraints for the path: $E^{x_n-1,x_n}_{\ddot{\mathcal{A}}_n}$ as a full relation and for the next edges $E^{x_n,x_{n+1}}_{\ddot{\mathcal{A}}_n}=\{(a_0,a_1)\}$, $E^{x_{n+1},x_{n+2}}_{\ddot{\mathcal{A}}_n}=\{(a_1,a_2)\}$,..., $E^{x_{2n-2},x_{2n-1}}_{\ddot{\mathcal{A}}_n}=\{(a_{n-2},a_{n-1})\}$. Note that such a CSP instance is a CSP instance over $\Gamma_{\mathcal{A}}$ since $\Omega$ is idempotent and every single element of $D_0\times...\times D_{n-1}$ is a subuniverse;
  \item For the instance $\Theta_R$ we do the same but in inversed manner (copy $\mathcal{X}$ and $\ddot{\mathcal{A}}$ for the variables $n,...,2n-1$);
\end{itemize}
Then from solutions to $\Theta_L$ and $\Theta_R$ by applying $\Omega$ we cannot generate the map $H$ such that for $i<n$, $H(i)=H(n+i)=\langle i,a_i\rangle$. Note that to define this fact, we need the third-order induction. The number of maps on $2n$ variables is $l^{2n}$, so to define the generated algebra $\mathscr{Cl}_{\mathscr{R}_{\Theta_L}\cup \mathscr{R}_{\Theta_R}}$ it is needed to consider at most $l^{2n}$ steps which we encode by strings:
\begin{equation}
  \begin{gathered}
     \hspace{0pt}\forall H\leq\langle 2n\langle 2n,l\rangle \rangle,\,\mathscr{Cl}^{[\emptyset]}_{\mathscr{R}_{\Theta_L}\cup \mathscr{R}_{\Theta_R}}(H)\iff \mathscr{R}_{\Theta_L}(H)\vee \mathscr{R}_{\Theta_R}(H)\wedge\\
     \forall \emptyset\leq T<2n\lceil log_2l\rceil, \forall H\leq\langle 2n\langle 2n,l\rangle \rangle,\\
     \mathscr{Cl}^{[S(T)]}_{\mathscr{R}_{\Theta_L}\cup \mathscr{R}_{\Theta_R}}(H)\iff\mathscr{Cl}^{[T]}_{\mathscr{R}_{\Theta_L}\cup \mathscr{R}_{\Theta_R}}(H) \vee\\
     \hspace {0pt}\vee \exists H_1,...,H_m\leq\langle 2n\langle 2n,l\rangle \rangle,\,\mathscr{Cl}^{[T]}_{\mathscr{R}_{\Theta_L}\cup \mathscr{R}_{\Theta_R}}(H_1)\wedge...\wedge \mathscr{Cl}^{[T]}_{\mathscr{R}_{\Theta_L}\cup \mathscr{R}_{\Theta_R}}(H_m)\wedge \\
     \hspace {0pt}\wedge \omega(H_1,...,H_m) = H.
  \end{gathered}
\end{equation}
The analogous definition can be formulated for any projection of the solution set $\mathscr{R}^{i_1,...,i_s}_{\Theta}$. 

\begin{lemmach}[Composed Lemma 7.6, \cite{zhuk2020proof}, and Theorem 6.9, \cite{zhuk2020strong}]
Suppose $\mathscr{R}_{\Theta}$ is defined by a $pp$-formula $\Theta(x_0,...,$ $x_{n-1})$ and $\Theta'$ is obtained from $\Theta$ by replacement of some constraint relations $\rho_1,...,\rho_s$ by constraint relations $\rho'_1,...,\rho'_s$ such that $\rho_k'$
is a central subuniverse for $\rho_k$ with a term operation $S$ for every $k$. Then $V^1$ proves that the relation $\mathscr{R}_{\Theta'}$ defined by
$\Theta'(x_0,...,x_{n-1})$ is a central subuniverse for $\mathscr{R}_{\Theta}$ with the term operation $S$.
\end{lemmach}
\begin{proof}
  Consider two CSP instances $\Theta = (\mathcal{X},\ddot{\mathcal{A}})$ and $\Theta' = (\mathcal{X}',\ddot{\mathcal{A}}')$, where $\mathcal{X}'=\mathcal{X}$ (again, the analogous reasoning can be applied to projections). Due to the assumption, there exists a ternary term $S\in Pol(\Gamma_{\mathcal{A}})$ such that for each $i<n$, $D'_{i}\subseteq D_i$ ternary absorbs $D_i$ and for all $i,j<n$ with $E_{\mathcal{X}}(i,j)$, $E^{ij}_{\ddot{\mathcal{A}}'}\subseteq E^{ij}_{\ddot{\mathcal{A}}}$ ternary absorbs $E^{ij}_{\ddot{\mathcal{A}}}$. The defining relation is analogous to (\ref{djdjjduuuhs6745ed}). Also, for each $i<n$, and for any $a\in D_i\backslash D'_i$, $(a,a)\notin Sg(X^{i}_{(a)})$, where
 \begin{equation}
X^{i}_{(a)}=\{\{a\}\times D'_i,D'_i\times \{a\}\},
\end{equation} 
  and for all $i,j<n$ with $E_{\mathcal{X}}(i,j)$, and for every $(a,b)\in E^{ij}_{\ddot{\mathcal{A}}}\backslash E^{ij}_{\ddot{\mathcal{A}}'}$ we have $(a,b,a,b)\notin Sg(X^{ij}_{(a,b)})$, where 
\begin{equation}
  X^{ij}_{(a,b)}=\{(a,b)\}\times E^{ij}_{\ddot{\mathcal{A}}'}\cup E^{ij}_{\ddot{\mathcal{A}}'}\times \{(a,b)\}.
\end{equation}
We will show how to define $Sg(X^{ij}_{(a,b)})$ analogously to a central subuniverse we defined before, using the closure operator $Cl(X^{ij}_{(a,b)})$. Since $\mathbb{A}$ is finite of size $l$ and $|X^{ij}_{(a,b)}|=2|E^{ij}_{\ddot{\mathcal{A}}'}|$, we do not need more than $(l^4-2|E^{ij}_{\ddot{\mathcal{A}}'}|)$ steps. Not to depend on $E^{ij}_{\ddot{\mathcal{A}}'}$, choose the value $l^4$. For set $X^{ij}_{(a,b)}\leq \langle\langle l,l\rangle, \langle l,l\rangle\rangle$, iteratively define the following set $Cl^{t}_{X^{ij}_{(a,b)}}$ up to $l^4$:
\begin{equation}
\begin{gathered}
 \hspace {0pt}\forall c,d,e,f<l,\,Cl^{0}_{X^{ij}_{(a,b)}}(c,d,e,f)\iff X^{ij}_{(a,b)}(c,d,e,f)\wedge\\
 \hspace {0pt}\wedge \forall 0<t<l^4, \forall c,d,e,f<l,\, Cl^{t}_{X^{ij}_{(a,b)}}(c,d,e,f)\iff Cl^{t-1}_{X^{ij}_{(a,b)}}(c,d,e,f)\vee\\
 \hspace{0pt}\vee\exists c_1,...,c_m,d_1,...,d_m,e_1,...,e_m,f_1,...,f_m\in A,\\
 \hspace {0pt}Cl^{t-1}_{X^{ij}_{(a,b)}}(c_1,d_1,e_1,f_1)\wedge...\wedge Cl^{t-1}_{X^{ij}_{(a,b)}}(c_m,d_m,e_m,f_m)\wedge\\
 \hspace{0pt}\wedge\Omega(c_1,...,c_m)=c\wedge \Omega(d_1,...,d_m)=d\wedge \Omega(e_1,...,e_m)=e\wedge \Omega(f_1,...,f_m)=f.
\end{gathered}
\end{equation}
Therefore, 
\begin{equation}\label{akakudu467763}
\begin{gathered}
     \forall i,j<n, \forall a,b<l,\,(E_{\mathcal{X}}(i,j)\wedge E^{ij}_{\ddot{\mathcal{A}}}(a,b)\wedge \neg E^{ij}_{\ddot{\mathcal{A}}'}(a,b))\rightarrow \neg Cl^{l^4}_{X^{ij}_{(a,b)}}(a,b,a,b).
\end{gathered}
\end{equation}
That $\mathscr{R}_{\Theta'}$ absorbs $\mathscr{R}_{\Theta}$ with ternary operation $S$ can be proved as in Lemma \ref{alalalur6456}. Suppose that $\mathscr{R}_{\Theta'}$ is not a central subuniverse of $\mathscr{R}_{\Theta}$. 
Then there exists a homomorphism $H$ from $\mathcal{X}$ to $\ddot{\mathcal{A}}$, sending each $i<n$ to $\langle i,a_i\rangle$, such that it is not a homomorphism from $\mathcal{X}'$ to $\ddot{\mathcal{A}}'$, and if we construct two instances $\Theta_L'$ and $\Theta_R'$ as above, the subalgebra generated by these two instances contains a homomorphism $H'$ sending both $i$ and $n+i$ to $\langle i,a_i\rangle$ for each $i<n$. It would mean that for all $i,j<n$ such that $E_{\mathcal{X}}(i,j)$, elements of the form $(a_i,a_j,a_i,a_j)\in D_{i}\times D_{j}\times D_{n+i}\times D_{n+j}$ must belong to sets
$Cl^{l^4}_{X^{ij}_{(a_i,a_j)}}$. But at least one $(a_i,a_j)$ must be from $E^{ij}_{\ddot{\mathcal{A}}}\backslash E^{ij}_{\ddot{\mathcal{A}}'}$ (otherwise $H$ is a homomorphism from $\mathcal{X}'$ to $\ddot{\mathcal{A}}'$). That contradicts with (\ref{akakudu467763}).
\end{proof}

\begin{remark}
  Note that in the above proof, we do not even use the definition of the third-order object. We lowered requirements to second-order objects and showed the contradiction. Thus, we can remain in $V^1$. 
\end{remark}

\begin{corollary}[Corollary 6.9.2, \cite{zhuk2020strong}]\label{akjshdgfuyfgwef}
Suppose $\mathscr{R}_{\Theta}\leq D_0\times...\times D_{n-1}$ is a relation such that $pr_0(\mathscr{R}_{\Theta})=D_0$ and $C=pr_0((C_0\times...\times C_{n-1})\cap \mathscr{R}_{\Theta})$, where $C_i$ is a central subuniverse in $D_i$ for every $i$. Then $V^1$ proves that $C$ is a central subuniverse in $D_0$.
\end{corollary}

\begin{corollary}[Corollary 6.9.3, \cite{zhuk2020strong}]\label{66666666666666}
Suppose $\mathscr{R}_{\Theta}\leq D_0\times...\times D_{n-1}$ and $C_i$ is a central subuniverse for every $i$. Then $V^1$ proves that $(C_0\times...\times C_{n-1})\cap \mathscr{R}_{\Theta}$ is a central subuniverse for $\mathscr{R}_{\Theta}$.
\end{corollary}

\begin{lemmach}[Lemma 7.7, \cite{zhuk2020proof}] Suppose $\mathscr{R}_{\Theta}$ is a non-trivial center of $D_0\times...\times D_{n-1}$. Then $V^2$ proves that for some $i$ there exists a non-trivial center $C_i$ of $D_i$.
  
\end{lemmach}

\subsubsection{Properties of a PC subuniverse on $\mathbb{A}^n$}
While considering the properties of a PC subuniverse on $D_0\times...\times D_{n-1}$, we keep in mind that we further use all auxiliary lemmas in the proof of the main statements about the next reduction, and all reductions are constructed based on separate domains upward to subuniverses on their product. That is, there is no need to consider arbitrary PC subuniverses on $D_0\times...\times D_{n-1}$ or on the solution set $\mathscr{R}_{\Theta}$. The problem here is that we have no definition of a PC algebra for an arbitrary product or a PC congruence on that product (recall that in the definition of PC algebra, we use fixed constraint language $\Gamma_{\mathcal{A}}$). Thus, any time in the proofs we consider an arbitrary algebra $\mathbb{D}$ and its arbitrary PC congruence $\sigma$, we may assume that $\mathbb{D}$ is a subuniverse $\mathscr{R}\leq D_0\times...\times D_{n-1}$ and $\sigma$ is an extended congruence for some PC congruence $\sigma_i$ on a domain $D_i$.

\begin{lemmach}[Lemma 6.20, \cite{zhuk2020strong}]\label{lalaksjdf}
  Suppose $\mathscr{R}_{\Theta}$ is a subdirect relation on $D_0\times...\times D_{n-1}$ and $E_i$ is a PC subuniverse of $D_i$ for every $i$. Then $W^1_1$ proves that $(E_0\times...\times E_{n-1})\cap \mathscr{R}_{\Theta}$ is a PC subuniverse of $\mathscr{R}_{\Theta}$.
\end{lemmach}
\begin{proof}

Due to the assumption, there is a PC subuniverse $E_i$ of $D_i$ for every $i$. If some $E_i$ is empty, then we are done. Otherwise, each $E_i$ is a block of some congruence $\delta_i=\cap_j\sigma_{i_j}$ on $D_i$, where $\sigma_{i_1},...,\sigma_{i_s}$ are PC congruences on $D_i$. We can extend every $\sigma_{i_j}$ to the product $D_0\times...\times D_{n-1}$, and consider the congruence $\mathscr{C}_{\theta_i}=\cap_{k\neq i}\mathscr{C}_{\nabla^{ext}_{D_k}}\cap \mathscr{C}_{\sigma_{i_j}^{ext}}$. Since $\mathscr{R}_{\Theta}$ is subdirect, it is easy to check that a third-order map $\mathscr{H}$ that sends every $H\in \mathscr{R}_{\Theta}/\mathscr{C}_{\theta_i}$ to $a\in D_i/\sigma_{ij}$ such that $H(i) = \langle i,a\rangle$, is an isomorphism from $(\mathscr{R}_{\Theta}/\mathscr{C}_{\theta_i}, \mathscr{F}_{\Omega/\theta_i})$ to $(D_i/\sigma_{ij},\Omega/\sigma_{ij})$. So $(\mathscr{R}_{\Theta}/\mathscr{C}_{\theta_i}, \mathscr{F}_{\Omega/\theta_i})$ acts like a PC algebra, and we will call $\mathscr{C}^{\mathscr{R}_{\Theta}}_{\sigma_{i_j}^{ext}}$ restricted to $\mathscr{R}_{\Theta}$ an extended PC congruence for $(\mathscr{R}_{\Theta}, \mathscr{F}_{\Omega})$. 

It follows that for each $E_i$, $\mathscr{E}_i^{ext}\cap \mathscr{R}_{\Theta}$ is an intersection of blocks of extended PC congruences restricted to $(\mathscr{R}_{\Theta}, \mathscr{F}_{\Omega})$, as well as $\mathscr{E}_0^{ext}\cap...\cap \mathscr{E}_{n-1}^{ext}\cap \mathscr{R}_{\Theta}$, and we can call it an extended PC subuniverse. By the definition of every $\mathscr{E}_i^{ext}$, it is just $E_0\times...\times E_{n-1}\cap \mathscr{R}_{\Theta}$.
\end{proof}
Note that from the same reasoning, it follows that $E_0\times...\times E_{n-1}$ is an extended PC subuniverse of $D_0\times...\times D_{n-1}$.

\begin{lemmach}[Lemma 6.18, \cite{zhuk2020strong}]\label{a;lsdkfjyr647ryv}
  Suppose that $D$ is a PC algebra and $\mathscr{R}_{\Theta}\leq D^n$ contains all constant tuples $(a,a,...,a)$. Then $V^0$ proves that $\mathscr{R}_{\Theta}$ can be represented as a conjunction of binary relations of the form $x_i=x_j$. 
\end{lemmach}
\begin{proof}
Since $\mathscr{R}_{\Theta}\leq D^n$ contains all constant tuples $(a,a,...,a)$, every domain $D_i$ of a CSP instance $\Theta$ is equal to $D$, and every binary constraint $E^{ij}_{\ddot{\mathcal{D}}}$ contains a diagonal relation $\Delta_{ij}$. Recall that an algebra is PC if there exists a ternary discriminator $P$ such that $Pol_{3}(P,D,\Gamma^{diag}_{\mathcal{D}})$. Therefore, $P$ must preserve $\mathscr{R}_{\Theta}$. We want to show that every $E^{ij}_{\ddot{\mathcal{D}}}$ that is not a full binary relation is equal to relation $\Delta_{ij}$. Suppose that for some $i,j<n$ $E^{ij}_{\ddot{\mathcal{D}}}$ is neither a full nor diagonal relation. Then there must exist some $a\neq b$ such that $E^{ij}_{\ddot{\mathcal{D}}}(a,b)$. But since $P$ preserves $E^{ij}_{\ddot{\mathcal{D}}}$, $(P(a,a,b),P(a,b,b))=(b,a)\in E^{ij}_{\ddot{\mathcal{D}}}$, and for every $c\neq b$, $(P(a,b,c),P(a,a,c))=(a,c)\in E^{ij}_{\ddot{\mathcal{D}}}$. Thus, $E^{ij}_{\ddot{\mathcal{D}}}$ is a full relation. \end{proof}

The following lemma follows from the previous lemma and some additional reasoning.
\begin{lemmach}[Lemma 6.19, \cite{zhuk2020strong}]\label{fkjerfqlerfhuurfhlrf}
  Suppose that $\mathscr{R}_{\Theta}\leq D_0\times...\times D_{n-1}$ is subdirect, $D_i$ is a PC algebra without non-trivial binary absorbing and central subuniverses for every $i\in\{1,...,n-1\}$, and $D_0$ has no non-trivial central subuniverse. Then $V^1$ proves that $\mathscr{R}_{\Theta}$ can be represented as a conjunction of binary relations $\delta_1(x_{i_1},x_{j_1}),...,\delta_s(x_{i_s},x_{j_s})$, where for every $l\leq s$ the first variable of $\delta_l$ is uniquely determined whenever $i_l\neq 1$ and the
second variable of $\delta_l$ is uniquely determined whenever $j_l\neq 1$.
\end{lemmach}

\begin{lemmach}[Lemma 6.21, \cite{zhuk2020strong}]\label{alalsokd7ju}
Suppose that $\mathscr{R}_{\Theta}$ is a subdirect relation on $D_0\times...\times D_{n-1}$, $E_i$ is a PC subuniverse of $D_i$ for all $i$ and $E=pr_0((E_0\times...\times E_{n-1})\cap \mathscr{R}_{\Theta})$. 
  Then $V^1$ proves that $E_0$ is a PC subuniverse of $D_0$.
\end{lemmach}
\begin{proof}
We consider $\mathscr{R}_{\Theta}$ to be a solution set to some CSP instance $\Theta$ on $n$ domains such that for all $i,j<n$ there is a constraint $E_{\mathcal{X}}(i,j)$, but for some of them $E^{ij}_{\ddot{\mathcal{A}}}$ are full relations. By Lemma \ref{alskjdiojkgfkj} (or by Definition \ref{';s;sdllfjjg}) each $E_i$ is a block of congruence $\delta_i$ such that there are PC congruences $\sigma_{i_0},...,\sigma_{i_{k-1}}$ with $k\leq log_2l$ and
  $$D_i/\delta_i\cong D_i/\sigma_{i_0}\times...\times D_i/\sigma_{i_{k-1}}.
  $$
Then we can consider the factorized instance $\Theta_{PC}$, constructed as follows. The instance digraph $\mathcal{X}_{PC}=\mathcal{X}$ does not change. For a target digraph $\ddot{\mathcal{A}}_{PC}$, the domain set is $D_{PC}=\{D_0,D_1/\delta_1,...,D_{n-1}/\delta_{n-1}\}$, i.e., we factorize every domain except the first one (or equivalently factorize it by $\Delta_0$). Constraint relations are defined by a set
$E_{\ddot{\mathcal{A}}_{PC}}$ such that
\begin{equation}
  \begin{gathered}
     \hspace{0pt}\forall 0<i,j<n,\, E^{ij}_{\ddot{\mathcal{A}}_{PC}}(a,b)\iff D_i/\delta_i(a)\wedge D_j/\delta_j(b)
    \wedge \\
    (\exists c,d<l,\,\,\delta_i(a,c)\wedge \delta_j(b,d)\wedge E^{ij}_{\ddot{\mathcal{A}}}(c,d)),\\
   \hspace {0pt}\forall i<n,\,E^{1i}_{\ddot{\mathcal{A}}_{PC}}(a,b)\iff D_0(a)\wedge D_i/\delta_i(b)
    \wedge (\exists c<l,\,\,\delta_i(b,c)\wedge E^{1i}_{\ddot{\mathcal{A}}}(a,c)),\\
   \hspace {0pt}\forall i<n,\,E^{i1}_{\ddot{\mathcal{A}}_{PC}}(a,b)\iff D_0(a)\wedge D_i/\delta_i(b)
    \wedge (\exists c<l,\,\,\delta_i(b,c)\wedge E^{i1}_{\ddot{\mathcal{A}}}(c,a)).\\ 
  \end{gathered}
\end{equation}
Then there is a canonical homomorphism $H_c\leq \langle \langle l,n\rangle,\langle l,n\rangle\rangle$ from the target digraph $\ddot{\mathcal{A}}$ to the target digraph $\ddot{\mathcal{A}}_{PC}$ such that for any homomorphism $H$ from $\mathcal{X}$ to $\ddot{\mathcal{A}}$, the map $H_{PC}$ defined as
$$
H_{PC}(i)=\langle i,a\rangle \iff \exists b<l, (H(i)= \langle i,b\rangle\wedge H_c(\langle i,b\rangle)=\langle i,a\rangle),
$$
is a homomorphism from $\mathcal{X}$ to $\ddot{\mathcal{A}}_{PC}$. It is easy to see that $pr_0((E_0\times E_1\times...\times E_{n-1})\cap \mathscr{R}_{\Theta}) = pr_0((E_0\times e_1\times...\times e_{n-1})\cap \mathscr{R}_{\Theta_{PC}})$ where for each $0<i<n$, $e_i$ is the representative of the class $E_i$ and that $\mathscr{R}_{\Theta_{PC}}$ is subdirect. 

For any $0<i<n$, we can find $M_i\in \mathrm{M}_{D_i,\delta_i\sigma_{i_0},...,\sigma_{i_{k-1}}}$ such that 
$$ISO_{alg}(D_i/\delta_i,\Omega/\delta_i,D_i/\sigma_{i_0}\times...\times D_i/\sigma_{i_{k-1}},\Omega/\cap_{j}\sigma_{i_j} ,M_i).
$$
Let us combine all of these maps into one set $M$. For every $i$ such that $M_i\in \mathrm{M}_{D_i,\delta_i\sigma_{i_0}},...,$ $_{\sigma_{i_{k-1}}}$ for $k<s=log_2l$ we add to the end $s-k$ trivial algebras $D_{i_{k}},...,D_{i_{s-1}}$ (containing one element $0$), set $\sigma_{i_{k}},...,\sigma_{i_{s-1}}$ to be trivial congruences and extend $M_i$ to the $s+1$ -ary set. Then for all $a_0,...,a_{s-1}<l$, 
\begin{equation}
  \begin{gathered}
     \forall 0<i<n,\,\forall a\in D_i/\delta_i,\,M(a)=(a_0,...,a_{s-1})\iff M_i(a)=(a_0,...,a_{s-1})\\
     \hspace{0pt}i=0,\,\forall a\in D_0,\,M(a)=(a,0,...,0).
  \end{gathered}
\end{equation}
Consider a CSP instance $\Theta'_{PC}$ on $ns$ variables, with domain set 
$$D'_{PC}=\{D_0,0,...,0, D_1/\sigma_{1_0},..., D_1/\sigma_{1_{s-1}},..., D_{n-1}/\sigma_{{n-1}_0},...,D_{n-1}/\sigma_{{n-1}_{s-1}}\}$$ 
such that $H_{PC}'\in \mathscr{R}_{\Theta'_{PC}}$ if and only if $H_{PC}\in \mathscr{R}_{\Theta_{PC}}$ where $H_{PC}'$ is a map from $[ns]$ to $[D'_{PC}]$ defined from a homomorphism $H_{PC}$ as follows: 
\begin{equation}\label{alskjdyhrf5564}
  \begin{gathered}
     \hspace {0pt}i=0,\,H'_{PC}(0) = \langle 0,a\rangle\iff H_{PC}(0) = \langle 0,a\rangle,\\
     \hspace {0pt}\forall 0<i<s,\,H'_{PC}(i) = \langle i,0\rangle,\\
     \hspace{0pt}\forall 0<j<n,\forall k<s,\,\forall a\in D_j/\sigma_{j_k},\, H'_{PC}(js+k) = \langle js+k,a\rangle\iff\\
     \iff \exists b\in D_j/\delta_j,\exists b_0\in D_{j}/\sigma_{j_0},...,\exists b_{k-1}\in D_{j}/\sigma_{j_{k-1}},\\
     \exists b_{k+1}\in D_{j}/\sigma_{j_{k+1}},...,\exists b_{s-1}\in D_{j}/\sigma_{j_{s-1}},\\
     \hspace {0pt}H_{PC}(j)=\langle j,b\rangle\wedge M(b)=(b_0,...,b_{k-1},a,b_{k+1},...,b_{s-1})
  \end{gathered}
\end{equation}
Strictly speaking, the instance $\Theta'$ is not an instance over language $\Gamma_{\mathcal{A}}$, but we still can define every $2s$-ary relation $R^{ij}_{\ddot{\mathcal{A}}'_{PC}}$,
\begin{equation}\label{a00a09s88))(}
\begin{gathered}
 R^{ij}_{\ddot{\mathcal{A}}'_{PC}}(a_0,...,a_{s-1}, b_0,...,b_{s-1})\iff \exists a\in D/\delta_i\exists b\in D_j/\delta_j,\,E^{ij}_{\ddot{\mathcal{A}}_{PC}}(a,b)\wedge\\
 \hspace {0pt}\wedge M(a)=(a_0,...,a_{s-1})\wedge M(b)=(b_0,...,b_{s-1}).
\end{gathered}
\end{equation}
Moreover, we can apply to $R^{ij}_{\ddot{\mathcal{A}}'_{PC}}$ a similar reasoning as in Lemma \ref{a;lsdkfjyr647ryv}. The solution set $\mathscr{R}_{\Theta'_{PC}}$ is subdirect since $\mathscr{R}_{\Theta_{PC}}$ is subdirect. Since every $D_j/\sigma_{j_k}$ is a PC algebra, it follows that by Lemma \ref{fkjerfqlerfhuurfhlrf}, $\mathscr{R}_{\Theta'_{PC}}$ can be represented as a conjunction of binary relations $\delta_1(x_{i_1},x_{j_1}),...,\delta_s(x_{i_s},x_{j_s})$, where for every $l\leq s$ the first variable of $\delta_l$ is uniquely-determined whenever $i_l\neq 1$ and the second variable of $\delta_l$ is uniquely-determined whenever $j_l\neq 1$. Thus, for all $0<j<n, k<s$, for every $a\in D_0$ there exists a unique $b\in D_j/\delta_{j_{k}}$ such that $E^{1(js+k)}_{\ddot{\mathcal{A}}_{PC}'}(a,b)$ (and analogously for the relation $E^{(js+k)1}_{\ddot{\mathcal{A}}_{PC}'}$). It follows that any such relation divides $D_0$ into $|D_j/\delta_{j_{k}}|$ classes, and we can check that this is a PC congruence on $D_0$. Thus, $pr_0((E_0\times e_1\times...\times e_{n-1})\cap \mathscr{R}_{\Theta_{PC}})$ can be represented as an intersection of blocks of $ns$ PC congruences (some of them can be trivial).
\end{proof}

\subsubsection{Properties of a linear subuniverse on $\mathbb{A}^n$}

It is known that the ability to simulate an affine CSP (or historically the ability to count) adds substantial complexity to the problem. Structures that cannot count are all tractable and can even be solved by a simple constraint propagation algorithm. 
\begin{notation}
  For linear algebras, we shall adhere to the following notation. We will continue to denote elements of different domains $D_i,D_j$ by $a_i,a_j$. If we consider several elements of the same domain, we add an index after the index of the domain, for example, $a_{i1},...,a_{im}$. To represent an element $a_{ij}$ as an element of the product of $k\leq log_2l$ prime fields $Z_{p_0},...,Z_{p_{k-1}}$, we add the superscript $\bar{a}^k_{ij} = (a_{ij}^0,...,a_{ij}^{k-1})$.
\end{notation}

The proof complexity of linear algebra (in the sense of a branch of mathematics) was well studied in \cite{SOLTYS2004277}, or in \cite{Thapen2005-THAWTO}. In particular, Gaussian elimination was considered and was shown to be formalizable in theory $V^1$. The detailed proof of the following lemma can be found, for example, in \cite{gaysin2023proof}. We show how to formalize it in $V^1$.

\begin{lemmach} [Lemma 7.20, \cite{zhuk2020proof}]\label{AffineSubspaces}
Suppose that the relation $\mathscr{R}_{\Theta}\leq (\mathbb{Z}_{p_1})^{n_1}\times ... \times (\mathbb{Z}_{p_k})^{n_k}$ is preserved by $x_1+...+x_m$, where $p_1,...,p_k$ are distinct prime numbers dividing $m-1$ and $\mathbb{Z}_{p_i} = (\mathbb{Z}_{p_i},x_1+...+x_m)$ for every $i$. Then $V^1$ proves that $\mathscr{R}_{\Theta} = L_1\times...\times L_k$, where each $L_i$ is an affine subspace of $(\mathbb{Z}_{p_i})^{n_i}$.
\end{lemmach}

\begin{proof}
Consider any CSP instance $\Theta$ on $n$ variables such that each $D_i$ is a linear algebra, i.e. $LinA(D_i,\Omega)$. That is, for every $i<n$ there are some $k\leq log_2l$, primes $p_0,...,p_{k-1}<l$, and an isomorphism $M_i\in \mathrm{M}_{A,\Delta_i,p_0,...,p_{k-1}}$ from $(D_i,\Omega)$ to $(Z_{p_0}\times...\times Z_{p_{k-1}},\bar{a}^k_{i1}+...+\bar{a}^k_{im})$ such that
$$M_i(a_{ij}) = \bar{a}^k_{ij}=(a_{ij}^0,...,a_{ij}^{k-1}).
$$
As for the PC subuniverses, to unify all $M_i$, for every $k<s=log_2l$, we add $s-k$ trivial algebras $\mathbb{Z}_{p_{k}},...,\mathbb{Z}_{p_{s-1}}$, representing their elements as $0$'s. Thus, $M_i(a_{ij})=\bar{a}^s_{ij}$. Then we can construct $M$. For all $a_i^0,...,a_i^{s-1}<l$,
\begin{equation}
  \begin{gathered}
     \forall 0<i<n,\,\forall a_i\in D_i,\,M(a_i)=(a_i^0,...,a_i^{s-1})\iff M_i(a_i)=(a_i^0,...,a_i^{s-1}).  \end{gathered}
\end{equation}
We consider CSP instance $\Theta_L$ on $ns$ domains such that every solution to $\Theta$, translated naturally (analogously to (\ref{alskjdyhrf5564})) is a solution to $\Theta_L$ and vice versa. Again, it is not a CSP instance over language $\Gamma_{\mathcal{A}}$, but all $2s$-ary relations $R^{ij}_{\ddot{\mathcal{A}}_{L}}$ can be easily defined; see
(\ref{a00a09s88))(}). The most important thing is that these relations are preserved by $m$-ary sum.

For this proof, we do not need to collect equal $Z_{p_j}$ from different domains to a group. We define a vector space on $Z_{p_0}\times...\times Z_{p_{ns-1}}$ as any subset of maps from $[ns]$ to $[D_L] = [Z_{p_0},...,Z_{p_{ns-1}}]$ such that it contains 'zero map' $H$ sending all $i<ns$ to $0$, and is closed under $+$, i.e. for any two maps $H_1,H_2$, the map $H_3=H_1+H_2$ such that
\begin{equation}
  \begin{gathered}
    (H_1+H_2)(i)= H_3(i)=\langle i,a\rangle\iff \exists b,c\in D_i,\\
    H_1(i)=\langle i,b\rangle\wedge H_2(i)=\langle i,c\rangle\wedge a=b+c,
  \end{gathered}
\end{equation}
is also in that set. We define an affine subspace $\mathscr{R}_{\Theta_L}$ of $Z_{p_0}\times...\times Z_{p_{ns-1}}$ as a shift of some linear subspace, i.e. such a set that for any map $H\in \mathscr{R}_{\Theta_L}$, the set of all maps $H'$ such that $H'+H\in\mathscr{R}_{\Theta_L}$ contains zero map and is closed under $+$. Note that when we are talking about solution sets, it is a second-order definition:
\begin{equation}
  \begin{gathered}
 \hspace {0pt}AffSubS(\mathscr{R}_{\Theta_L})\iff \forall H\leq \langle n\langle n,l\rangle\rangle,\forall H_1,H_2\leq \langle n\langle n,l\rangle\rangle,\,\ddot{HOM}(\mathcal{X},\ddot{\mathcal{A}},H)\wedge \\
 \wedge \ddot{HOM}(\mathcal{X},\ddot{\mathcal{A}},H+H_1)\wedge \ddot{HOM}(\mathcal{X},\ddot{\mathcal{A}},H+H_2)\rightarrow \\
 \rightarrow\ddot{HOM}(\mathcal{X},\ddot{\mathcal{A}},H+(H_1+H_2)).
  \end{gathered}
\end{equation}
Then it is easy to show that the solution to $\Theta_L$ is an affine subspace defining $x-y+z (mod \,p_j) = \Omega(x,z,y,...,y)$. 
\end{proof}

To prove some auxiliary lemmas and theorems, in addition to Definition \ref{a'a'a's;ldkfijgt}, we need to define a linear algebra on a product $D_0\times...\times D_{n-1}$ of algebras of size at most $|A|$. We could define a linear algebra $\mathscr{R}_{\Theta}$ on $n$ domains, each of which is isomorphic to a product of $k\leq log_2l$ prime fields, as a solution set to some CSP instance $\Theta$ or as a set of maps closed under $x_1+...+x_m$. This defining relation is second-sorted and it would give us that $V^1$ proves that the set of linear algebras is closed under taking subalgebras. The problem here is that the resulting direct product isomorphic to relation $\mathscr{R}_{\Theta}$ will not be related to the initial CSP instance $\Theta$, it is considered exclusively as an algebra. We know that any relation preserved by $x_1+...+x_m$ can be represented as a system of linear equations over at most $ns$ variables, where $s=log_2l$, and this system can be solved by Gaussian elimination. Different domains $\mathbb{Z}_{p_i}$ cannot be mixed in one subsystem of equations, but variables over the same $\mathbb{Z}_{p_i}$ representing different $x_i,x_j$ can. For example, suppose that domains $D_i,D_j$ are both isomorphic to $\mathbb{Z}_3$, and that the solution to CSP instance $\Theta$ for these two variables contains the affine subspace (coset) $\{(1,1),(2,0),(0,2)\}$. This is isomorphic to $(\mathbb{Z}_3,\Omega)$, but we lose a domain. The same would happen while taking the quotient and in the proofs of some results about linear algebras on $n$ domains, we need the closeness of a set of linear algebras under taking quotients, especially extended ones. So we cannot avoid an isomorphism between third-order objects. Recall that we allow trivial algebras $\mathbb{Z}_1$ and that solution set to any CSP instance $\Theta_{null}$ on $n$ domains is a full relation.

\begin{definition}[Linear algebra on a product of algebras of size at most $l$]\label{aklsdry564}For an algebra $(\mathscr{R},\mathscr{F})$ defined on a set of maps from $[k]$ to $[D'_0,...,D'_{k-1}]$, where $k\leq n$, we say that $\mathscr{R}$ is a linear algebra on $n$ domains if there exists a set of $n$ domains $D$ and the set of $n$ $m$-ary operations $F$ on these domains such that every algebra $(D_i,F_i)$ is linear and there is an isomorphism from $(\mathscr{R},\mathscr{F})$ to an algebra $(\mathscr{D}_0\times...\times \mathscr{D}_{n-1},\mathscr{F}_{F_0,...,F_{n-1}})$, where $\mathscr{D}_0\times...\times \mathscr{D}_{n-1}$ and $\mathscr{F}_{F_0,...,F_{n-1}}$ are defined as in (\ref{kqwhf;48r;3ihert})-(\ref{l;alsd'jqf;hd}):
\begin{equation}
  \begin{gathered}
     LinA(\mathscr{R},\mathscr{F})\iff \exists D\leq \langle n,l \rangle ,\exists F\leq \langle n,\underbrace{l,l,...,l}_{m\text{ times}}\rangle,\forall i<n,\ LinA(D_i,F_i) \wedge \\
     \hspace{0pt}\wedge ISO_{alg}^{3,3}(\mathscr{R},\mathscr{F}, \mathscr{D}_0\times...\times \mathscr{D}_{n-1},\mathscr{F}_{F_0,...,F_{n-1}}).
  \end{gathered}
\end{equation}
Note that the defining relation is third-order because of the relation $ISO_{alg}^{3,3}$. In an obvious way, we can define a linear and a minimal linear congruence $\mathscr{C}_{\theta}$ for any algebra on $n$ domains $(\mathscr{R},\mathscr{F})$.
\end{definition}

From Lemma \ref{AffineSubspaces}, Definition \ref{a'a'a's;ldkfijgt}, and Definition \ref{aklsdry564}, we conclude the following corollary.

\begin{corollary}
  $V^1$ proves that the set of linear algebras is closed under taking subalgebras and quotients. $W^1_1$ proves that the set of linear algebras on a product of algebras of size at most $|A|$ is closed under taking subalgebras and quotients.
\end{corollary}

In the presence of the third-order definition of linear algebra on $n$ domains, the following lemmas can be proved almost exactly as in \cite{zhuk2020proof}.

\begin{lemmach}[Lemma 7.21, \cite{zhuk2020proof}]
  $W^1_1$ proves that a linear algebra has no non-trivial absorbing subuniverse, non-trivial central subuniverse, or non-trivial PC subuniverse. 
\end{lemmach}

\begin{lemmach}[Lemma 7.24.1, \cite{zhuk2020proof}]\label{''';dlkykugkdegfef}
Suppose that $\mathscr{R}_{\Theta}\leq D_0\times...\times D_{n-1}$ is a relation such that $pr_0(\mathscr{R}_{\Theta})=D_0$, there are no non-trivial binary absorbing subuniverses on $D_0$, and $L = pr_0((L_1\times...\times L_{n-1})\cap \mathscr{R}_{\Theta})$ where $L_i$ is a linear subuniverse of $D_i$ for every $i<n$. Then $W_1^1$ proves that $L$ is a linear subuniverse of $D_0$. 
\end{lemmach}

The following lemma can be proved similarly to Lemma \ref{lalaksjdf}.
\begin{lemmach}[Lemma 7.25, \cite{zhuk2020proof}]\label{alksjdlkjlahskdhlk}
  Suppose $\mathscr{R}_{\Theta}$ is a subdirect relation on $D_0\times...\times D_{n-1}$ and $L_i$ is a linear subuniverse of $D_i$ for every $i$. Then $W^1_1$ proves that $(L_0\times...\times L_{n-1})\cap \mathscr{R}_{\Theta}$ is a linear subuniverse of $\mathscr{R}_{\Theta}$.
\end{lemmach}

\subsubsection{Common properties and Interaction between subuniverses}

The common property for subuniverses $C_0,...,C_{n-1}$ of a fixed type (any but linear) is that there does not exist $(C_0,...,C_{n-1})$-essential relation $\mathscr{R}$ of any arity greater than $2$. For PC subuniverses we additionally require the relation to be subdirect. 

\begin{lemmach}[Lemma 7.27,   \cite{zhuk2020proof}]
  Suppose $C_i$ is a non-trivial binary absorbing subuniverse of $D_i$ with a term $T$ for all $i\in \{0,1,2,...,n\}$ and $n>1$. Then $V^0$ proves that there does not exist a $(C_0,...,C_{n-1})$-essential solution set $\mathscr{R}_{\Theta}\leq D_0\times...\times D_{n-1}$.
\end{lemmach}
\begin{proof}
  Suppose that such solution set $\mathscr{R}_{\Theta}$ to some CSP instance over $\Gamma_{\mathcal{A}}$ exists. Consider two solutions, $H_1\in (D_0\times C_1\times...\times C_{n-1})\cap \mathscr{R}_{\Theta}$ and $H_2\in (C_0\times C_1\times...\times D_{n-1})\cap \mathscr{R}_{\Theta}$. Then $usepol_{2}(T,H_1,H_2)$ is a new solution to $\Theta$ and it is in $C_0\times...\times C_{n-1}$.
\end{proof}
\begin{lemmach}[Lemma 6.11, \cite{zhuk2020strong}]
  Suppose $C_i$ is a central subuniverse of $D_i$ for all $i\in \{0,1,2,...,n\}$ and $n>2$. Then $V^1$ proves that there does not exist a $(C_0,...,C_{n-1})$-essential solution set $\mathscr{R}_{\Theta}\leq D_0\times...\times D_{n-1}$.
\end{lemmach}

\begin{lemmach}[Corollary 7.13.3, \cite{zhuk2020proof}]
  Suppose $C_i$ is a PC subuniverse of $D_i$ for all $i\in \{0,1,2,...,n\}$ and $n>2$. Then $V^1$ proves that there does not exist a $(C_0,...,C_{n-1})$-essential subdirect solution set $\mathscr{R}_{\Theta}\leq D_0\times...\times D_{n-1}$.
\end{lemmach}
For our purposes, the last two lemmas can be proved by an exhaustive search. For any subuniverses $D_0,D_1,D_2$ of the fixed algebra $\mathbb{A}=(A,\Omega)$, for any of their central /PC subuniverses $C_0,C_1,C_2$ and for any subalgebras $R$ of $D_0\times D_1\times D_2$ check that $R$ is not $(C_0,C_1,C_2)$-essential or $(C_0,C_1,C_2)$-essential and subdirect. For a relation $\mathscr{R}\leq D_0\times...\times D_{n-1}$ of higher arity note that from $(C_0,...,C_{n-1})$-essential relation we can get $(C_0,C_1,C_{2})$-essential relation $\mathscr{R}'$ by setting
$$\mathscr{R}'=pr_{1,2,3}(\mathscr{R}\cap D_0\times D_1\times D_2\times C_3\times... \times C_{n-1}).
$$
In the case of the solution set $\mathscr{R}_{\Theta}$ to the instance $\Theta$, we consider those homomorphisms that send variables $x_3,...,x_{n-1}$ to sets $C_3,...,C_{n-1}$. The definition of $\mathscr{R}'_{\Theta}$ requires $\Sigma^{1,b}_1$-COMP, this gives us theory $V^1$. 

The following three lemmas about the interaction of subuniverses of different types are formulated for an arbitrary algebra $D$, i.e. for example, they can be used for $D=D_0\times...\times D_{n-1}$ and its subuniverses $\mathscr{R}_1$ and $\mathscr{R}_2$. For these cases, we can think about $D$ as of a domain set, and about $B_1$ and $B_2$ as of reductions $D^{(i)}_0\times...\times D^{(i)}_{n-1}$ or solution sets $\mathscr{R}_{\Theta}$. Sometimes we consider $D=\mathscr{R}_{\Theta}\leq D_0\times...\times D_{n-1}$ and $B_1=\mathscr{R}_{\Theta}\cap D^{(\bot)}_0\times...\times D^{(\bot)}_{n-1}$, $B_2=\mathscr{R}_{\Theta}\cap D^{(\top)}_0\times...\times D^{(\top)}_{n-1}$, where $D^{(\bot)}, D^{(\top)}$ are reductions of some (different) types. The proof of these lemmas is based on simple universal algebra reasoning, and in the presence of all third-order objects, their formalization in $W^1_1$ does not differ much from \cite{zhuk2020proof}, \cite{zhuk2020strong}. 

\begin{lemmach}[Lemma $7.28$,   \cite{zhuk2020proof}, Lemma $6.25$]\label{a'a'65777yutd}
  Suppose $B_1$ is a binary absorbing, central, or linear subuniverse of $D$, $B_2$ is a subuniverse of $D$. Then $W^1_1$ proves that $B_1\cap B_2$ is a binary absorbing, central, or linear subuniverse of $B_2$, respectively. 
\end{lemmach}

\begin{lemmach}[Lemma 7.29,   \cite{zhuk2020proof}]
Suppose $B_1$ and $B_2$ are non-empty one-of-four subuniverses of $D$, $B_1\cap B_2=\emptyset$.
Then $W^1_1$ proves that $B_1$ and $B_2$ are subuniverses of the same type.
\end{lemmach}

\begin{lemmach}[Theorem 7.30,   \cite{zhuk2020proof}]\label{alskdjhfyr654}
  Suppose $B_1$ and $B_2$ are one-of-four subuniverses of $D$ of types $\mathcal{T}_1$ and $\mathcal{T}_2$. Then $W^1_1$ proves that $B_1\cap B_2$ is a one-of-four subuniverse of $B_2$ of type $\mathcal{T}_1$.
\end{lemmach}

The following lemma is proved by induction and is used for third-order objects, namely for reductions in strategies.

\begin{lemmach}[Lemma 7.31,   \cite{zhuk2020proof}]
Suppose $\mathscr{A}_0=\mathscr{B}_0$, $s\geq 1$, $t\geq 0$, $\mathscr{A}_{i}$ is a one-of-four subuniverse of $\mathscr{A}_{i-1}$ for every $i\in\{1,...,s\}$, and $\mathscr{B}_{i}$ is a one-of-four subuniverse of $\mathscr{B}_{i-1}$ for every $i\in\{1,...,t\}$. Then $W_1^1$ proves that $\mathscr{A}_{s}\cap \mathscr{B}_t$ is a one-of-four subuniverse of $\mathscr{A}_{s-1}\cap \mathscr{B}_t$ of the same type as $\mathscr{A}_{s}$.
\end{lemmach}
\begin{proof}
The proof of the claim goes by induction on $s+t$. If $t=0$, then the claim follows from the statement. If $t\leq 1$, then by the inductive assumption, $\mathscr{A}_{s-1}\cap \mathscr{B}_{t}$ and $\mathscr{A}_{s}\cap \mathscr{B}_{t-1}$ are both one-of-four subuniverses of $\mathscr{A}_{s-1}\cap \mathscr{B}_{t-1}$, and the second one is of type $\mathcal{T}$. Then by Theorem \ref{alskdjhfyr654} their intersection $\mathscr{A}_{s}\cap \mathscr{B}_{t}$ is a one-of-four subuniverses of $\mathscr{A}_{s-1}\cap \mathscr{B}_{t}$ of type $\mathcal{T}$. 

We will formalize the proof for the specific case that we further need in the proofs of auxiliary lemmas about strategies. Suppose that $\mathscr{A}_0=\mathscr{B}_0=\mathscr{R}_{\Theta}\leq D_0\times...\times D_{n-1}$ for some CSP instance $\Theta$ with domain set $D$, where $\mathscr{R}_{\Theta}$ is its solution set, and for each $i\in\{1,...,s\}$, 
$$\mathscr{A}_{i} = \mathscr{R}_{\Theta}\cap D^{(i)}_0\times...\times D^{(i)}_{n-1}
$$
where $D=D^{(0)}, D^{(1)},...,D^{(s)}$ is some strategy for $\Theta$, and analogously, for each $i\in\{1,...,t\}$, 
$$\mathscr{B}_{i} = \mathscr{R}_{\Theta}\cap D^{(i)'}_0\times...\times D^{(i)'}_{n-1}
$$
for some (other) strategy $D=D^{(0)'}, D^{(1)'},...,D^{(t)'}$ for $\Theta$. Recall that we can formalize any strategy by one set $\Theta_{Str}< \langle nl,instsize(n,l)\rangle$. By Corollaries \ref{==-=-=-=-=-=-=-=-}, \ref{66666666666666} and Lemmas \ref{lalaksjdf}, \ref{alksjdlkjlahskdhlk} we know that $W_1^1$ proves that $\mathscr{A}_{i}$ is a one-of-four subuniverse of $\mathscr{A}_{i-1}$ and $\mathscr{B}_{i}$ is a one-of-four subuniverse of $\mathscr{B}_{i-1}$. Since in any step we reduce at least one domain, the number of steps $t,s$ cannot be greater than $nl$ and $t+s< 2nl$. This induction is available in $W_1^1$: the formula itself is $\Sigma^{1,b}_0$, but in the proof of the induction step we use the results proved in $W^1_1$.
\end{proof}

\begin{lemmach}[Lemma 7.32,   \cite{zhuk2020proof}]\label{a;lsdkkjef}
Suppose $\mathscr{R}_{\Theta}\subseteq \mathscr{A}_0\times \mathscr{B}_0$ is a subdirect relation, $\mathscr{B}_i$ is a one-of-four subuniverse of $\mathscr{B}_{i-1}$ for every $i\in\{1,2,...,s\}$, $\mathscr{A}_1$ is a one-of-four subuniverse of $\mathscr{A}_0$. Then $W_1^1$ proves that $pr_1(\mathscr{R}_{\Theta}\cap (\mathscr{A}_1\times \mathscr{B}_s))$ is a one-of-four subuniverse of $pr_1(\mathscr{R}_{\Theta}\cap (\mathscr{A}_1\times \mathscr{B}_{s-1}))$ of the same type as $\mathscr{B}_s$.
\end{lemmach}

\begin{proof}
  The statement of this lemma will eventually be used in the proof of Lemma 8.28 \cite{zhuk2020proof}, which is used further for constraints and subconstraints in proofs of Theorem \ref{---098} and Theorem \ref{########6}. So we will formalize the proof of the lemma for one specific case of Lemma $8.28$ \cite{zhuk2020proof}. Consider some subdirect solution set $\mathscr{R}_{\Theta}\leq D_0\times...\times D_{n-1}$ (for constraints and projections the reasoning is similar). Let $D=D^{(0)}, D^{(1)},...,D^{(s)}$ be some strategy for $\Theta$, and set for $i=0,1$
$$\mathscr{A}_i = pr_{0,1,...t-1}(D^{(i)}_0\times...\times D^{(i)}_{t-1}\times D_t\times...\times D_{n-1}),$$
and for $i=0,...,t$ 
$$\mathscr{B}_i=pr_{t,t+1,...n-1}(D_0\times...\times D_{t-1}\times D^{(i)}_t\times...\times D^{(i)}_{n-1}).$$
Then by Corollaries \ref{==-=-=-=-=-=-=-=-}, \ref{66666666666666} and Lemmas \ref{lalaksjdf}, \ref{alksjdlkjlahskdhlk}, $W_1^1$ proves that $\mathscr{R}_{\Theta}\cap(\mathscr{A}_0\times \mathscr{B}_i)$ is a one-of-four subuniverse of $\mathscr{R}_{\Theta}\cap(\mathscr{A}_0\times \mathscr{B}_{i-1})$ of the same type as $\mathscr{B}_i$, and $\mathscr{R}_{\Theta}\cap(\mathscr{A}_1\times \mathscr{B}_{0})$ is a one-of-four subuniverse of $\mathscr{R}_{\Theta}$. By Lemma \ref{a;lsdkkjef}, $\mathscr{R}_{\Theta}\cap(\mathscr{A}_1\times \mathscr{B}_{s})$ is a one-of-four subuniverse of $\mathscr{R}_{\Theta}\cap(\mathscr{A}_1\times \mathscr{B}_{s-1})$ of the same type as $\mathscr{B}_{s}$. Consider a congruence $\mathscr{C}_{\sigma}$ on $\mathscr{R}_{\Theta}\cap(\mathscr{A}_1\times \mathscr{B}_{0})$ such that two elements are equivalent whenever their projections on the second coordinate are equal, 
$$\mathscr{C}_{\sigma}(H_1,H_2)\iff \forall t\leq i<n,\, H_1(i)=H_2(i).
$$
Note that this is $\Sigma^{1,b}_0$ definition. Then for every coordinate $i=0,1$, $Stable_1(\mathscr{R}_{\Theta}\cap(\mathscr{A}_1\times \mathscr{B}_{t}),\mathscr{C}_{\sigma})$, which means that if $\mathscr{R}_{\Theta}\cap(\mathscr{A}_1\times \mathscr{B}_{t})$ contains one element of the block of $\mathscr{C}_{\sigma}$, then it contains the entire block. We now need Lemma $7.26$ from \cite{zhuk2020proof}, which is used just once in this proof. Therefore, we will formalize it only for this specific case as a claim.

\begin{claimm}
  Suppose $\mathscr{C}_{\sigma}$ is a congruence on $\mathscr{R}_{\Theta}\cap(\mathscr{A}_1\times \mathscr{B}_{0})$, $\mathscr{R}_{\Theta}\cap(\mathscr{A}_1\times \mathscr{B}_{t})$ is a one-of-four subuniverse of $\mathscr{R}_{\Theta}\cap(\mathscr{A}_1\times \mathscr{B}_{0})$ stable under $\mathscr{C}_{\sigma}$. Then 
 $W^1_1$ proves that $\{H/\mathscr{C}_{\sigma}|H\in \mathscr{R}_{\Theta}\cap(\mathscr{A}_1\times \mathscr{B}_{t})\}$ is a one-of-four subuniverse of $\mathscr{R}_{\Theta}\cap(\mathscr{A}_1\times \mathscr{B}_{0})/\mathscr{C}_{\sigma}$ of the same type as $\mathscr{R}_{\Theta}\cap(\mathscr{A}_1\times \mathscr{B}_{t})$.
\end{claimm}
The proof of the claim is as in \cite{zhuk2020proof}. The statement follows immediately from the claim. 
\end{proof}

\subsubsection{Some technical lemmas}
In the following two lemmas, $\Theta(z)$ is the set of all $a\in D_z$ such that there is a solution to $\Theta$ with $z=a$. Analogously, $\Theta^{(1)}(z)$ is the set of all $a\in D^{(1)}_z$ such that there is a solution to $\Theta^{(1)}$ with $z=a$.

\begin{lemmach}[Lemma 8.1,   \cite{zhuk2020proof}]\label{a''a;sdlekrygh}
  Suppose $D^{(1)}$ is a one-of-four reduction for an instance $\Theta$ of type $\mathcal{T}$, which is not of the PC type. Then $W^1_1$ proves that $\Theta^{(1)}(z)$ is a one-of-four subuniverse of $\Theta(z)$ of type $\mathcal{T}$ for every variable $z$.
\end{lemmach}
\begin{proof}
  Consider a CSP instance $\Theta$ on $n$ domains and its solution set $\mathscr{R}_{\Theta}$. Since $\mathscr{R}_{\Theta}$ is preserved by $\omega$, $\Theta(i)$ is a subuniverse of $D_i$ for every $i$, and by the definition of reductions, $D^{(1)}_i$ is a subuniverse of type $\mathcal{T}$. Thus, by Lemma \ref{a'a'65777yutd}, $\Theta(i)\cap D^{(1)}_i$ is a subuniverse of $\Theta(i)$ of type $\mathcal{T}$ for every $i<n$. Consider the reduction $\Theta'$ of $\Theta$ to the domain set $[\Theta(0),...,\Theta(n-1)]$, $\mathscr{R}_{\Theta'}$ is a subdirect relation. Then, by Corollaries \ref{akjshdutjgjf}, \ref{akjshdgfuyfgwef} and Lemma \ref{''';dlkykugkdegfef}, $\Theta^{(1)}(z)$ is a one-of-four subuniverse of $\Theta(z)$ of type $\mathcal{T}$.
\end{proof}

\begin{lemmach}[Lemma 8.2,   \cite{zhuk2020proof}]
  Suppose $D^{(1)}$ is a PC reduction for a $1$-consistent instance $\Theta$, for every variable $y$ appearing at least twice in $\Theta$ the $pp$-formula $\Theta(y)$ defines $D_y$ and $\Theta(z)$ defines $D_z$ for a variable $z$. Then $V^1$ proves that $\Theta^{(1)}(z)$ is a PC subuniverse of $D_z$.
\end{lemmach}
\begin{proof}
  For the proof, we first rename all variables in $\Theta$ so that every variable occurs just once. This instance is denoted as $\Theta_0$. Then, step by step, we identify each two variables back to obtain the original instance, by the sequence $\Theta_0,\Theta_1,...\Theta_s=\Theta$. We show that these transformations can be held in $V^1$.

Recall that we are allowed to have only one constraint relation for any two variables $x,y$ (in that order). That is, for the instance $\Theta$ with $n$ variables, the number of possible constraints that involve one variable is at most $(2n-1)$ (and the number of all possible constraints is at most $n^2$). First, define the set of all variables that occur in $\Theta$ more than once:
\begin{equation}
  \begin{gathered}
     \forall x<n,\, S(x)\iff \exists y\neq z<n, \, (E_{\mathcal{X}}(x,y)\vee E_{\mathcal{X}}(y,x))\wedge (E_{\mathcal{X}}(x,z)\vee E_{\mathcal{X}}(z,x)).
  \end{gathered}
\end{equation}
Note that if we have two edges of the form $E_{\mathcal{X}}(x,y), E_{\mathcal{X}}(y,x)$, we need to rename one $x$ to $x'$ and one $y$ to $y'$, and do it at different steps. To perform this, we further define two sets of variables for any such $x$, $S^{in}, S^{out}$: 
\begin{equation}
  \begin{gathered}
     \forall x,y<n,\, S^{out}(x,y)\iff S(x)\wedge E_{\mathcal{X}}(x,y),\\
     \forall x,z<n,\, S^{in}(x,z)\iff S(x)\wedge E_{\mathcal{X}}(z,x).
  \end{gathered}
\end{equation}
When we rename every occurrence of a variable $x$, we can get at most $2n$ new variables (there are at most $(2n-1)$ constraints with $x$, and one of them could be a loop $E_{\mathcal{X}}(x,x)$). It works for every of $n$ variables, so the maximal number of steps is $2n^2$. We now set $s=2n^2$, $\Theta_s=\Theta$, $S_s=S, S^{out}_{s}=S^{out}, S^{in}_{s}=S^{in}$, and then for any $t=1,...,2n^2$ we will define a new CSP instance $\Theta_{s-t}$ based on the following rules. If the set $S_{s-(t-1)}$ is empty (neither of the variables occurs at least twice) we just replicate the instance $\Theta_{s-(t-1)}$. Otherwise, for odd $t$, we consider the set $S^{in}_{(s-(t-1))}$ (and for even the set $S^{out}_{(s-(t-1))}$). If it is empty, replicate the instance and move on. If not, choose elements $x,y$ such that $\langle x,y\rangle$ is the minimum element of $S^{in}_{(s-(t-1))}(x,y)$. Rename a variable $x$ in that constraint by the next number after the maximum element in $V_{\mathcal{X}_{s-(t-1)}}$. For this construction, 
for each odd step $t$ we consider additional sets $L_{(s-(t-1))}$ and $R_{(s-(t-1))}$, defined as follows: 
  \begin{equation}
    \begin{gathered}
       L_{(s-(t-1))}(x)\iff \exists y<max(V_{\mathcal{X}_{s-(t-1)}})+1,\, min(S^{in}_{(s-(t-1))})=\langle x,y\rangle,\\
       R_{(s-(t-1))}(y)\iff \exists x<max(V_{\mathcal{X}_{s-(t-1)}})+1,\, min(S^{in}_{(s-(t-1))})=\langle x,y\rangle.  
    \end{gathered}
  \end{equation}
Now we are ready to define an instance digraph for step $t$:
\begin{equation}
  \begin{gathered}
     \forall x<max(V_{\mathcal{X}_{s-(t-1)}})+1,\,V_{\mathcal{X}_{s-t}}(x)\iff V_{\mathcal{X}_{s-(t-1)}}(x),\\
     V_{\mathcal{X}_{s-t}}(max(V_{\mathcal{X}_{s-(t-1)}})+1)\\
     \forall x,y<max(V_{\mathcal{X}_{s-(t-1)}})+1,\,E_{\mathcal{X}_{s-t}}(x,y)\iff E_{\mathcal{X}_{s-(t-1)}}(y,x)\wedge \\
     \wedge (\neg L_{(s-(t-1))}(x)\vee \neg R_{(s-(t-1))}(y))\\
     E_{\mathcal{X}_{s-t}}(y,max(V_{\mathcal{X}_{s-(t-1)}})+1)\iff L_{(s-(t-1))}(x)\wedge R_{(s-(t-1))}(y).
   \end{gathered}
\end{equation}
In parallel, we define a target digraph $\ddot{\mathcal{A}}_{s-t}$ by adding there a new domain $D_{max(V_{\mathcal{X}_{s-(t-1)}})}$ $_{+1}$ equal to $ D_x$ for a new variable $max(V_{\mathcal{X}_{s-(t-1)}})+1$ and $E^{yx}_{\ddot{\mathcal{A}}}$ as a constraint for a new edge $E_{\mathcal{X}_{s-t}}(y,max(V_{\mathcal{X}_{s-(t-1)}})+1)$. Eventually, we will get an instance $\Theta_0$. Since we consider all sets in a particular order and address only sets from the previous step $t-1$, all of them exist by $\Sigma^{1,b}_1$ induction. 

The proof of the statement then goes by induction on $s$, and the implication $s\rightarrow s+1$ follows from the reasoning that can be easily formalized in $V^1$. We refer the reader to the source \cite{zhuk2020proof}. 
\end{proof}

For a relation $\mathscr{R}$ of arity $n$ denote by $UnPol^{\mathscr{R}}$ the set of all unary vector functions preserving the relation $\mathscr{R}$. For a solution set $\mathscr{R}_{\Theta}$ for some CSP instance $\Theta$, due to (\ref{laalks34523})
\begin{equation}
    \Psi\in UnPol^{\mathscr{R}_{\Theta}}\iff VecFun(\mathscr{R}_{\Theta},\Psi),
\end{equation}
which is a $\Pi^{1,b}_1$-formula. For every map $H$ from $[n]$ to $[D_0,...,D_{n-1}]$, and every unary vector function $\Psi$, we can define a map $\Psi(H)$ using bit-definition:
\begin{equation}
  \begin{gathered}
     \Psi(H)(\langle i,\langle i,a\rangle\rangle)=H^{\Psi}(\langle i,\langle i,a\rangle\rangle)\iff \exists b\in D_i,\, H(\langle i,\langle i,b\rangle\rangle)\wedge \Psi(i,b,a).
  \end{gathered}
\end{equation}
\begin{lemmach}[Lemma 8.12,   \cite{zhuk2020proof}]\label{()()()()(  *^ }
  Suppose a $pp$-formula $\Lambda(x_0,...,x_{n-1})$ defines a relation $\mathscr{R}_{\Lambda}$, $H\in D_{x_0}\times...\times D_{x_{n-1}}$, and $\mathscr{R}'=\{H^{\Psi}:\Psi\in UnPol^{\mathscr{R}_{\Lambda}}\}$. Then $W^1_1$ proves that there exists $\Upsilon\in Covering(\Lambda)$ such that $\Upsilon(x_0,...,x_{n-1})$ defines $\mathscr{R}'$.
\end{lemmach}
\begin{proof}
The idea of the universal algebra proof is the following. Consider any relation $\mathscr{R}$ on $n$ variables. Suppose that there are $l$ elements in each domain, $d_0,...,d_{l-1}$. Then the formula 
$$\mathscr{S}(x^{d_0}_0,...,x^{d_{l-1}}_0,...,x^{d_0}_{n-1},...,x^{d_{l-1}}_{n-1}) = \bigwedge_{(b_0,...,b_{n-1})\in \mathscr{R}}\mathscr{R}(x^{b_0}_0,...,x^{b_{n-1}}_{n-1})
$$
expresses that the vector-function preserves $\mathscr{R}$ (we think about $x^{b_i}_i$ as about $x_i$ being sending to $b_i$). Then, if we consider any tuple $\alpha = (a_0,...,a_{n-1})$, the projection of $\mathscr{S}$ to $x^{a_0}_0,...,x^{a_{n-1}}_{n-1}$ defines the relation $\{f(\alpha):f\in UnPol^{\mathscr{R}}\}$.

  We will consider $\Lambda$ as a CSP instance on $n$ variables, $|V_{\mathcal{X}_{\Lambda}}|=n$ (for projections the reasoning is analogous). Suppose that for some $a_0,...,a_{n-1}$, for all $i<n$, $H(i)=\langle i,a_i\rangle$. We need to define a new CSP instance $\Upsilon$ such that the projection of its solution set to some subset of vertices is exactly $\mathscr{R}'$. Consider a CSP instance $\Upsilon_{null}$ on $nl$ variables, where for $i<n,a<l$ we think about vertex $il+a$ as about vertex $i$ that was sent to $\langle i,a\rangle$ (or if we use labels, $x_i\to a\in D_i$). Then for every $H'\in \mathscr{R}_{\Lambda}$ such that for $i<n,b_i<l,$ $H'(i)=\langle i,b_i\rangle$ we copy instance $\Lambda$ to domains $D_{b_0},D_{l+b_1},D_{2l+b_2},...,D_{(n-1)l+b_{n-1}}$. Denote the resulting instance by $\Upsilon$. It is clear that $\Upsilon\in Covering(\Lambda)$. Then the projection $\mathscr{R}_{\Upsilon}^{a_0,l+a_1,...,(n-1)l+a_{n-1}}$ defines $\mathscr{R}'$. 

The algorithm of the construction is clear, but to perform it we need the number of steps that is bounded only by $l^n$ (the number of possible homomorphisms from $[n]$ to $[D_0,...,D_{n-1}]$). Since every homomorphism $H$ is expressed by a string of length $\langle n,\langle n,l\rangle \rangle<n^4$, we encode the number of steps by strings $\emptyset<T<n^4$ of length up to $n^4$, run the algorithm (if $T$ represents some homomorphism to the instance $\Lambda$, copy $\Lambda$ to corresponding domains), and then use $\Sigma^{\mathscr{B}}_1$-induction to show that such instance exists. 
\end{proof}

\begin{corollary}[Corollary 8.12.1,   \cite{zhuk2020proof}]\label{9899889((())))}
  Suppose a $pp$-formula $\Lambda(x_0,...,x_{n-1})$ defines a relation $\mathscr{R}_{\Lambda}$ without a tuple $H\in D_{x_0}\times...\times D_{x_{n-1}}$, $\Sigma$ is the set of all relations defined by $\Upsilon(x_0,...,x_{n-1})$ where $\Upsilon\in Covering(\Lambda)$, and $\mathscr{R}_{\Lambda}$ is an inclusion-maximal relation in $\Sigma$ without the tuple $H$. Then $W^1_1$ proves that $H$ is a key tuple for $\mathscr{R}_{\Lambda}$. 
\end{corollary}

\begin{proof}
Consider $\Lambda$ as a CSP instance on $n$ variables, let $S$ be any map from $[n]$ to $[D_0,...,D_{n-1}]$ that is not in $\mathscr{R}_{\Lambda}$. Then by Lemma \ref{()()()()(  *^ } the set of maps $\mathscr{R}'=\{S^{\Psi}:\Psi\in UnPol^{\mathscr{R}_{\Lambda}}\}$ is a projection of the solution set to some $\Upsilon\in Covering(\Lambda)$. Since $\Psi$ can be constant mapping to a homomorphism of $\Lambda$ and identity mapping, $\mathscr{R}_{\Lambda}\subsetneq\mathscr{R}'$, and since $\mathscr{R}_{\Lambda}$ is inclusion-maximal, $H\in \mathscr{R}'$. By the definition, $H$ is a key tuple for $\mathscr{R}_{\Lambda}$.
\end{proof}

Lemmas we consider next in this section are
\begin{enumerate}
\item either related exclusively to binary relations since we consider languages with at most binary constraint relations,
\item or are used in the further proofs only for constant arity relations,
\item or related to constant arity relations and constant sizes over algebra $\mathbb{A}$.
\end{enumerate}
In the first and the second cases, they can be formalized and proved in $V^1$ exactly as they are proved in \cite{zhuk2020proof}. In the third case, such properties must be listed in the $\mathbb{A}$-Monster set. For these reasons, we will mention a few examples, but for the proofs and the rest, we refer the reader to the source \cite{zhuk2020proof}.

\begin{lemmach}[Lemma 7.19, \cite{zhuk2020proof}]
  Suppose $R\subseteq D\times B\times B$ is a subdirect relation, $D$ is a PC algebra without a non-trivial binary absorbing or central subuniverse, and for every $b\in B$ there exists $a\in A$ such that $(a,b,b)\in R$. Then $V^1$ proves that for every $a\in A$ there exists $b\in B$ such that $(a,b,b)\in R$.
\end{lemmach}

The following two lemmas are formulated in \cite{zhuk2020proof} for $t$ variables, but then they are only used for relations on one and two variables, so instead of using induction on $t$ we can consider just cases $\Theta(x_0)$ and $\Theta(x_0,x_1)$.
\begin{lemmach}[Lemma 8.3,   \cite{zhuk2020proof}]
  Suppose $D^{(1)}$ is a minimal absorbing, central or linear reduction for an instance $\Theta$, and $\Theta(x_0,x_1)$ defines a full relation. Then $V^1$ proves that $\Theta^{(1)}(x_0,x_1)$ defines a full or empty relation. 
\end{lemmach}

\begin{lemmach}[Lemma 8.4,   \cite{zhuk2020proof}]
  Suppose $D^{(1)}$ is a minimal PC reduction for a $1$-consistent instance $\Theta$. For every variable $y$ appearing at least twice in $\Theta$ the $pp$-formula $\Theta(y)$ defines $D_y$ and $\Theta(x_0,x_1)$ defines a full relation. Then $V^1$ proves that $\Theta^{(1)}(x_0,x_1)$ defines a full or empty relation
\end{lemmach}

The next examples of technical lemmas are the following. 

\begin{lemmach}[Lemma 8.10,   \cite{zhuk2020proof}]
  Suppose $R\leq D_i\times D_j$ is a critical rectangular binary relation, and $R'$ is a cover of $R$. Then $V^1$ proves that $Con_2^{(R,i)}\subsetneq Con_2^{(R,i)}$.
\end{lemmach}

\begin{lemmach}[Theorem 8.15,   \cite{zhuk2020proof}]
  Suppose $R\leq D^4$ is a strongly rich relation preserved by an idempotent
WNU. Then $V^1$ proves that there exists an abelian group $(D,+)$ and bijective mappings $\phi_0,\phi_1,\phi_2,\phi_3:D\to D$ such that 
$$R=\{(x_0,x_1,x_2,x_3): \phi_0(x_0)+ \phi_1(x_1)+\phi_2(x_2)+\phi_3(x_3)=0\}.
$$
\end{lemmach}

\begin{lemmach}[Theorem 8.17,   \cite{zhuk2020proof}]
  Suppose $\sigma\subseteq D^2$ is a congruence, $\rho$ is a bridge from $\sigma$ to $\sigma$ such that $\tilde{\rho}$ is a full relation, $pr_{1,2}(\rho)=\omega$, $\omega$ is a minimal relation stable under $\sigma$ such that $\sigma\subsetneq \omega$. Then $V^1$ proves that there exists a prime number $p$ and a relation $\zeta\subseteq D\times D\times \mathbb{Z}_p$ such that $pr_{1,2}\zeta = \omega$ and $(a_1,a_2,b)\in \zeta$ implies that $(a_1,a_2)\in \sigma\iff (b=0).$
\end{lemmach}

\begin{lemmach}[Lemma 8.18,   \cite{zhuk2020proof}]
  Suppose that $\rho\subseteq D^4$ is an optimal bridge from $\sigma_1$ to $\sigma_2$, and $\sigma_1$ and $\sigma_2$ are different irreducible congruences. Then $V^1$ proves that $\sigma_2\subseteq \tilde{\rho}$.
\end{lemmach}

\begin{lemmach}[Lemma 8.20,   \cite{zhuk2020proof}]
  Suppose $R\leq D_i\times D_j$ is a subdirect rectangular relation and there exists $(b_i,a_j),(a_i,b_j)\in R$ such that $(a_i,a_j)\notin R$. Then $V^1$ proves that there exists a bridge $\delta$ from $Con_2^{(R,i)}$ to $Con_2^{(R,j)}$ such that $\tilde{\delta} = R$.
\end{lemmach}

These and some other lemmas imply the following result about CSP instances.  

\begin{lemmach}[Lemma 8.22,   \cite{zhuk2020proof}]\label{'a'a;slddkfkjg}
  Suppose $\Theta$ is a cycle-consistent connected instance. Then $V^1$ proves that for any constraints $C,C'$ with variables $x,x'$ there exists a bridge $\delta$ from $Con^{(C,x)}$ to $Con^{(C',x')}$ such
that $\tilde{\delta}$ contains all pairs of elements linked in $\Theta$. Moreover, if $Con^{(C'',x'')}\neq Linked_{[x'',x'',\Theta]}$
for some constraint $C''\in\Theta$ and a variable $x''$, then $\delta$ can be chosen so that $\tilde{\delta}$ contains all pairs
of elements linked in $\Theta'$, where $\Theta'$ is obtained from $\Theta$ by replacing every constraint relation
by its cover. 
\end{lemmach}
The next lemma is easily proved by the application of the definition of a crucial instance and expanded covering. If we replace any constraint in a crucial instance $\Theta$ with all weaker constraints, we get a solution. All relations in the expanded covering $\Theta'$ are diagonal relations or weaker or equivalent to the relations in $\Theta$. 

\begin{lemmach}[Lemma 8.24,   \cite{zhuk2020proof}]\label{alksjdhf}
Suppose $\Theta_{\mathcal{X}}=(\mathcal{X},\ddot{\mathcal{A}})$ is a crucial instance in $D^{(1)}$, $\Theta_{\mathcal{Y}}=(\mathcal{Y},\ddot{\mathcal{B}})\in ExpCov(\Theta_{\mathcal{X}})$ through the homomorphism $H$ from $\mathcal{Y}$ to $\mathcal{X}$, and $\Theta_{\mathcal{Y}}$ has no solution in $D^{(1)}$. Then $V^1$ proves that for every constraint $E^{x_ix_j}_{\ddot{\mathcal{A}}}$ in $\Theta_{\mathcal{X}}$ there exists a constraint $E^{y_ky_p}_{\ddot{\mathcal{B}}}$ such that $H(y_k)=x_i$, $H(y_p)=x_j$ and $E^{x_ix_j}_{\ddot{\mathcal{A}}} = E^{y_ky_p}_{\ddot{\mathcal{B}}}$.
\end{lemmach}

Lemmas \ref{alksjdhf} and \ref{'a'a;slddkfkjg} imply Lemma \ref{;asldkljf}. 
\begin{lemmach}[Lemma 8.25,   \cite{zhuk2020proof}]\label{;asldkljf}
  Suppose $\Theta_{\mathcal{X}}=(\mathcal{X},\ddot{\mathcal{A}})$ is a crucial instance in $D^{(1)}$, $\Theta_{\mathcal{Y}}=(\mathcal{Y},\ddot{\mathcal{B}})\in ExpCov(\Theta_{\mathcal{X}})$ has no solution in $D^{(1)}$, every constraint relation of $\Theta_{\mathcal{X}}$ is a critical rectangular relation, and $\Theta_{\mathcal{Y}}$ is connected.
Then $V^1$ proves that $\Theta_{\mathcal{X}}$ is connected.
\end{lemmach}

\subsection{Formalization of the main theorems}

\subsubsection{The existence of the next reduction}

\begin{lemmach}[Lemma 9.1, \cite{zhuk2020proof}]\label{;allskdleijf5}
  Suppose $D^{(0)},D^{(1)},...,D^{(s)}$ is a strategy for a $1$-consistent CSP instance $\Theta$, and $D^{(\bot)}$ is a reduction of $\Theta^{(s)}$. Then $V^1$ proves that:
  \begin{enumerate}
    \item if there exists a $1$-consistent reduction contained in $D^{(\bot)}$ and $D^{(s+1)}$ is maximal among such reductions, then for every variable $x$ of $\Theta$ there exists a tree-formula $\Upsilon_x\in Coverings(\Theta)$ such that $\Upsilon^{(\bot)}_x(x)$ defines $D^{(s+1)}_x$;
    \item otherwise, there exists a tree-formula $\Upsilon\in Coverings(\Theta)$ such that $\Upsilon^{(\bot)}$ has no solutions.
  \end{enumerate}
\end{lemmach}
\begin{proof}
  The proof of this theorem is based on constraint propagation. At the beginning for every variable $x$ we consider an empty tree-formula $\Upsilon_x$. Then $\Upsilon^{(\bot)}_x$ defines the reduction $D^{(\bot)}$. Then the recursive algorithm works as follows: if at some step the reduction defined by these tree-formulas is $1$-consistent, it stops. Otherwise, it considers any constraint $C=R(x_1,...,x_t)$ that breaks $1$-consistency. The current restrictions of variables $x_1,...,x_t$ in $C$ imply stronger restriction of some variable $x_i$, and the algorithm changes the formula $\Upsilon_{x_i}$ as follows:
  $$\Upsilon_{x_i} =_{def} C\wedge \Upsilon_{x_1}\wedge...\wedge \Upsilon_{x_t}.$$
  To keep the formula $\Upsilon_{x_i}$ tree, any time the algorithm joins $\Upsilon_{x_j}$ and $\Upsilon_{x_k}$ it renames the variables so that they do not have common variables. Finally, for each $\Upsilon_{x_j}$ we consider the reduction of this instance on the domain set $D^{(\bot)}$. Projection of the solution set to $\Upsilon^{(\bot)}_{x_j}$ on variable $x_j$, $\Upsilon^{(\bot)}_{x_j}(x_j)$ defines $D^{(s+1)}_{x_j}$. That will be a maximal $1$-consistent reduction since it is defined by
tree-formulas.

  Let us formalize this algorithm. The formalization is based on a $1$-consistency algorithm (see \cite{barto_et_al}). Recall that any CSP instance can be converted in polynomial time to a $1$-consistent one with the same set of solutions. Moreover, any implementation of a $1$-consistency algorithm derives the same unary constraints. Thus, we can first define recursive sets of edges and vertices, based on which we can construct our tree-formulas. 

For the steps $t=0,t=1$ for any $i,j<n$ we set 
\begin{equation}\label{88djeu6453}
    \begin{gathered}
     E^{ij}_{\ddot{\mathcal{A}},0}(a,b) \iff E^{ij}_{\ddot{\mathcal{A}}}(a,b),\\
 E^{ij}_{\ddot{\mathcal{A}},1}(a,b) \iff a\in D_i^{(\bot)}\wedge b\in D_j^{(\bot)}\wedge E^{ij}_{\ddot{\mathcal{A}}}(a,b).
    \end{gathered}
  \end{equation}

For any further step $t>1$ we will propagate constraints recursively until we cannot change any domain further (i.e. until the instance is $1$-consistent) or some domain is empty. For any $i,j<n$ we set: 
\begin{equation}\label{dhhjsyu65d}
\begin{gathered}
 \hspace {0pt}E^{ij}_{\ddot{\mathcal{A}},t}(a,b) \iff \big[(\forall p,r<n,\exists e,f<l,\, E_{\mathcal{X}}(p,r)\rightarrow E^{pr}_{\ddot{\mathcal{A}},t-1}(e,f)) \wedge \\
 \hspace {0pt}\wedge E^{ij}_{\ddot{\mathcal{A}},t-1}(a,b)\wedge \forall q<n,\,E_{\mathcal{X}}(i,q) \rightarrow \exists d \in D^{(\bot)}_{q},\,E^{iq}_{\ddot{\mathcal{A}},t-1}(a,d)\wedge\\
 \hspace {0pt}\hspace {0pt}\wedge\forall k<n, \, E_{\mathcal{X}}(k,i) \rightarrow \exists c \in D^{(\bot)}_{k},\,\, E^{ki}_{\ddot{\mathcal{A}},t-1}(c,a)\wedge\\
 \hspace{0pt}\wedge\forall q<n,\,E_{\mathcal{X}}(j,q) \rightarrow \exists d \in D^{(\bot)}_{q},\,E^{jq}_{\ddot{\mathcal{A}},t-1}(b,d)\wedge\\
 \hspace {0pt}\hspace {0pt}\wedge\forall k<n, \, E_{\mathcal{X}}(k,j) \rightarrow \exists c \in D^{(\bot)}_{k},\,\, E^{kj}_{\ddot{\mathcal{A}},t-1}(c,b)\big] \vee\\
 \hspace {0pt}\vee \big[(\exists p,r<n,\forall e,f<l,\,E_{\mathcal{X}}(p,r)\wedge \neg E^{pr}_{\ddot{\mathcal{A}},t-1}(e,f)) \wedge E^{ij}_{\ddot{\mathcal{A}},t-1}(a,b)\big].
\end{gathered}
\end{equation}
The expression in the first square brackets holds when neither of the relations at the previous step is empty, and the expression in the second square brackets holds otherwise (it would mean that some domain of the instance at this step is already empty). In both cases after some step $t$ the relation set $E^{ij}_{\ddot{\mathcal{A}},t-1}$ stops changing. The maximal number of edges in a directed graph with loops on $n$ vertices is $n^2$. Therefore, the maximal number of edges in the instance $\Theta$ is $n^2l^2$ and since at each step we reduce some relation at least by one edge, it is enough to consider at most $n^2l^2$ steps. Moreover, since we remove an edge if at least one of its endpoints $a\in D^{(\bot)}_i$ violates $1$-consistency (so within one step we remove all edges in $E^{ij}_{\ddot{\mathcal{A}},t-1}$ connected with $a$ for all $j<n$), the actual number of steps is $nl$ (the number of elements). The existence of this set is ensured by $\Sigma^{1,b}_1$-induction: consider the formula 
\begin{equation}
  \begin{gathered}
     \phi(t)=\exists E_{\ddot{\mathcal{A}}}<\langle t, \langle\langle n,l \rangle, \langle n,l \rangle\rangle\rangle,\, \forall i,j<n,\forall a,b<l,\,E^{ij}_{\ddot{\mathcal{A}},1}(a,b)\leftrightarrow (\ref{88djeu6453})\wedge\\
     \hspace {0pt}\wedge \forall 1<p<t,\forall i,j<n,\forall a,b<l,\,E^{ij}_{\ddot{\mathcal{A}},p}(a,b)\leftrightarrow (\ref{dhhjsyu65d}).
  \end{gathered}
\end{equation}
Since (\ref{dhhjsyu65d}) is a $\Sigma^{1,b}_0$-formula, to provide the implication $\phi(t)\rightarrow \phi(t+1)$ we can use comprehension axiom $\Sigma^{1,b}_0$-COMP.

Note that in (\ref{dhhjsyu65d}) we do not need to track the domain's changes separately (they are all recorded in the relations $E^{ij}_{\ddot{\mathcal{A}},t}$). We will proceed with recursive propagation of the domain set $V_{\ddot{\mathcal{A}}}$ after this procedure based on the resulting relation set. For any $i<n$, for steps $t=0,t=1$ we set 
\begin{equation}
  \begin{gathered}
   V_{\ddot{\mathcal{A}},0}(i,a)\iff D_i(a), \\
     V_{\ddot{\mathcal{A}},1}(i,a)\iff D_i^{(\bot)}(a),
  \end{gathered}
\end{equation}
and for all $1<t<nl$
\begin{equation}\label{ssskiia8a6yu}
    \begin{gathered}
     V_{\ddot{\mathcal{A}},t}(i,a)\iff V_{\ddot{\mathcal{A}}^,t-1}(i,a) \wedge (\forall j<n,\,E_{\mathcal{X}}(i,j) \rightarrow (\exists b<l,\, V_{\ddot{\mathcal{A}}^,t-1}(j,b)\wedge\\
 \hspace {0pt}\wedge E^{ij}_{\ddot{\mathcal{A}},t}(a,b)))\wedge(\forall k<n, \, E_{\mathcal{X}}(k,i) \rightarrow (\exists c<l,\,V_{\ddot{\mathcal{A}},t-1}(k,c)\wedge \\
 \hspace {0pt}\wedge E^{ki}_{\ddot{\mathcal{A}},t}(c,a))).
    \end{gathered}
\end{equation}
Again, this set exists due to $\Sigma^{1,b}_1$-induction. Note that in (\ref{ssskiia8a6yu}) for the step $t<nl$ we use $E_{\ddot{\mathcal{A}},t-1}$ and not $E_{\ddot{\mathcal{A}},nl}$: we need recursive changing of the domains for further reconstruction of the tree-formulas $\Upsilon_i$. We also define a set $C^{list}_i$ that for any step $0<t<nl$ collect elements that were deleted from $V_{\ddot{\mathcal{A}},t,i}$: 
\begin{equation}
  \begin{gathered}
     \forall 0<t<nl,\forall a<l,\, C^{list}_i(t,a)\iff V_{\ddot{\mathcal{A}},t-1}(i,a)\wedge \neg V_{\ddot{\mathcal{A}},t}(i,a).
  \end{gathered}
\end{equation}

Further, we need to construct tree-instances $\Upsilon_i$. We want them to be coverings, so for each $i$ we only need to define an instance graph $\mathcal{X}_i$ and remember parents for renamed variables. To do it, we again use recursion. For each $i<n$, we start with an instance $\Upsilon_{i,0}$ with a domain set $D$ and with an empty set of constraints. Then for further steps $u$ we:
\begin{itemize}
  \item either do nothing with instance $\Upsilon_{i,u}$ - if for any $k<n$ and for some $t$ constraints $E^{ik}_{\ddot{\mathcal{A}},t}$ and $E^{ki}_{\ddot{\mathcal{A}},t}$ that violate $1$-consistency do not imply stronger restriction of a domain $D_i$, 
  \item or we need to consider a union of two CSP instances (that corresponds to the intersection of their constraints) for every constraint $E^{ij}_{\ddot{\mathcal{A}},t}$ (or $E^{ji}_{\ddot{\mathcal{A}}^,t}$) restricting domain $D_i$, namely 
$$\Upsilon_{i,u}:= E^{ij}_{\ddot{\mathcal{A}}}\wedge \Upsilon_{i,u-1}\wedge \Upsilon_{j,u-1}.
$$
\end{itemize}
Note that while our next move depends on reduced by the step constraint $E^{ik}_{\ddot{\mathcal{A}},t}$, to the instance $\Upsilon_{i,u}$ we add original constraint $E^{ik}_{\ddot{\mathcal{A}}}$. After we add enough such constraints and previous instances, the reduction of the resulted instance $\Upsilon_{i,u}$ to $D^{(\bot)}$ will give us the same projection to the coordinate $D_{i}$ as it gives (\ref{ssskiia8a6yu}). 

Also note that if two constraints at step $t$ reduce domain $D_i$ by the same values, we do not need to and we will not construct both intersections. Recall that when joining any two $\Upsilon_{i,t-1}, \Upsilon_{j,t-1}$ we have to rename all variables to retain the instances tree. Since we have at most binary relations and for any two variables there can be only two constraints containing them, namely $E^{ij}_{\ddot{\mathcal{A}}}$ and $E^{ji}_{\ddot{\mathcal{A}}}$, at the first step of the recursion process, we can add to each $\Upsilon_{i,1}$ at most $2n$ new vertices (if for every $j<n$ there are both $E_{\mathcal{X}}(i,j)$ and $E_{\mathcal{X}}(j,i)$). Then for every new constraint restricting at step $t$, we can at most double the number of variables of the largest instance of step $(t-1)$. Still, it will not make sense after the first $l$ intersections for every $\Upsilon_i$ since in this case we will get an empty domain set $D_i$ and thus justify the case $2$ of the theorem. Thus, even if we start with instances $\Upsilon_{i,1}$ on $2n$ variables, after $l$ intersections we will not need more than $2^l2n$ variables. 

First, for every $\Upsilon_{i,0}$, define $V_{\mathcal{X}_i,0}$ as a set of length $2^l2n$ that contains only one element $i$, and $E_{\mathcal{X}_i,0}$ as an empty set of length $\langle 2^l2n,2^l2n\rangle$. By $V_{\ddot{\mathcal{A}}_i,0}$ denote the set of length $\langle 2^l2n,l\rangle$, with only one non-empty domain $V_{{\ddot{\mathcal{A}}_i},0} = D_i$. By $E_{{\ddot{\mathcal{A}}_i},0}$ denote an empty set of length $\langle\langle 2^l2n,l\rangle,\langle 2^l2n,l\rangle \rangle$. 
\begin{remark}
  Strictly speaking, we are not allowed to use empty sets of some length. But we can bypass it by choosing a set, for example, for $V_{\mathcal{X}_i,0}$ with two elements, $i$, and that we will never properly use $2^l2n+1$. Further, we also consider the number function $max'(V_{\mathcal{X}_i,u})$ with the following value:
  \begin{equation}
    \begin{gathered}
       max'(V_{\mathcal{X}_i,u}) = m \iff m\neq 2^l2n+1\wedge \forall u\in V_{\mathcal{X}_i,u}, (u\neq 2^l2n+1\rightarrow u\leq m).
    \end{gathered}
  \end{equation}
\end{remark}
We construct $\Upsilon_0$,...,$\Upsilon_{n-1}$ simultaneously. The entire construction takes $0<u<2n(nl)$ steps. Each $\Upsilon_{i,u}$ consists of 
\begin{itemize}
\item set $V_{\mathcal{X}_i,u}$, representing the current number of vertices,
  \item set $E_{\mathcal{X}_i,u}$, representing the current set of edges,
  \item set $V_{{\ddot{\mathcal{A}}_i},u}$, representing domains for current variables,
  \item set $E_{{\ddot{\mathcal{A}}_i},u}$, representing constraint relations for current variables,
  \item and set $C^{erase}_{i,u}$ that keeps track of elements that must be deleted from the domain $D_i$ during each outer step $t$ and erases them one by one if we need to change the instance $\Upsilon_{i,u}$. 
\end{itemize}
We consider these sets in the above order. The description of the algorithm is as follows. For every $\Upsilon_i$, within any step $1<t<nl$ we run $2n$ internal steps. At the beginning of each new internal iteration, for some $2nt$ step, we check the list $C^{list}_{i,t+1}$ for elements that we will exclude from $D_i$ by adding new constraints to the instance during this internal iteration. We write them down to $C^{erase}_{i,2nt}$. For each step $u=2nt+j$ we consider the constraint $E^{ij}_{{\ddot{\mathcal{A}}},t+1}$ (for step $2nt+j+1$ we consider the opposite constraint $E^{ji}_{{\ddot{\mathcal{A}}},t+1}$) and decide whether it kills any of the elements from $C^{erase}_{i,2nt+j}$. 
For any $a\in C^{erase}_{i,2nt+j}$ it happens when there exists at least one edge $(a,b)\in E^{ij}_{{\ddot{\mathcal{A}}},t}$ and no edges connected with $a$ in $E^{ij}_{{\ddot{\mathcal{A}}},t+1}$:
\begin{equation}
  \exists a\in C^{erase}_{i,2nt+j} (\exists b<l,\,E^{ij}_{\ddot{\mathcal{A}},t}(a,b) \wedge \forall b<l,\,\neg E^{ij}_{\ddot{\mathcal{A}},t+1}(a,b)) .\,\,\text{[formula (\ref{dhhjsyu65d})]}
\end{equation}
If it is the case, we define our instance $\Upsilon_{i,u+1}$ as follows: we first replicate instance $\Upsilon_{i,u}$ and then add all vertices of $\Upsilon_{j,u}$ to part from $max(V_{\mathcal{X}_i,u})+1$ as well as all edges of $\Upsilon_{j,u}$, and add an edge $E_{\mathcal{X}_i,u+1}(i,max(V_{\mathcal{X}_i,u})+1+j)$.
\begin{equation}
  \begin{gathered}
     \forall k<max'(V_{\mathcal{X}_i,u})+1,\,V_{\mathcal{X}_i,u+1}(k)\iff V_{\mathcal{X}_i,u}(k),\\
     \forall k_1,k_2<max'(V_{\mathcal{X}_i,u})+1,\,E_{\mathcal{X}_i,u+1}(k_1,k_2)\iff E_{\mathcal{X}_i,u}(k_1,k_2),\\
     \forall max'(V_{\mathcal{X}_i,u})<k< 2^l2n+1,\,V_{\mathcal{X}_i,u+1}(max'(V_{\mathcal{X}_i,u})+1+k)\iff V_{\mathcal{X}_j,u}(k), \\
      \forall max'(V_{\mathcal{X}_i,u})<k_1,k_2< 2^l2n+1,\\
     \hspace {0pt}E_{\mathcal{X}_i,u+1}(max'(V_{\mathcal{X}_i,u})+1+k_1,max'(V_{\mathcal{X}_i,u})+1+k_2)\iff E_{\mathcal{X}_j,u}(k_1,k_2),\\
     E_{\mathcal{X}_i,u+1}(i,max'(V_{\mathcal{X}_i,u})+1+j).
  \end{gathered}
\end{equation}
In the same way, we define a target digraph $\ddot{\mathcal{A}}_{i,u+1}$ by adding new domains for new variables (from the list $\{D_0,...,D_{n-1}\}$) and $E^{ij}_{\ddot{\mathcal{A}}}$ as a constraint for the new edge $E_{\mathcal{X}_i,u+1}(i,$ $max(V_{\mathcal{X}_i,u})+1+j)$. 

If this is not the case, we just replicate instance $\Upsilon_{i,u}$ to $\Upsilon_{i,u+1}$. Finally, we either replicate the set $C^{erase}_{i,u}$ or change it to $C^{erase}_{i,u+1}$ as follows: 
\begin{equation}
  \begin{gathered}
     \forall a<l,\,C^{erase}_{i,(2nt+j)+1}(a)\iff C^{erase}_{i,2nt+j}(a)\wedge (\exists b<l,\, E^{ij}_{\ddot{\mathcal{A}},t+1}(a,b)).
  \end{gathered}
\end{equation}
Thus, after we pass constraint $E^{ji}_{{\ddot{\mathcal{A}}},t+1}$ we leave in $C^{erase}_{i,(2nt+j)+1}$ those elements that will be deleted in $V_{{\ddot{\mathcal{A}}},t+1}$ but because of another constraint that will lose all edges adjacent to them. We keep track of already deleted elements for the outer step $t$ not to intersect the instances with constraints that kill the same set of vertices – because we want to stop after intersection $l$ with an empty domain. 
 
At each step, sets $V_{\mathcal{X}_i,u}$, $V_{{\ddot{\mathcal{A}}_i},u}$,
$E_{\mathcal{X}_i,u}$, $E_{{\ddot{\mathcal{A}}_i},u}$ and $C^{erase}_{i,u}$ address themselves and each other in previous steps. They also address different levels of the already defined set $C^{list}_i(t,a)$ based on $V_{\ddot{\mathcal{A}},t}$ and $E^{ij}_{\ddot{\mathcal{A}},t}$. The existence of them is given by $\Sigma^{1,b}_1$-induction. At some point, we stop with tree-instances $\Upsilon_i$, each of them defining $D^{(s+1)}_i$ on $D^{(\bot)}$.
\end{proof}

The next three theorems follow from Lemma \ref{;allskdleijf5} and some previous results, formalized in $W^1_1$.

\begin{theorem}[Theorem 9.2, \cite{zhuk2020proof}]
  Suppose $D^{(0)},D^{(1)},...,D^{(s)}$ is a strategy for a cycle-consistent CSP instance $\Theta$. Then $W^1_1$ proves that:
  \begin{enumerate}
    \item if $D^{(s)}_x$ has a non-trivial binary absorbing subuniverse $B$ then there exists a $1$-consistent absorbing reduction $D^{(s+1)}$ of $\Theta^{(s)}$ with $D_x^{(s+1)}\subseteq B$;
    \item if $D^{(s)}_x$ has a non-trivial central subuniverse $C$ then there exists a $1$-consistent central reduction $D^{(s+1)}$ of $\Theta^{(s)}$ with $D_x^{(s+1)}\subseteq B$;
    \item if $D^{(s)}$ has no non-trivial binary absorbing or central subuniverse for every $y$ but there exists a non-trivial PC subuniverse $B$ in $D^{(s)}_x$ for some $x$,
    then there exists a $1$-consistent PC reduction $D^{(s+1)}$ of $\Theta^{(s)}$ with $D_x^{(s+1)}\subseteq B$.
  \end{enumerate}
\end{theorem}

\begin{theorem}[Theorem 9.3, \cite{zhuk2020proof}]
  Suppose that $D^{(0)},D^{(1)},...,D^{(s)}$ is a strategy for a $1$-consistent CSP instance $\Theta$, and $D^{(\bot)}$ is a non-linear $1$-consistent reduction of $\Theta^{(s)}$. Then $W^1_1$ proves that there exists a $1$-consistent minimal reduction $D^{(s+1)}$ of $\Theta^{(s)}$ of the same type such that $D_i^{(s+1)}\subseteq D_i^{(\bot)}$ for every variable $i$.
\end{theorem}

\begin{theorem}[Theorem 9.4, \cite{zhuk2020proof}]
  Suppose $D^{(\bot)}$ is a $1$-consistent PC reduction for a cycle-consistent irreducible
CSP instance $\Theta$, $\Theta$ is not linked and not fragmented. Then $W^1_1$ proves that there exist a reduction $D^{(1)}$ of $\Theta$ and a minimal strategy $D^{(1)},...,D^{(s)}$ for $\Theta^{(1)}$ such that the solution set to $\Theta^{(1)}$ is subdirect, the reductions $D^{(2)},...,D^{(s)}$ are non-linear, $D^{(s)}_x\subseteq D^{(\bot)}_x$ for every variable $x$. 
\end{theorem}

\subsubsection{Main theorems proved by induction}
In this section, we consider the main five theorems, proved simultaneously by induction on the size of the domain set (to be defined further). We will not consider the formalization of their proofs, since it is based on the formalization of previous results. However, some reasoning from the proofs is used for the formalization of the theorems.

\begin{remark}
We will use the same notation $D^{(s)}$ for the reductions of the initial instance, its subinstances, subconstraints, differences, unions, and both coverings and expanded coverings to avoid unnecessary indices. These, of course, cannot be the same sets of domains, but once given $D^{(s)}$ for an instance $\Theta_{\mathcal{X}}$ we can easily construct a similar reduction for any of these objects, denoted by $\Theta_{\mathcal{Y}}$, under the simple rule
$$\forall x_i\forall y_j,\,\, D_{x_i}=D_{y_j}\implies D^{(s)}_{x_i}=D^{(s)}_{y_j}.
$$
This is well-defined since we can additionally require in the reduction $D^{(s)}$ of instance $\Theta_{\mathcal{X}}$ that equal domains be reduced to equal domains (see \ref{ajdkjfhuei4}). In a minimal $1$-consistent one-of-four reduction, every $D^{(s)}_{x_i}$ must be minimal by inclusion.
\end{remark}

\begin{theorem}[Theorem 9.5, \cite{zhuk2020proof}]\label{'a;sdl6457ytr}
Suppose $D^{(1)}$ is a minimal $1$-consistent one-of-four reduction of a cycle-consistent irreducible CSP instance $\Theta$, $\Lambda(x_0,...,x_{n-1})$ is a subconstraint of $\Theta$, the solution set to $\Lambda^{(1)}$ is subdirect, $\Theta\backslash\Lambda$ has a solution in $D^{(1)}$, and $\Theta$ has no solutions in $D^{(1)}$. Then $W_1^1$ proves that there exist instances $\Upsilon_1,...,\Upsilon_t\in Coverings(\Lambda)$ such that $\Phi = (\Theta\backslash\Lambda)\cup\Upsilon_1\cup...\cup \Upsilon_{t}$ has no solutions in $D^{(1)}$, each $\Upsilon_i(x_0,...,x_{n-1})$ is a subconstraint of $\Phi$, and $\Upsilon^{(1)}_i(x_0,...,x_{n-1})$ defines a subdirect key relation with the parallelogram property for every $i$.
\end{theorem}

The formalization of the theorem will be based on its proof. Since $\Lambda(x_0,...,x_{n-1})$ is a subconstraint of $\Theta$, it follows that $\Lambda$ is a subinstance of $\Theta$ that involves variables $x_0,...,$ $x_{n-1},$ $y_0,...,y_{k-1}$, and $\Theta$ as an instance on variables $x_0,...,x_{n-1},y_0,...,y_{k-1}, z_0,...,z_{s-1}$ such that $\Theta\backslash\Lambda$ involves variables $x_0,...,x_{n-1},z_0,...,z_{s-1}$. $\Upsilon_i(x_0,...,x_{n-1})$ here denotes all tuples $(a_0,...,a_{n-1})$ such that instance $\Upsilon_i$ has a solution with $x_0=a_0$,...,$x_{n-1}=a_{n-1}$. That is, it is a projection of the solution set to $\Upsilon_i$ onto coordinates $x_0,...,x_{n-1}$, which can be expressed by the formula 
  $$\exists y^i_0...\exists y_{m_i-1}^i\Upsilon_i(x_0,...,x_{n-1},y_0^i,...,y_{m_i-1}^i).
  $$
$\Upsilon^{(1)}_i(x_0,...,x_{n-1})$ thus expressed the projection of the solution set to the instance $\Upsilon_i^{(1)}$ after the reduction $D^{(1)}$. We can denote this projection using a third-order object $\mathscr{R}^{x_0,...,x_{n-1}}_{\Upsilon_i^{(1)}}$. Note that when we talk not about a solution set to an instance but about projection to the solution set, we add to the formula an additional second-sorted existential quantifier, see (\ref{akllsoei657}).

  Since both $\Lambda$ and $\Theta\backslash\Lambda$ have solutions in $D^{(1)}$, but $\Theta$ does not, it follows that $\Lambda^{(1)}(x_0,...,$ $x_{n-1})$ and $\Theta\backslash\Lambda^{(1)}(x_0,$ $...,x_{n-1})$ define relations $\mathscr{R}^{x_0,...,x_{n-1}}_{\Lambda^{(1)}}$ and $\mathscr{R}^{x_0,...,x_{n-1}}_{\Theta\backslash\Lambda^{(1)}}$ that do not intersect. Every solution to $\Lambda^{(1)}$ is a solution to any $\Upsilon^{(1)}$ from $Coverings(\Lambda)$. According to the proof of the theorem, for every tuple $H_i$ of the relation $\mathscr{R}^{x_0,...,x_{n-1}}_{\Theta\backslash\Lambda^{(1)}}$, we find an instance $\Upsilon^{(1)}_i$ such that the relation $\mathscr{R}^{x_0,...,x_{n-1}}_{\Upsilon^{(1)}_i}$ defined by $\Upsilon^{(1)}_i(x_0,...,x_{n-1})$ is an inclusion-maximal relation that contains $\mathscr{R}^{x_0,...,x_{n-1}}_{\Lambda^{(1)}}$ and does not contain $H_i$. Thus, $\Phi = (\Theta\backslash\Lambda)\cup\Upsilon_1\cup...\cup \Upsilon_{t}$ does not have solutions in $D^{(1)}$, but if we replace any $\Upsilon_i$ by a weaker instance $\Upsilon$ that produces a greater relation $\mathscr{R}^{x_0,...,x_{n-1}}_{\Upsilon^{(1)}}$, we get an instance with solution $H_i$. That is, the number of coverings $\Upsilon_1,...,\Upsilon_t$ is bounded by the number of tuples in $\mathscr{R}^{x_0,...,x_{n-1}}_{\Theta\backslash\Lambda^{(1)}}$, which is bounded by $l^n/2 - |\mathscr{R}^{x_0,...,x_{n-1}}_{\Lambda^{(1)}}|$. Note that we do not need to know the precise number of $\Upsilon_i$ to write the formula; some of them can be repeated as many times as necessary. So, we stick to the bound $l^n$, since it can be conveniently rewritten as $(2^n)^{log_2l}$. Then, following the reasoning of Lemma \ref{()()()()(  *^ }, we can roughly bound the number of variables in each $\Upsilon_i$ by $(n+k)+nl$ (we introduce a new variable $x^{a}_i$ for all $i\in\{0,...,n-1\}$ and $a<l$). Thus, every instance $\Upsilon_i$ can be bound by a unique number $b_{\Lambda} = instsize((n+k)+nl,l)$. It follows that we can encode the set of all $\Upsilon_0,...,\Upsilon_{(2^n)^{\ log_2l}}$ by one class $\mathscr{Y}$, where each $\Upsilon_i$ is encoded by a string $X$ of length at most $nv$, with $v=\lceil log_2l\rceil$, $\mathscr{Y}(X,\Upsilon)$. Then $\tilde{row}(X,\mathscr{Y},b_{\Lambda}) = \Upsilon$, which we denote as $\Upsilon_{[X]}$.

Due to the assumption, each $\Upsilon_{i}$ is a covering for $\Lambda$ on some set of variables $x^i_0,...,x^i_{n-1},$ $ y^i_0,...,y^i_{m_i-1}$ such that for all $j<n$, $x^i_j = x_j$. Therefore, there is a homomorphism $H$ from $\mathcal{X}_{\Upsilon_i}$ to $\mathcal{X}_{\Lambda}$ that sends $x^i_j$ to $x_j$. Each $\Upsilon_i$ is a subconstraint of $\Phi$, hence it has no common variables with $\Theta\backslash\Lambda$ and any $\Upsilon_j$ except for $x_0,...,x_{n-1}$. Recall that for the union of instances we have the function $uni$ well-defined by $\Sigma^{1,b}_0$-formula, as well as the function $dif$ for the difference. In the union of two instances, we add all variables of the second instance after all variables of the first instance, shift their labels, and add equality constraints between vertices with labels that were the same. The problem here is that when we join $\Upsilon_i$ to $(\Theta\backslash\Lambda)\cup\Upsilon_1\cup...\cup \Upsilon_{i-1}$ we rename all the vertices and since the number of $\Upsilon_i$ can be exponential, we have no space to represent $\Phi$ as a second-order object. We cannot represent $\Phi$ as a third-sorted object either (with vertices labeled by strings) since the statement that there is no solution to $\Phi$ in $D^{(1)}$ would be the $\Pi^{\mathscr{B}}_1$-formula. To avoid this problem, we will not define $\Phi$, but define the preconditions that lead to the situation where $\Phi$ does not have a solution in $D^{(1)}$. By Corollary \ref{9899889((())))} these preconditions also lead to each $\mathscr{R}^{x_0,...,x_{n-1}}_{\Upsilon_i^{(1)}}$ being a key relation (so we do not need to write it down explicitly in the formula). We order all projections to $x_0,...,x_{n-1}$ of solutions to an instance $\Theta\backslash\Lambda^{(1)}$ in one class $\mathscr{H}$, where each $H$ is encoded by a string $X$ of length at most $nv$, denoted $H_{[X]}$. That $H_{[X]}$ will correspond to $\Upsilon_{[X]}$ in the sense that $\mathscr{R}_{\Upsilon^{(1)}_{[X]}}$ is an inclusion-maximal relation that does not contain $H_{[X]}$. This is reflected in square brackets in the formula. 

  The function $redinst$ is definable by the $\Sigma^{1,b}_0$-formula and returns the reduction of an instance on $D^{(1)}$. Thus, 
  \begin{gather*}
\Theta^{(1)}=redinst(\Theta,D^{(1)})\\
    \Theta\backslash\Lambda^{(1)}=redinst(\Theta\backslash\Lambda,D^{(1)})\\
 \Upsilon_i^{(1)}=redinst(\Upsilon_i, D^{(1)}).
  \end{gather*}
We also cannot use the relation $subConst_n(\Phi,\Upsilon_i, X)$ since we cannot define $\Phi$ and technically $\Upsilon_i$ is not a subinstance of $\Phi$. But we can explicitly write this condition for each $\Upsilon_i$. In the $7$th line of the formula (\ref{dfkdkdjfelrifu}) we require that the first $n$ variables in each $\Upsilon_i$ be labeled exactly by $x_0,...,x_{n-1}$ (we are talking about the existence), in the $8-9$th lines we ensure that the common variables of any $\Upsilon_i$ and $\Theta\backslash\Lambda$ are only $x_0,...,x_{n-1}$, and in the last two lines we require the same for each pair $\Upsilon_i,\Upsilon_j$. 

The relation $subD(\mathscr{R}^{x_0,...,x_{n-1}}_{\Upsilon_i^{(1)}})$ is $\Sigma^{1,b}_1$, the relation $ParlPr(\mathscr{R}^{x_0,...,x_{n-1}}_{\Upsilon_i^{(1)}})$ remains $\Pi^{1,b}_2$. Relations $min1of4Red(D^{(1)},D)$, $1C(\Theta^{(1)})$ and $subConst_n(\Theta,\Lambda, x_0,...,x_{n-1})$ are described by $\Sigma^{1,b}_0$ formulas, relations $subDSSInst(\Lambda^{(1)})$, $Cov(\Upsilon_i,\Lambda)$ and $\ddot{HOM}(\mathcal{X}_{\Theta\backslash\Lambda^{(1)}}$, $\ddot{\mathcal{A}}_{\Theta\backslash\Lambda^{(1)}})$ are $\Sigma^{1,b}_1$, relation $\neg \ddot{HOM}(\mathcal{X}_{\Theta^{(1)}},$ $\ddot{\mathcal{A}}_{\Theta^{(1)}})$ becomes $\Pi^{1,b}_1$, and $ CCInst(\Theta)$ and $IRDInst(\Theta)$ are $\Pi^{1,b}_2$-formulas (see \cite{gaysin2023proof}). This gives us the $\Sigma^{\mathscr{B}}_1$-formula.

\begin{equation}\label{dfkdkdjfelrifu}
  \begin{gathered}
     \hspace {0pt}T9.5(\Theta , D^{(1)}, \Lambda, X):=\big(CCInst(\Theta)\wedge IRDInst(\Theta) \wedge min1of4Red(D^{(1)},D)\wedge \\
     \wedge 1C(\Theta^{(1)})\wedge subConst(\Theta,\Lambda, X)\wedge subDSSInst(\Lambda^{(1)})\wedge \\
     \hspace{0pt}\wedge\ddot{HOM}(\mathcal{X}_{\Theta\backslash\Lambda^{(1)}}, \ddot{\mathcal{A}}_{\Theta\backslash\Lambda^{(1)}})\wedge \neg \ddot{HOM}(\mathcal{X}_{\Theta^{(1)}},\ddot{\mathcal{A}}_{\Theta^{(1)}})\big) \implies \exists \mathscr{H}\exists \mathscr{Y} \forall X<nv,\, \\
     \hspace {0pt}\big[Cov(\Upsilon_{[X]},\Lambda)\wedge H_{[X]}\in \mathscr{R}^{x_0,...,x_{n-1}}_{\Theta\backslash\Lambda^{(1)}}\wedge H_{[X]}\notin \mathscr{R}^{x_0,...,x_{n-1}}_{\Upsilon_{[X]}^{(1)}} \wedge \forall \Upsilon <b_{\Lambda}\\
     (Cov(\Upsilon,\Lambda)\wedge  \mathscr{R}^{x_0,...,x_{n-1}}_{\Upsilon_{[X]}^{(1)}}\subsetneq \mathscr{R}^{x_0,...,x_{n-1}}_{\Upsilon^{(1)}}))\rightarrow H_{[X]}\in \mathscr{R}^{x_0,...,x_{n-1}}_{\Upsilon^{(1)}})\big]\wedge\\
         \hspace {0pt}\wedge \forall X<nv,\, subDSSInst(\mathscr{R}^{x_0,...,x_{n-1}}_{\Upsilon_{[X]}^{(1)}}) \wedge ParlPr(\mathscr{R}^{x_0,...,x_{n-1}}_{\Upsilon_{[X]}^{(1)}})\wedge \\
     \hspace {0pt}\wedge \forall X<nv, \forall j<n,\,V_{\mathcal{X}_{\Upsilon_{[X]}}}(j,x_j) \wedge \\
     \hspace {0pt}\wedge\forall X<nv, \forall r<(s+n),\forall g<b_{n+s},\forall p< (n+k)+nl,\\
     V_{X_{\Theta\backslash\Lambda}}(r,g) \wedge V_{\mathcal{X}_{\Upsilon_{[X]}}}(p,g)\rightarrow (g=x_0\vee...\vee g=x_{n-1})\wedge \\
     \hspace {0pt}\wedge \forall X<nv, \forall X'<nv, \forall r,p< (n+k)+nl,\forall g<b_{(n+k)+nl},\\ V_{\mathcal{X}_{\Upsilon_{[X]}}}(r,g)\wedge V_{\mathcal{X}_{\Upsilon_{[X']}}}(p,g)\rightarrow (g=x_0\vee...\vee g=x_{n-1})\wedge \\
  \end{gathered}
\end{equation}

\begin{theorem}[Theorem 9.6, \cite{zhuk2020proof}]\label{===-15524}
Suppose $D^{(1)}$ is a minimal $1$-consistent one-of-four reduction of a cycle-consistent irreducible CSP instance $\Theta$, $\Theta$ is crucial in $D^{(1)}$ and is not connected. Then $W_1^1$ proves that there exists an instance $\Theta'\in ExpCov(\Theta)$ that is crucial in $D^{(1)}$ and contains a linked connected component whose solution set is not subdirect. 
\end{theorem}

To formalize this theorem, we first have to formalize some additional notions used in its proof since we need a bound on the instance $\Theta'$. For every variable $x$ of instance $\Theta$, all constraints of which are critical and rectangular, we assign the pair of sets $\xi^{\Theta,x}=(\Sigma^{\Theta,1}_{\mathcal{D}_x},\Sigma^{\Theta,2}_{\mathcal{D}_x})$ such that for all $i<2^{l^2}$ and $a,b<l$
\begin{equation}
    \begin{gathered}
 \Sigma^{\Theta,1}_{\mathcal{D}_x}(i,a,b)\iff \Sigma_{\mathcal{D}_x}(i,a,b)\wedge \big((\exists y<n, \forall a',b'<l,\\
 Con_2^{(\Theta,x)}(0,y,a',b')\leftrightarrow \Sigma_{\mathcal{D}_x}(i,a',b'))\wedge \\
 \wedge (\forall z\neq y<n, (E_{\mathcal{X}}(x,z)\rightarrow Con_2^{(E_{\mathcal{X}}(x,z),x)}\not\subseteq \Sigma_{\mathcal{D}_x,i}\wedge \\
 \wedge E_{\mathcal{X}}(z,x)\rightarrow Con_2^{(E_{\mathcal{X}}(z,x),x)}\not\subseteq \Sigma_{\mathcal{D}_x,i}))\big)\vee\\
 \hspace {0pt}\big((\exists y<n, \forall a',b'<l,\,Con_2^{(\Theta,x)}(y,0,a',b')\leftrightarrow \Sigma_{\mathcal{D}_x}(i,a',b'))\wedge \\
 \wedge (\forall z\neq y<n, (E_{\mathcal{X}}(x,z)\rightarrow Con_2^{(E_{\mathcal{X}}(x,z),x)}\not\subseteq \Sigma_{\mathcal{D}_x,i}\wedge \\
 E_{\mathcal{X}}(z,x)\rightarrow Con_2^{(E_{\mathcal{X}}(z,x),x)}\not\subseteq \Sigma_{\mathcal{D}_x,i}))\big).
    \end{gathered}
\end{equation}
and 
\begin{equation}
    \begin{gathered}
 \Sigma^{\Theta,2}_{\mathcal{D}_x}(i,a,b)\iff \Sigma^{\Theta,1}_{\mathcal{D}_x}(i,a,b)\wedge \forall j<2^{l^2}\\
 (j\neq i \wedge \Sigma^{\Theta,1}_{\mathcal{D}_x,j}\neq \emptyset\rightarrow \neg Adj(\Sigma^{\Theta,1}_{\mathcal{D}_x,i},\Sigma^{\Theta,1}_{\mathcal{D}_x,j} )).
    \end{gathered}
\end{equation}
Thus, $\Sigma^{\Theta,1}_{\mathcal{D}_x}$ is the set of all minimal congruences among the set $Con_2^{(\Theta,x)}$, and $\Sigma^{\Theta,2}_{\mathcal{D}_x}$ is the set of all minimal congruences among the congruences of $Con_2^{(\Theta,x)}$ that are not adjacent to any other congruence from $\Sigma^{\Theta,1}_{\mathcal{D}_x}$. That is, both sets contain mutually non-inclusive congruences among $Con_2^{(\Theta,x)}$ (note that a minimal congruence among $Con_2^{(\Theta,x)}$ is not necessarily a minimal congruence on $D_x$). The lists can be empty for some $i$. We call $\xi^{\Theta,x}$ a characteristic of $x$. Next, we define a partial order of such characteristics. We further consider only sets of irreducible congruences. For two sets $\Sigma$ and $\Sigma'$, define the relations $\leq$ and $<$ as follows:
\begin{equation}
\begin{gathered}
        \Sigma\leq \Sigma'\iff \forall i<2^{l^2},\, (\Sigma_i\neq\emptyset\rightarrow \exists j<2^{l^2},\, \Sigma'_j\subseteq \Sigma_i),\\
    \hspace {0pt}\Sigma< \Sigma'\iff \Sigma\leq \Sigma'\wedge \neg \Sigma'\leq \Sigma.
\end{gathered}
\end{equation}
Relations $\Sigma\backslash\Sigma'$ and $\Sigma=\Sigma'$ are defined in the usual way. Also, for any set of irreducible congruences $\Sigma'_{\mathcal{D}_x}$ define a function $\uparrow optset$ that returns the set of all congruences $\sigma$ on $D_x$ such that $\delta\subseteq\sigma$ for some $\delta\in optset(\Sigma'_{\mathcal{D}_x})$: 
\begin{equation}
\begin{gathered}
      \uparrow optset(\Sigma'_{\mathcal{D}_x})(i,a,b)\iff \Sigma_{\mathcal{D}_x}(i,a,b)\wedge \\
    \wedge (\exists j<2^{l^2},\,optset(\Sigma'_{\mathcal{D}_x,j})\neq \emptyset \wedge optset(\Sigma'_{\mathcal{D}_x,j})\subseteq \Sigma_{\mathcal{D}_x,i}).  
\end{gathered}
\end{equation}
Finally, if $(\Sigma_1,\Sigma_2)$ and $(\Sigma_1',\Sigma_2')$ are two characteristics, then define a relation $\lesssim$ as
\begin{equation}
    \begin{gathered}
        \hspace{0pt}(\Sigma_1,\Sigma_2)\lesssim (\Sigma_1',\Sigma_2')\iff (\Sigma_1<\Sigma'_1)\vee(\Sigma_1=\Sigma'_1\wedge \Sigma_2\leq\Sigma'_2) \vee\\
        \vee (\Sigma_1=\Sigma'_1\wedge \neg \Sigma_2\leq\Sigma'_2\wedge \neg \Sigma'_2\leq\Sigma_2\wedge \Sigma_2\backslash(\uparrow optset(\Sigma_1))< \Sigma'_2\backslash(\uparrow optset(\Sigma_1))).
    \end{gathered}
\end{equation}
That is, we say that $(\Sigma_1,\Sigma_2)\lesssim (\Sigma_1',\Sigma_2')$ if \begin{enumerate}
    \item either every congruence in $\Sigma_1$ contains some congruence of $\Sigma_1'$ and these sets are not equal;
    \item or $\Sigma_1$ is equal to $\Sigma_1'$ and every congruence in $\Sigma_2$ contains some congruence of $\Sigma_2'$;
    \item or $\Sigma_1$ is equal to $\Sigma_1'$, sets $\Sigma_2$ and $\Sigma_2'$ are incomparable (there exists a congruence in $\Sigma_2$ that does not contain any congruence of $\Sigma_2$ and vise versa) and every congruence in $\Sigma_2\backslash(\uparrow optset(\Sigma_1))$ contains some congruence of $\Sigma_2\backslash(\uparrow optset(\Sigma_1))$ and these sets are not equal.
\end{enumerate}
When we decrease a characteristic of a variable, we can decrease the number of congruences in either of sets $\Sigma_1,\Sigma_2,$ $\Sigma_2\backslash(\uparrow optset(\Sigma_1))$ or enlarge congruences. Since for an algebra $\mathbb{A}$ and all its subuniverses we have at most $2^{l^2}$ congruences, we can decrease a characteristic of a variable at most $2\cdot 2^{l^2}$ times, which is a constant.  

We then define three transformations of an instance,
giving an expanded covering of the original instance. These transformations do not increase the characteristics of related variables. The first transformation $T_1$ makes an instance crucial in some reduction $D^{(1)}$: it replaces constraints by all weaker constraints until the instance is crucial in $D^{(1)}$. The second transformation $T_2$ splits a variable $x$ based on two congruences on $D_x$. Finally, the third transformation $T_3$ makes some changes for a connected component of an instance. Transformations $T_1,T_2,T_3$ are not unique, but we do not need them to be unique and therefore can formalize them as relations. Thus, for two instances $\Theta$ and $\Theta'$ we say that $\Theta'$ is a $T_1$ transformation of $\Theta$ if
\begin{equation}
    \begin{gathered}
         T_1(\Theta',\Theta)\iff V_{\mathcal{X}}=V_{\mathcal{X'}}\wedge \forall i,j<n,\, E_{\mathcal{X'}}(i,j) \rightarrow E_{\mathcal{X}}(i,j) \wedge\forall i,j<n,\\
         \hspace {0pt}Weaker(E^{ij}_{\ddot{\mathcal{A}}'}, E^{ij}_{\ddot{\mathcal{A}}})\wedge CrucInst(\Theta',D^{(1)}).
    \end{gathered}
\end{equation}
For the second transformation $T_2$, we choose a variable $x$, choose two congruences $\sigma_1,\sigma_2$ on $D_x$, and define two subsets of constraints in $\Theta$ containing $x$, $\Lambda_1=\{C_1^1,C_1^2,...,C_1^k\}$ and $\Lambda_2=\{C_2^1,C_2^2,...,C_2^s\}$ such that $Con_2^{(C_1^i,x)}=\sigma_1$ and $Con_2^{(C_2^i,x)}=\sigma_2$. Denote by $\Lambda_0$ all constraints in $\Theta\backslash \Lambda_1\cup\Lambda_2$ containing $x$. Then the instance $\Theta$ is transformed as follows. We choose two new variables $x_1,x_2$ and
\begin{enumerate}
    \item rename $x$ by $x_1$ in all constraints from $\Lambda_0$ and $\Lambda_1$;
    \item rename $x$ by $x_2$ in all constraints from $\Lambda_2$;
    \item add the constraints $\sigma^*_1(x_1,x_2)$ and $\sigma^*_2(x_1,x_2)$;
    \item for every $\sigma\in Con_2^{(\Lambda_0,x)}$ add the constraint $\sigma(x_1,x_2)$.
\end{enumerate}
Both $x_1,x_2$ are \emph{children} for $x$. To formalize this transformation, we will use labels for variables. We choose new labels $x_1 = max(V_{\mathcal{X}})+1,x_2=max(V_{\mathcal{X}})+2$. To simplify the following formula, we abbreviate by $E_{\mathcal{X}}(x,y)$ both $E_{\mathcal{X}}(x,y)$ and $E_{\mathcal{X}}(y,x)$.
\begin{equation}\label{alla764745ruygt34r4t}
    \begin{gathered}
         \hspace {0pt}T_2(\Theta',\Theta, \sigma_1,\sigma_2,x)\iff irCong_m(\sigma_1,D_x)\wedge irCong_m(\sigma_2,D_x)\wedge \\
         \wedge ExpCov(\Theta',\Theta)\wedge \forall t,s<b_n,\, t\neq x\wedge s\neq x \rightarrow \\
         \rightarrow (E_{\mathcal{X}'}(t,s)\leftrightarrow E_{\mathcal{X}}(t,s)\wedge E^{ts}_{\ddot{\mathcal{A}}'}(a,b) \leftrightarrow E^{ts}_{\ddot{\mathcal{A}}}(a,b))\wedge\\
         \hspace {0pt}\wedge (\forall y< b_n,\,E_{\mathcal{X}}(x,y)\wedge Con_2^{(E_{\mathcal{X}}(x,y),x)}=\sigma_2 \rightarrow E_{\mathcal{X'}}(x_2,y))\wedge (\forall y< b_n,\,E_{\mathcal{X}}(x,y)\wedge \\
 \hspace{7pt}\wedge(Con_2^{(E_{\mathcal{X}}(x,y),x)}=\sigma_1\vee  (Con_2^{(E_{\mathcal{X}}(x,y),x)}\neq\sigma_1\wedge Con_2^{(E_{\mathcal{X}}(x,y),x)}\neq\sigma_2) )\rightarrow \\
 \rightarrow E_{\mathcal{X'}}(x_1,y)) \wedge \\
         \hspace {0pt}\wedge E_{\mathcal{X'}}(x_1,x_2)\wedge \forall a,b<l,\,E^{x_1x_2}_{\mathcal{A'}}(a,b)\leftrightarrow \sigma^*_1(a,b)\wedge \sigma^*_2(a,b)\wedge \\
         \wedge (\forall y<b_n,\,E_{\mathcal{X}}(x,y)\wedge  (Con_2^{(E_{\mathcal{X}}(x,y),x)}\neq\sigma_1\wedge Con_2^{(E_{\mathcal{X}}(x,y),x)}\neq\sigma_2)\rightarrow\\
         \rightarrow Con_2^{(E_{\mathcal{X}}(x,y),x)}(a,b)).
    \end{gathered}
\end{equation}
The second and third lines of the formula (\ref{alla764745ruygt34r4t}) reflect the fact that we do not change any constraint not containing $x$. The last three lines add to the instance new constraints from items $3,4$ (recall that we allowed one to have only one constraint relation for any two variables $x_1,x_2$ and instead of the set of constraints consider its intersection). 

Finally, the third transformation $T_3$ uses as an argument a connected component $\Lambda\subseteq\Theta$. By MinVar$(\Lambda,\Theta)=\{x_1,...,x_s\}$, where $s\geq 1$, we denote the set of all variables $x_i$ such that
there exists $\sigma\in Con_2^{(\Lambda,x_i)}$ that is minimal among $Con_2^{(\Theta,x_i)}$. Then the new instance $\Theta'$ is defined as follows. We choose new variables $x_1',...,x_s'$ and
\begin{enumerate}
    \item rename the variables $x_1,...,x_s$ by $x_1',...,x_s'$ in $\Theta\backslash\Lambda$;
    \item add the covers of all constraints from $\Lambda$ with $x_1',...,x_s'$ instead of $x_1,...,x_s$;
    \item for every $j$ and every $\sigma\in Con_2^{(\Theta\backslash\Lambda,x_j)}$ add the constraint $\sigma^*(x_j,x_j')$;
    \item for every $j$ and $\sigma\in Con_2^{(\Theta\backslash\Lambda,x_j)}$ such that $Linked(a,b,x_j,x_j,\Lambda)\nsubseteq \sigma$, add the constraint $\delta_j(x_j,x_j'),$ where $\{\delta_j\}=optset(Con^{(\Lambda,x_j)}).$
\end{enumerate}
We call each $x_i$ a \emph{parent} for $x_i'$.  We can formalize the transformation $T_3$ as a relation $T_3(\Theta',\Theta,\Lambda)$ in $V^1$ in the same way as the previous two, and we do not perform it here. The complexity of all these formulas does not exceed $\Sigma^{1,b}_2$. All transformations $T_1,T_2,T_3$ produce expanded coverings. The important thing is that transformation $T_1$ does not change the number of variables, transformation $T_2$ increases the number of variables by $1$, and transformation $T_3$ increases the number of variables by $s\leq n$.

In the proof of Theorem \ref{===-15524}, we consider a sequence of instances $\Theta_1,\Theta_2,...,\Theta_k,\Theta_{k+1}...$ such that every $\Theta_{i+1}$ is produced from $\Theta_i$ either by composition of transformations $T_1T_2$, or by composition $T_1T_3$. Due to some auxiliary lemmas in \cite{zhuk2020proof}, the compositions $T_1T_2$ and $T_1T_3$ do not increase the characteristic of any variable. The composition $T_1T_2$, splitting a variable $x$ into $x_1$ and $x_2$, decreases the number of minimal congruences between $Con_2^{(\Theta,x)}$ by one for both $x_1,x_2$. Since the number of all congruences on $A$ (and any of its subuniverse $D$) is bounded by $2^{l^2}$, the number of total new variables that we can produce from $x$ by applying $T_1T_2$ to it and all its children is bounded by $2^{2^{l^2}}$. The composition $T_1T_3$ also decreases the characteristic of $x$ by increasing all the congruences of $x'$, thus the number of descendants in one chain is also bounded by $2^{l^2}$.

Let us call a variable $x$ in instance $\Theta$ \emph{stable} if all congruences in $Con_2^{(\Theta,x)}$ are adjacent. Also, two variables $x_1,x_2$ are \emph{friends} if there is $E_{\mathcal{X}}(x_1,x_2)$ or $E_{\mathcal{X}}(x_2,x_1)$. Applying $T_1T_3$, we also decrease the characteristics of all non-stable $y$'s in MinVar$(\Lambda_i,\Theta_i)$, and we can reuse every non-stable variable at most $2^{l^2}$ times. Thus, after at most $2^{l^2}$ steps, every variable in instance $\Theta_i$ for some $i$ is stable. A stable variable $y$ cannot be a friend with both a variable $z'$ and its parent $z$. Taking into account the set of friends of $y$ in $\Theta_j$ for $j > i$, we thus see that going from $\Theta_j$ to $\Theta_{j+1}$ we can replace an old friend of $y$ with at most $2$ new weaker friends and cannot add a new friend keeping its parent. Therefore, after $\Theta_i$ with $n_i$ variables, at any step $j>i$ any variable $y$ will have at most $(n_i-1)2^{2^{l^2}}$ friends. Since any instance in the sequence $\Theta_1,\Theta_2,...,\Theta_k,\Theta_{k+1}...$ is not fragmented, from some auxiliary axioms in \cite{zhuk2020proof} it follows that there is an instance $\Theta_s$ for some $s$ that satisfies all conditions posed on instance $\Theta'$ in Theorem \ref{===-15524}.

Taking into account all the above, we can conclude that the number of instances in a sequence $\Theta_1,\Theta_2,...,\Theta_s$ cannot exceed the exponential bound, and the size of any instance $\Theta_i$ has some bound $b_{\Theta}$ polynomial in $n$ which we will not calculate precisely. The important thing is that we can formalize the sequence by a third-order class $\mathscr{Y}$, where each instance $\Theta_i$ for $1\leq i\leq s$ is encoded by a string $X$ of length at most $v$, $\mathscr{Y}(X,\Theta)$. We denote such an instance by $\Theta_{[X]}$.

In the formula (\ref{a;a;lsdkj7fr}), $redinst(\Theta', linkcomp(\Theta',D_i,a))$ is a composed function, where $linkcomp$ is expressed by the $\Sigma^{1,b}_1$-formula and returns the reduction of the domain set. The complexity of the relations $min1of4Red(D^{(1)},D)$ and $1C(\Theta^{(1)})$ is $\Sigma^{1,b}_0$. The complexity of the relations $$\neg subDSSInst(redinst(\Theta',  linkcomp(\Theta',D_i,a)))$$ and $\neg Connected(\Theta)$ is $\Pi^{1,b}_1$. $CrucInst(\Theta,D^{(1)})$ and $CrucInst(\Theta',D^{(1)})$ are expressed by formulas from $\mathfrak{B}(\Sigma^{1,b}_1)$. The complexity of $ExpCov(\Theta',\Theta)$ is $\Sigma^{1,b}_1$. Finally, the complexity of relations $ CCInst(\Theta)$ and $ IRDInst(\Theta)$ is $\Pi^{1,b}_2$. This gives us the $\Sigma^{\mathscr{B}}_1$-formula.
    \begin{equation}\label{a;a;lsdkj7fr}
    \begin{gathered}
        T9.6(\Theta , D^{(1)}):=\big[CCInst(\Theta)\wedge  IRDInst(\Theta)  \wedge min1of4Red(D^{(1)},D)\wedge\\  
        \wedge 1C(\Theta^{(1)})\wedge CrucInst(\Theta,D^{(1)})\wedge \neg Connected(\Theta)\big]\implies \exists \mathscr{Y}, \Theta_{[\emptyset]}=\Theta \wedge\\
        \forall X<v, \big[\Theta_{[S(X)]}=\Theta_{[X]} \vee (\exists x<v,\exists \sigma_1<\langle l,l\rangle ,\exists \sigma_2<\langle l,l\rangle,\exists \Theta<b_{\Theta}\\
        (T_2(\Theta,\Theta_{[X]},\sigma_1,\sigma_2,x) \wedge T_1(\Theta_{[S(X)]},\Theta))) \vee\\
        \vee (\exists \Lambda <b_{\Theta}, \exists \Theta<b_{\Theta} (T_3(\Theta,\Theta_{[X]},\Lambda) \wedge T_1(\Theta_{[S(X)]},\Theta))) \wedge \\
        \wedge ExpCov(\Theta_{[S(X)]},\Theta_{[X]})\wedge CrucInst(\Theta_{[S(X)]},D^{(1)})\big]\wedge \forall X<v,\,S(X)=v\rightarrow \\
        \rightarrow \exists a\in D_0,\neg subDSSInst(redinst(\Theta_{[X]}, linkcomp(\Theta_{[X]},D_i,a))).
    \end{gathered}
\end{equation}

\begin{theorem}[Theorem 9.7, \cite{zhuk2020proof}]\label{'aa's;s;;dldkfjf}
Suppose $D^{(1)}$ is a $1$-consistent non-linear reduction of a cycle-consistent irreducible instance $\Theta$. If $\Theta$ has a solution, then $W_1^1$ proves that it has a solution in $D^{(1)}$.
\end{theorem}
The complexity of relations $nonLNRed(D^{(1)},D)$, $1C(\Theta^{(1)})$, $\ddot{HOM}(\mathcal{X}^{(1)},\ddot{\mathcal{A}}^{(1)},H')$ and $\ddot{HOM}(\mathcal{X},\ddot{\mathcal{A}},H))$ is $\Sigma^{1,b}_0$, and the complexity of relations $ CCInst(\Theta)$ and $ IRDInst(\Theta)$ is $\Pi^{1,b}_2$. This gives us the $\Sigma^{1,b}_2$-formula.
  \begin{equation}
  \begin{gathered}
     T9.7(\Theta,D^{(1)}):= \big(CCInst(\Theta)\wedge IRDInst(\Theta)\wedge \\\wedge nonLNRed(D^{(1)},D)\wedge 1C(\Theta^{(1)})\wedge \\
     \hspace {0pt}\wedge\exists H< \langle n,\langle n,l\rangle\rangle,\ddot{HOM}(\mathcal{X},\ddot{\mathcal{A}},H)\big)\implies \exists H'< \langle n,\langle n,l\rangle\rangle,\ddot{HOM}(\mathcal{X}^{(1)},\ddot{\mathcal{A}}^{(1)},H').
  \end{gathered}
\end{equation}

\begin{theorem}[Theorem 9.8, \cite{zhuk2020proof}]\label{---098}
Suppose $D^{(0)},...,D^{(s)}$ is a minimal strategy for a cycle-consistent irreducible CSP instance $\Theta$, and a constraint $\rho(x_0,...,x_{n-1})$ of $\Theta$ is crucial in $D^{(s)}$. Then $W_1^1$ proves that $\rho$ is a critical relation with the parallelogram property.
\end{theorem}
Since we consider only binary constraints of an instance $\Theta$, the relations $Critical_2(E^{ij}_{\ddot{\mathcal{A}}})$, $ParlPr_2(E^{ij}_{\ddot{\mathcal{A}}})$, and $minStrategy(\Theta,\Theta_{Str},s)$ are $\Sigma^{1,b}_0$, the relation $CrucConst(E^{ij}_{\ddot{\mathcal{A}}},\Theta,D_Str^{(s)})$ is a Boolean combination of formulas $\Sigma^{1,b}_1$ and $\Pi^{1,b}_1$, and the relations $CCInst(\Theta)$, $ IRDInst(\Theta)$ are $\Pi^{1,b}_2$. This gives us the $\Sigma^{1,b}_2$-formula.  
\begin{equation}
  \begin{gathered}
     T9.8(\Theta, \Theta_{Str}):=\big(CCInst(\Theta)\wedge IRDInst(\Theta)\wedge minStrategy(\Theta,\Theta_{Str},s)\wedge\\
     \hspace {0pt}\exists i,j<n,\,E_{\mathcal{X}}(i,j)\wedge CrucConst(E^{ij}_{\ddot{\mathcal{A}}},\Theta,D_Str^{(s)})\big)\implies \\
     \implies Critical_2(E^{ij}_{\ddot{\mathcal{A}}})\wedge ParlPr_2(E^{ij}_{\ddot{\mathcal{A}}}).
  \end{gathered}
\end{equation}

\begin{theorem}[Theorem 9.9, \cite{zhuk2020proof}]\label{########6}
Suppose $D^{(0)},...,D^{(s)}$ is a minimal strategy for a cycle-consistent irreducible CSP instance $\Theta$, $\Upsilon(x_0,...,x_{n-1})$ is a subconstraint if $\Theta$, the solution set to $\Upsilon^{(s)}$ is subdirect, $k\in\{0,2,...,n-2\}$, $Var(\Upsilon)=\{x_0,...,x_{n-1},u_0,...,u_{t-1}\}$,
$$\Lambda = \Upsilon^{y_0,...,y_{k-1},v_0,...,v_{t-1}}_{x_0,...,x_{k-1},u_1,...,u_{t-1}}\wedge\Upsilon^{y_{k},...,y_{n-1},v_{t},...,v_{2t-1}}_{x_{k},...,x_{n-1},u_0,...,u_{t-1}}\wedge \Upsilon^{y_0,...,y_{n-1},v_{2t},...,v_{3t-1}}_{x_{0},...,x_{n-1},u_0,...,u_{t-1}}=\Upsilon_1\wedge \Upsilon_2\wedge \Upsilon_3
$$
and $\Theta^{(s)}$ has no solutions. Then $W_1^1$ proves that $(\Theta\backslash\Upsilon)\cup\Lambda$ does not have solutions in $D^{(s)}$.
\end{theorem}

  To get $\Upsilon_1, \Upsilon_2$ and $\Upsilon_3$ we use function $substitute_k$, 
$substitute_{n-k}$ and $substitute_n$ that have $\Sigma^{1,b}_0$ definition. 
After the substitution, 
\begin{gather*}
    Var(\Upsilon_1)=\{y_0,...,y_{k-1},x_k,...,x_{n-1}, v_0,...,v_{t-1}\},\\
    Var(\Upsilon_2)=\{x_0,...,x_{k-1},y_k,...,y_{n-1}, v_t,...,v_{2t-1}\},
\end{gather*}
and $$Var(\Upsilon_3)=\{y_0,...,y_{k-1},y_k,...,y_{n-1}, v_{2t},...,v_{3t-1}\}.$$
The instance $\Lambda$ here is just an intersection of all constraints of three new instances, i.e. the union $\Upsilon_1\cup \Upsilon_2\cup \Upsilon_3$. The relations $subConst_n(\Theta,\Upsilon, x_0,...,x_{n-1})$ and $minStrategy(\Theta,\Theta_{Str},s)$ are expressed by the $\Sigma^{1,b}_0$-formulas, the relation $subDSSInst(\Upsilon^{(s)})$ is the $\Sigma^{1,b}_1$-formula. The relations $\neg \ddot{HOM}(\mathcal{X}_{\Theta^{(s)}},$ $\ddot{\mathcal{A}}_{\Theta^{(s)}})$ and $\neg \ddot{HOM}(\mathcal{X}_{(\Theta\backslash\Upsilon)\cup\Lambda^{(s)}},\ddot{\mathcal{A}}_{(\Theta\backslash\Upsilon)\cup\Lambda^{(s)}})$ are $\Pi^{1,b}_1$. Finally, the relations $CCInst(\Theta)$ and $ IRDInst(\Theta)$ are $\Pi^{1,b}_2$. This gives us the $\Sigma^{1,b}_2$-formula:
\begin{equation}
  \begin{gathered}
     \hspace{0pt}T9.9(\Theta, \Theta_{Str},\Upsilon,X): = \big(CCInst(\Theta)\wedge IRDInst(\Theta)\wedge minStrategy(\Theta,\Theta_{Str},s)\wedge\\
     \wedge \forall i<n,\,V_{\mathcal{X}_{\Upsilon}}(i,x_i)\wedge \forall i<t,\,V_{\mathcal{X}_{\Upsilon}}(n+i,u_i)\wedge subConst(\Theta,\Upsilon, X)\wedge \\\wedge subDSSInst(\Upsilon^{(s)})\wedge \neg \ddot{HOM}(\mathcal{X}_{\Theta^{(s)}},\ddot{\mathcal{A}}_{\Theta^{(s)}})\big)\implies \\
     \implies \neg \ddot{HOM}(\mathcal{X}_{(\Theta\backslash\Upsilon)\cup\Lambda^{(s)}},\ddot{\mathcal{A}}_{(\Theta\backslash\Upsilon)\cup\Lambda^{(s)}}).
  \end{gathered}
\end{equation}

The proof of the above $5$ theorems goes by induction simultaneously on the size of domain sets. For this, the partial order on domain sets is introduced. For every domain set $D$, we assign a tuple of integers $Size(D)=(|D_{i_1}|,|D_{i_2}|,...,|D_{i_t}|)$, where $D_{i_1},...,D_{i_t}$ is the set of all \emph{different domains} of $D$ ordered by their size starting from the largest. That is, if for two variables $x_i,x_j$ we have $D_i=D_j$, these domains will be represented by one integer in $Size(D)$, but for different domains $D_i\neq D_j$ such that $|D_i|= |D_j|$ there will be two equal integers. Then the lexicographic order on tuples of integers induces a partial order on domain sets, i.e. we say that $(a_1,...,a_k)<(b_1,...,b_l)$ if there exists $j\in\{1,2,...,min(k+1,l)\}$ such that $a_i=b_i$ for all $i<j$, and $a_j<b_j$ or $j=k+1$. That is, $(a_1,...,a_k)<(b_1,...,b_l)$ in two cases:
\begin{itemize}
  \item these tuples are of any lengths and there is the first $j<min(k+1,l)$ such that $a_j<b_j$;
  \item $k<l$ and for all $i\leq k$, $a_i=b_i$.
\end{itemize}
It follows from the definition that $\leq$ is transitive and there does not exist an infinite descending chain of reductions. Also, duplicating domains does not affect this partial order, so the size of a domain set of any covering of the instance is not larger than the size of a domain set of the instance. If we consider any minimal (proper) one-of-four reduction $D^{(1)}$ of the instance with a domain set $D^{(0)}$, then $Size(D^{(1)}) < Size(D^{(0)})$ since we reduce equal domains simultaneously. Further, we will use the induction on the size of domain sets exclusively either for reductions of an instance or for instance and its (expanded) coverings. We never compare domain sets of totally unrelated instances.

String induction can be formalized as follows. For every CSP instance $\Theta$ with domain set $D=V_{\ddot{\mathcal{A}}}$ we define two new sets $D_{dif, D}$, $D_{ size, D}$ (here we suppose that none of the domains $D_0,...,D_{n-1}$ is empty). First, we need to remove duplicated domains: 
\begin{equation}
  \begin{gathered}
     \hspace{0pt}\forall a<l,\,D_{dif, D}(0,a)\iff D_0(a)\wedge\\
     \hspace {0pt}\wedge \forall 0<i<n,\forall a<l,\,D_{dif, D}(i,a)\iff (D_{i}(a) \wedge (\forall j<i\exists a<l,\\
     \hspace {0pt}(D_{i}(a)\wedge \neg D_{dif, D}(j,a))\vee(\neg D_{i}(a)\wedge D_{dif, D}(j,a))).
  \end{gathered}
\end{equation}
That is, if for any $i<n$ such that there is $j<i$, $D_i=D_j$, we define $D_{dif, D,i}$ to be an empty set. $D_{dif, D}$ exists due to $\Sigma^{1,b}_1$-induction up to $n$. Based on this set, we define $D_{size, D}$ using the census function:
\begin{equation}
  \begin{gathered}
     \hspace{0pt}\forall i<n,\forall s<l,\,D_{size, D}(i,s)\iff \texttt{\#}D_{dif, D,i} = s.
  \end{gathered}
\end{equation}
Then we can sort a given sequence of natural numbers $D_{size, D}$ using a number function $rank(i,n-1,D_{size, D})$, where $s=rank(i,n-1,D_{size, D})$ is the number
that appears at the $i$th position when $D_{size, D}$ is sorted in non-increasing order (see \cite{10.5555/1734064}): 
\begin{equation}
  \begin{gathered}
     \hspace{0pt}\forall i<n,\forall s<l,\,D_{\geq size, D}(i,s)\iff s=rank(i,n-1,D_{size, D}).
  \end{gathered}
\end{equation}
The formalization of the order on domain sets is straightforward (at the end of the string $D_{\geq size, D}$ there could be $0$s, but this does not affect the order). We will denote this order between strings by $\leq_{size}$. It is easy to see that if we view a string $X$ as a number $\sum_{i}X(i)2^i$, then for any two domain sets $D,D'$ such that either $\Theta'$ is a minimal reduction of the CSP instance $\Theta$ or $\Theta'\in ExpCov(\Theta)$,
$$ D_{\geq size, D'}\leq_{size}D_{\geq size, D}\implies \sum_{i}D_{\geq size, D'}(i)2^i\leq \sum_{i}D_{\geq size, D}(i)2^i.
$$
The problem can arise only if we compare the domain sets of two completely unrelated instances (for example, $D_{\geq size, D}(x)\iff x=\langle 0,l\rangle$ (all domains are $A$, $|A|=l$) and $D_{\geq size, D'}(i,l-1)$ for all $i<n$), but we never do. Thus, here we can use the order on strings (viewed as the binary representation of numbers), successor function, and string minimization axiom (see \cite{10.5555/1734064}).

It turns out that one can reduce string induction in this case to number induction. We again consider sets $D_{dif, D}$ and $D_{ size, D}$. Since we work in a fixed algebra $\mathbb{A}=(A,\Omega)$ of size $l$, there are $k_0\leq t$ domains of size $0$, $k_1\leq t$ domains of size $1$, $k_2\leq t$ domains of size $2$,...,$k_l\leq t$ domains of size $l$, with $k_0+...+k_l=t\leq n$, where $t$ is the number of different domains of the instance, 
$$
t = \texttt{\#}D_{dif, D}.$$
Then let us define $l$ sets, $K_1,K_2,...,K_l$ in the following way:
\begin{equation}
    \begin{gathered}
     K_s(0,0)\wedge \forall 0<i<n,\forall r< t,\, (K_s(i-1,r)\wedge D_{\geq size, D}(i,s)\rightarrow K_s(i,r+1)) \wedge \\
     \hspace{0pt}\wedge (K_s(i-1,r)\wedge \neg D_{\geq size, D}(i,s) \rightarrow K_s(i,r)).
  \end{gathered} 
\end{equation}
Such sets exist due to $\Sigma^{1,b}_1$-induction. Define $k_0:=0$ and $k_s=seq(n-1,K_s)$ for every $0<s\leq l$. Then the tuple
\begin{equation}
 size(D)=\langle k_l,k_{l-1},...,k_0\rangle= \langle ...\langle \langle k_l, k_{l-1}\rangle, k_{l-2} \rangle ,..., k_0 \rangle<(n(l+1)+1)^{2^{l+1}} 
\end{equation}
codes the size of the domain set $D$ by one integer. 

It is easy to see that for any $\Theta'\in ExpCov(\Theta)$, $\langle k'_l,k_{l-1}',...,k'_0\rangle\leq \langle k_l,k_{l-1},...,k_0\rangle$. Consider a minimal reduction $D'$ of $D$, suppose that we reduced one domain $D_{i_j}$ from the list $D_{i_1},...,D_{i_t}$ of the size $q$ to the size $p$. Thus, the tuple coding the size of $D'$ is 
$$ \langle k_l,...,k_{q}-1,...,k_{p}+1,...,k_0\rangle.$$
Recall the ordering property of the pairing function
\begin{equation}
  \langle x_1,x_2\rangle <\langle y_1,y_2\rangle\iff x_1+x_2<y_1+y_2\vee x_1+x_2=y_1+y_2\wedge x_2<y_2.
\end{equation}
Since we never consider the trivial case with domains of size $1$, there must be at least three integers, $\langle k_2,k_1,k_0\rangle$, and we never decrease $k_1$. 
Thus, for each reduction $D'$ we have the following situation:
$$\langle...\langle\langle...\langle a,k_{q}-1\rangle,...\rangle, k_{p}+1,...,\rangle k_0\rangle
$$
for some $a\geq 0, k_q>1$. Since 
$$\langle a,k_{q}\rangle - \langle a,k_{q}-1\rangle = a+1+k_{q},$$
for any $b>\langle 0,1\rangle$ and any $b>c\geq 0$ we have $\langle c,k_{p}+1\rangle< \langle b,k_{p}\rangle$. It follows that for any two domain sets $D,D'$ such that either $\Theta'$ is a minimal reduction of CSP instance $\Theta$ or $\Theta'\in ExpCov(\Theta)$ 
$$\langle k'_l,k_{l-1}',...,k'_0\rangle<\langle k_l,k_{l-1},...,k_0\rangle.$$
Thus, we can use the standard number induction axiom available in $W_1^1$. 

\begin{lemmach}
  For any CSP instance $\Theta$, the induction in $size(D)$ follows in $W^1_1$.
\end{lemmach}

The proof of all theorems goes simultaneously by induction on the size of the
domain sets. We assume that formulas T9.5, T9.6, and T9.7 hold in instances $\Theta$ with a domain set $D^{(0)}$ if $Size(D^{(0)})<Size(D^{(\bot)})$, and formulas T9.8 and T9.9 hold in instances $\Psi$ with a domain set $D^{(s)}$ if $Size(D^{(s)})<Size(D^{(\bot)})$. The induction step proves the formulas T9.5, T9.6, T9.7 in instances $\Theta$ with a domain set $D^{(0)}$ if $Size(D^{(0)})=Size(D^{(\bot)})$, and the formulas T9.8 and T9.9 in instances $\Psi$ with a domain set $D^{(s)}$ if $Size(D^{(s)})=Size(D^{(\bot)})$. Consider the following $\Sigma^{\mathscr{B}}_1$-formula $\phi$:
\begin{equation}
  \begin{gathered}
     \phi(t):= T.9.5(\Theta,D^{(1)}_{\Theta},\Lambda,X)\wedge T9.6(\Theta,D^{(1)}_{\Theta})\wedge T9.7(\Theta,D^{(1)}_{\Theta})  \wedge \\
     \wedge T9.8(\Psi,\Psi_{Str})\wedge T9.9(\Psi,\Psi_{Str},\Upsilon,X)\wedge \\
     \hspace {0pt}\wedge size(D^{(1)}_{\Theta})= t\wedge size(D^{(s)}_{\Psi})= t.
  \end{gathered}
\end{equation}
With the application of the Strong Induction Scheme,
\begin{equation}
  \forall x \big( (\forall t< x \phi(t))\rightarrow \phi(x)\big) \rightarrow \forall z\phi(z),
\end{equation}
we can formulate the following result.

\begin{theorem}
  Theory $W^1_1$ proves Theorems \ref{'a;sdl6457ytr}, \ref{===-15524}, \ref{'aa's;s;;dldkfjf}, \ref{---098} and \ref{########6}.
\end{theorem}

It follows immediately from Theorem \ref{'aa's;s;;dldkfjf} that $W^1_1$ proves three universal algebra axiom schemes.
\begin{theorem}
 For any fixed relational structure $\mathcal{A}$ which corresponds to an algebra with WNU operation and therefore leads to a $p$-time-solvable CSP, the theory $W_1^1$ proves universal algebra axiom schemes BA$_{\mathcal{A}}$-axioms, CR$_{\mathcal{A}}$-axioms, and PC$_{\mathcal{A}}$-axioms.
\end{theorem}

This, together with Theorem \ref{Translation[[[po]]]}, proves Theorem \ref{mainhhhdkruf}:
\begin{mythm}{1}[The main theorem]
\label{mainhhhdkruf}
For any relational structure $\mathcal{A}$ such that CSP($\mathcal{A}$) is in $P$:
\item [1.] The theory $W_1^1$ proves the soundness of Zhuk's algorithm. That is, the theory proves the formula $Reject_{\mathcal{A}}(\mathcal{X},W) $ $\implies \neg HOM(\mathcal{X},\mathcal{A})$.
\item [2.] There exists a $p$-time algorithm $F$ such that for any unsatisfiable instance $\mathcal{X}$, i.e. such that $\neg HOM(\mathcal{X},\mathcal{A})$, the output $F(\mathcal{X})$
of $F$ on $\mathcal{X}$ is a propositional proof of the proposition translation of
formula $\neg HOM(\mathcal{X},\mathcal{A})$ in propositional calculus $G$.

\end{mythm}

\section{Conclusion}\label{akkajshddf}
We have shown that Zhuk's algorithm, solving any tractable CSP($\mathcal{A}$) in polynomial time, may be augmented so that it also provides independent witnesses – propositional proofs – for negative answers. Witnesses of the non-existence of a solution, i.e. the non-existence of a homomorphism for relational structures, are proofs in the propositional proof system G. To achieve this, we proved the soundness of Zhuk's algorithm in the theory of bounded arithmetic $W^1_1$. More precisely, we based on our results in \cite{gaysin2023proof} and formalized the proofs of Theorems \ref{fhfhfhydj}, \ref{llldju67yr}, and \ref{fjfjhgd} in \cite{zhuk2020proof}. 

In Section \ref{'a'a;sldfkjfjfuryt} we first formalized in the third-sorted setting all the universal algebra notions used in the proof and then, in Section \ref{'a'tteyrghwgdfh}, showed the formalization of the proofs of theorems and lemmas themselves, concentrating on the key statements and arguments in \cite{zhuk2020proof}. We did not treat those statements whose formalization is straightforward by literally translating the original notions into bounded arithmetic language: the formalization of proofs would just exactly repeat the universal algebra reasoning of Zhuk's paper (although we have considered a few examples of those too). Also, we did not consider proofs that require nothing that thorough and tedious formalization of all tiny details in $V^1$, though they require a lot of imagination from a universal algebra point of view. Our goal was not to mechanically rewrite all the proofs in the language of the theory of bounded arithmetic. Instead, we wanted to clearly show the idea of formalization and treat notions and statements whose formalization needs some additional insight.

As far as we could, we tried to keep the formalization, even for exponentially large objects, at the second-sorted level, working with definitions of the objects rather than with the objects themselves. For this, we used some tricks and simplifications that were allowed by the fact that although some universal algebraic statements hold in general, we needed to treat only the special instances applied in \cite{zhuk2020proof}, i.e. related to the objects formed from bottom to top from the domains for variables and constraint relations. If we could stay in the second-order setup all the way, our formalization would
stay in theory $V^1$, which corresponds to the Extended Frege system. However, eventually, we still were forced to use third-order objects since even elementary (and seems to be unavoidable) from a universal algebra point of view reasoning about factor algebras when algebra is not a constant product requires exponential size.  

The most interesting open question is whether the formalization of the algorithm in a weaker theory of bounded arithmetic is possible and whether it can be done without changes in the level of Zhuk's proof of the CSP dichotomy itself. Bulatov \cite{8104069} uses different methods of universal algebra to prove the dichotomy, so the problem of formalizing Bulatov's algorithm in a theory of bounded arithmetic is an open problem of particular interest. Another question is to study some smaller tractability classes of CSP, such as CSPs with Mal'tsev polymorphisms, few subpowers, and so forth. 

The bounded formulas involved in the formalization are sometimes quite long.
It may be that a formalization that does not use formal arithmetic, but rather one of the
modern (computer-oriented) systems for formalization, namely proof assistants, such as
Lean or Isabelle, would be more suitable for this. However, the link between these
systems and propositional logic is currently missing. We think this could be another interesting avenue for future research.

\bigskip
\noindent
\textbf{Acknowledgments:}{ I would like to thank Jan Krajíček for helpful comments that resulted in many improvements to this paper. I thank Michael Kompatscher for a number of discussions on universal algebra.}

\bibliographystyle{plain}    

\bibliography{bibliography}

\end{document}